\documentclass{article}
\usepackage{graphicx} % Required for inserting images
\usepackage{fullpage}
\usepackage[utf8]{inputenc}
\usepackage{amsmath,amsfonts,amsthm,amssymb}
\usepackage{graphicx}
\usepackage{graphics}
\usepackage{url}
\usepackage{hyperref}
\usepackage{color}
\usepackage{enumitem}
\usepackage{subfiles}
\usepackage{mathrsfs}
\usepackage{caption}
\usepackage{algorithm}
\usepackage{algpseudocode}
\usepackage{varwidth}
\usepackage{mathabx}
\usepackage{svg}
\usepackage{booktabs}
\usepackage{multirow}
\usepackage{multicol}
\usepackage{subcaption}
\usepackage{here}
\usepackage[export]{adjustbox} % 画像を valign で高さ調整するため

\usepackage{tikz}
\usepackage{tkz-graph}
\usetikzlibrary{shapes.geometric}
\usetikzlibrary{positioning}
\usetikzlibrary {decorations.markings}
\usetikzlibrary{arrows.meta}
\usetikzlibrary{calc}

\newtheorem{thm}{Theorem}[section]
\newtheorem{lem}[thm]{Lemma}

\newtheorem{claim}[thm]{Claim}

\newtheorem{dfn}[thm]{Definition}

\newtheorem{fact}[thm]{Fact}

\newtheorem{conjecture}{Conjecture}

\newcommand\drop[1]{}

\def\showlabel#1{}
\def\showfiglabel#1{}
\algnewcommand{\Break}{\State \textbf{break}}
\algnewcommand{\Continue}{\State \textbf{continue}}
\algnewcommand{\True}{\textbf{true}}
\algnewcommand{\False}{\textbf{false}}

\title{Three-edge-coloring (Tait coloring) cubic graphs on the torus:\\ 
A proof of Gr\"unbaum's conjecture}
\begin{document}
\author{
Yuta Inoue\thanks{The University of Tokyo, Tokyo, Japan, \texttt{yutainoue@is.s.u-tokyo.ac.jp}.  Supported by JSPS Kakenhi 22H05001, by JST ASPIRE JPMJAP2302, and by Hirose Foundation Scholarship.}
\and
Ken-ichi Kawarabayashi\thanks{National Institute of Informatics \& The University of Tokyo, Tokyo, Japan, \texttt{k\_keniti@nii.ac.jp}.  Supported by JSPS Kakenhi 22H05001, by JSPS Kakenhi JP20A402 and by JST ASPIRE JPMJAP2302.}
\and
Atsuyuki Miyashita\thanks{The University of Tokyo, Tokyo, Japan, \texttt{miyashita-atsuyuki869@is.s.u-tokyo.ac.jp}.  Supported by JSPS Kakenhi 22H05001 and by JST ASPIRE JPMJAP2302.}
\and
Bojan Mohar\thanks{Simon Fraser University, Burnaby, BC, Canada \& FMF, University of Ljubljana, Slovenia, \texttt{mohar@sfu.ca}. Supported in part by NSERC Discovery Grant R832714 (Canada),
by the ERC Synergy grant (European Union, ERC, KARST, project number 101071836),
and by the Research Project N1-0218 of ARIS (Slovenia).}
\and
Tomohiro Sonobe\thanks{National Institute of Informatics, Japan, \texttt{tomohiro\_sonobe@nii.ac.jp}.  Supported by JSPS Kakenhi 22H05001 and by JST ASPIRE JPMJAP2302.}
 }
\date{\today}
\maketitle
%\author{ }

%\date{\today}
%\maketitle

\begin{abstract}
We prove that every cyclically 4-edge-connected cubic graph that can be embedded in the torus, with the exceptional graph class called ``Petersen-like", is 3-edge-colorable. 
This means every (non-trivial) toroidal snark can be obtained from several copies of the Petersen graph using the dot product operation. The first two snarks in this family are the Petersen graph and one of Blanu\v{s}a snarks; the rest are exposed by Vodopivec in 2008. This proves a strengthening of the well-known, long-standing conjecture of Gr\"unbaum \cite{grunbaum} from 1968.

This implies that a 2-connected cubic (multi)graph that can be embedded in the torus is not 3-edge-colorable if and only if it can be obtained from a dot product of copies of the Petersen graph by replacing its vertices with 2-edge-connected planar cubic (multi)graphs.
%or replacing two edges with a specific graph. 
Here, replacing a vertex $v$ in a cubic graph $G$ is the operation that takes a 2-connected planar cubic multigraph $H$ and one of its vertices $u$ of degree 3, unifying $G-v$ and $H-u$ and connecting the neighbors of $v$ in $G-v$ with the neighbors of $u$ in $H-u$ with a matching.

This result is a highly nontrivial generalization of the Four Color Theorem, and its proof requires a combination of extensive computer verification and computer-free extension of existing proofs on colorability. 

An important consequence of this result is a very strong version of the Tutte 4-Flow Conjecture for toroidal graphs. We show that a 2-edge connected graph embedded in the torus admits a nowhere-zero 4-flow unless it is Petersen-like (in which case it does not admit nowhere-zero 4-flows). Observe that this is a vast strengthening over the Tutte 4-Flow Conjecture on the torus, which assumes that the graph does not contain the Petersen graph as a minor because almost all toroidal graphs contain the Petersen graph minor, but almost none are Petersen-like.

As a consequence of our proofs, a quadratic-time algorithm for 3-edge-coloring cubic graphs of genus one is obtained.

Some of our proofs require extensive computer verification. The necessary source codes, together with the input and output files and the complete set of more than 14000 reducible configurations, are
available on Github\footnote{\url{https://github.com/edge-coloring}. Refer to the ``README.md'' file in each directory for instructions on how to run each program.} which can be considered as an addendum to this paper.  Moreover, we provide pseudocodes for all our computer verifications.
\end{abstract}

\section{Introduction}

\subsection{The Four Color Theorem, snarks, and our main results}

The four-color theorem (4CT) states that every loopless planar graph is 4-colorable. 
It was first shown by Appel and Haken \cite{4ct1,4ct2} in 1977, and Robertson et al.~\cite{RSST} gave a simplified version of the proof.  

The 4CT has another equivalent formulation by taking planar duality (by Tait in 1880 \cite{Tait}):

\begin{thm}\label{thm:3ECplanar}
   Every $2$-connected cubic planar graph is $3$-edge-colorable. 
\end{thm}

In this paper, we are interested in extending Theorem \ref{thm:3ECplanar} to other surfaces, but before mentioning our result, our motivation is also from the famous Tutte's \emph{$4$-Flow Conjecture}.

\begin{conjecture}[Tutte, \cite{tutte}]\label{conj:Tutte4Flow}
   Every $2$-connected graph without a Petersen graph minor admits a nowhere-zero $4$-flow\footnote{For the precise definition of nowhere-zero $4$-flow, we refer the reader to the textbook by Diestel \cite{Diestel_book}.}. 
\end{conjecture}

While this conjecture is still open, its special case restricted to cubic graphs was proved in 1997 by Robertson, Seymour, and Thomas \cite{3edgecoloring}. The proof of this significant achievement is based on two nontrivial results, one about doublecross graphs \cite{doublecross}, and the other one about apex cubic graphs. (The full proof of the latter case was not published.)

For cubic graphs, having a nowhere-zero $4$-flow is equivalent to having a 3-edge-coloring. 
Motivated by questions about colorings and flows for graphs on surfaces, an ongoing effort is being made to formulate which cubic graphs that can be embedded in a fixed surface are 3-edge-colorable. 
This question can be reduced to cyclically 4-edge-connected graphs.\footnote{A graph is \emph{cyclically $k$-edge-connected} if after deleting fewer than $k$ edges from the graph, there is at most one component that contains a cycle. See also Section \ref{sect:minimal}.}
Cyclically 4-edge-connected cubic graphs with girth of at least 5 that are not 3-edge-colorable are called \emph{snarks}. 
The Petersen graph was the first snark discovered \cite{petersen-1898}. For a while, only sporadic examples were known, see, e.g.~\cite{blanuvsa1946problem, isaacs1975infinite}, but today it is clear that snarks form a very rich family, although their complete description is totally out of reach (unless P = NP). 
Snarks are an interesting and important class of graphs because they often appear as counterexamples to various intriguing conjectures. For example, the famous Cycle Double Cover Conjecture (CDCC) claims that every 2-edge-connected graph has a collection of cycles that cover each edge twice. For cubic graphs, this is equivalent to saying that the graph can be embedded in a surface so that all faces of the embedding are simple cycles. When a cubic graph has a 3-edge-coloring, it has a cycle double cover consisting of the cycles formed by each pair of colors. Therefore, snarks are the only possible counterexamples for CDCC among cubic graphs. In fact, if CDCC is true for snarks, it is true for all graphs \cite{JAEGER19851}.

The version of Theorem \ref{thm:3ECplanar} has a chance to hold under some additional assumptions. In fact, recently, the following result was obtained in \cite{proj2024}. 

\begin{thm}\label{thm:3ECproject}
    The only snark embeddable in the projective plane is the Petersen graph. 
    In other words, a $2$-connected cubic graph $G$ that can be embedded in the projective plane is not $3$-edge-colorable if and only if $G$ can be obtained from the Petersen graph by replacing each vertex by a $2$-connected planar cubic graph. 
\end{thm}

This settles a conjecture by Mohar from 2004 and resolves a problem of Neil Robertson from the 1990s, see \cite{Mohar-snarksPP}. 

A well-known conjecture by Gr\"{u}nbaum from 1968 \cite{grunbaum} says the following.

\begin{conjecture}[Gr\"unbaum \cite{grunbaum}, 1968]\label{conj:grunbaum torus}
    If $G$ is a $2$-connected cubic graph embedded in the torus and $G$ is not $3$-edge-colorable, then $G$ contains two edges on the same face of the embedding, whose removal gives a planar graph in which the four endpoints of the removed edges lie on the boundaries of two faces. 
\end{conjecture}

This conjecture motivated a more general conjecture by Thompson, which says that if $G$ is a dual graph of a triangulation of the $(d-1)$-dimensional sphere (or even manifold) and has an even number of vertices, then $G$ is $d$-edge colorable (see Kalai \cite{kalai}). 
By Steinitz' theorem, the case $d=3$ is exactly Theorem \ref{thm:3ECplanar}. 

Although the main interest of Gr\"unbaum was about the torus, he originally asked a more general question of whether there are any snarks that are duals of triangulations of surfaces. 
Today, it is known that the general conjecture does not hold for the surface of genus 5, as proved by Kochol \cite{kochol}. 
Vodopivec \cite{Vodopivec} proved that there are infinitely many snarks that can be embedded in the torus. 
They are certain dot products of copies of the Petersen graph; each of them contains two edges whose removal yields a planar graph. 
In this paper, we outline another infinite family of toroidal snarks that are also dot products of copies of the Petersen graph. Our results then show that these two families exhaust all toroidal snarks. By classifying all their embeddings of genus 1, we obtain a proof of Gr\"unbaum's Conjecture \cite{grunbaum}.

\begin{thm}\label{mainth}\showlabel{mainth}
    The only snarks embeddable in the torus are the Petersen graph and its dot products, the Blanuša-H$k$ ($k \geq 1$) family and the Blanuša-V$k$ ($k \geq 1$).\footnote{See Section \ref{sect:minimal} for the definition of these graphs.} 
\end{thm} 

With this result, we give a complete characterization of non-3-edge-colorable cubic graphs that can be embedded in the torus, see Section \ref{sect:minimal}. 
Gr\"unbaum \cite{grunbaum} conjectured that the dual graph of a triangulation of every two-dimensional manifold is always 3-edge-colorable. This is false as mentioned above, but our result implies that the dual graph of a triangulation in the torus is always 3-edge-colorable. 
This is also interesting because the research community (see \cite{grunbaumtorus, grunbaumeven} for example) is actively working on \emph{Gr\"unbaum colorings}, by which we mean a partition of the edges of a triangulation of $G$ of some surface in three parts such that for each face $f$ of $G$, the three boundary edges of $f$ are in distinct parts of the edge-partition. Clearly, $G$ has a Gr\"unbaum coloring if and only if its dual cubic graph has a proper 3-edge-coloring. Hence, our result implies that every triangulation of the torus has a Gr\"unbaum coloring. 

Our Theorem \ref{mainth} yields another result about nowhere-zero 4-flows in more general (not necessarily cubic) graphs. It can be considered as a very strong version of the Tutte 4-Flow Conjecture. For that effect, we extend the definition of \emph{Petersen-like} graphs to allow vertices of higher degrees. We start with a cubic graph that is the dot product of copies of the Petersen graph, and then we replace some of its vertices with planar graphs $H$ that need not be cubic; we only request that the vertex $u$ of $H$ that is split to connect with 3 edges to the Petersen graph be of degree 3.

\begin{thm}\label{thm:nz4flow}
     Every cyclically-$4$-edge-connected toroidal graph different from the Petersen graph and the Blanuša-Vodopivec snark family admits a nowhere-zero $4$-flow.
\end{thm}

The proof of Theorem \ref{thm:nz4flow} is given in the appendix (section \ref{nonzero4}).   
Theorem \ref{thm:nz4flow} implies the following conjecture of
Neil Robertson from the mid-1990s; see Problem 5.5.19 in \cite{MT}.

\begin{conjecture}[Robertson, 1994]\label{neil}
    Every $2$-edge-connected graph in the torus whose representativity is at least $3$, has a nowhere-zero $4$-flow.
\end{conjecture}

Robertson made a more general conjecture resembling Gr\"unbaum's Conjecture \ref{conj:grunbaum torus}.

\begin{conjecture}[Robertson, 1994]
\label{conj:Robertson general 4flow}
    For every surface $S$, there is a constant $r$ such that every $2$-edge-connected graph embedded in $S$ with representativity at least $r$ has a nowhere-zero $4$-flow.
\end{conjecture}

By the result in \cite{proj2024}, the value $r=4$ works for the projective plane, and by Theorem \ref{thm:nz4flow}, $r=3$ works for the torus.

We obtain the following algorithmic result using Theorem \ref{mainth}.

\begin{thm}\label{algo}\showlabel{algo}
The following algorithmic question can be solved in quadratic time, $O(n^2)$, where $n=|V(G)|$.\\
{\bf Input}: A cubic graph $G$.\\
{\bf Output}: Either a 3-edge-coloring of $G$ or an obstruction showing that $G$ is not 3-edge colorable, or the conclusion that $G$ cannot be embedded in the torus (certified by obtaining a forbidden minor for embedding in the torus that is contained in $G$).
\end{thm}

Note that it is NP-complete to decide the 3-edge-colorability of cubic graphs, see \cite{Holyer}.

\subsection{Sketch of the proof: Technical difficulty over the proof for the projective planar case}

The basic outline of our proof of Theorem \ref{mainth} is similar to that for the projective planar case. A description of the proof and its main differences over that of the 4CT is given in \cite{proj2024}, and we omit it here. 

However, there are essential differences between the toroidal and projective planar case. 
We list below some of the main issues in the torus case. 

\begin{itemize}   
    \item[1.] It is harder to obtain a set of valid discharging rules and reducible configurations.
    
    The difficulty of finding reducible configurations has a great impact on the overall size of the unavoidable set of reducible configurations, as evidenced by the projective planar case. 
    (The planar case needs 633 configurations, whereas the projective planar case uses 5706 configurations.)
    Although we have no mathematical explanation why the reducibility condition is strictly harder to satisfy for the torus than for the projective plane, we found that many configurations that are reducible for the projective plane are not reducible for the torus, whereas we have no examples for the converse direction. As a result, compared to the projective plane, we have to search further for reducible configurations and discharging rules, thus obtaining an even bigger unavoidable set of 14000+ configurations.
    
    \item[2.] The Petersen graph is not the only snark embeddable in the torus.

    For the projective plane, the Petersen graph is the only embeddable snark.
    The Petersen graph in the projective plane has many nice properties (most notably, the dual graph is the complete graph $K_6$), and these structural properties were used in \cite{proj2024} to prove that some given cubic subgraph can not be part of a snark.
    However, for the torus, since there are infinitely many snarks of genus 1, there is much more flexibility in the structure of the snarks (examples are shown in Figures 
    \ref{fig:petersen-embedding}--\ref{fig:blanusaV2-embedding}).
    This makes it harder to confirm whether a given graph must not be a subgraph of a snark.
    
    \item[3.] The properties of non-contractible cycles in the torus are more diverse.

    As with the projective plane, the torus also has non-contractible cycles.
    For the projective plane, if there is a non-contractible cycle in a graph, one can obtain a planar subgraph by replacing the cycle with some vertices, and some nice properties for the plane could be applied to the remaining subgraph.
    However, we cannot do this for the torus because the remaining subgraph would be embedded in an annulus.
    The subgraphs embedded in the annulus must be dealt with differently than those in the plane. This distinction is resolved by introducing ``annular configurations").
    
    \item [4.] The Euler characteristic of the torus is equal to zero.

    The Euler characteristic plays a central role in the discharging method.
    Essentially, the sum of the charge of every vertex is proportional to the Euler characteristic and stays the same throughout the discharging process.
    For the projective plane, the Euler characteristic is positive, and the proof indicating that the resulting charge of every vertex must be nonpositive is sufficient to derive a contradiction.
    
    On the other hand, for the torus, we need to confirm that a reducible configuration appears even when every vertex has a final charge of exactly zero rather than being strictly positive. This is because it is possible that every vertex ends up with a final charge of zero, no matter how we set up the discharging rules. 
\end{itemize}

In Section \ref{sect:discharging}, we discuss the reducibility of configurations.
Roughly speaking, configurations are subgraphs in cubic graphs, and a configuration in $G$ is \emph{reducible} if it can be replaced by a smaller one (by deleting edges and suppressing the resulting vertices of degree 2) without changing the non-3-edge-colorability of $G$. 
Since the surface is more complicated, we need more reducible configurations. 
Our complete list of reducible configurations consists of approximately 14000 configurations\footnote{They are available in the aforementioned Google drive}, which is much more than that for the projective planar case (5706 configurations).

The discharging rules require a major revision as well.
Ultimately, we have constructed a list of 201 discharging rules (vs. the 169 discharging rules for the projective planar case). All the rules are described in Section \ref{sect:rule} in the Appendix. 
More details are provided in Sections \ref{sect:reducible configurations} and~\ref{sect:discharging}. 

The second issue arises when trying to reduce a configuration whose reducibility is based on having two or more edges deleted\footnote{Following previous work, we adapt ``contractions'' in the dual graph rather than ``deletions'' in the primal. See Section \ref{sect:reducible configurations}.}. 
The same issue (with the size of contraction size ``two'' replaced by ``three'') appears in the projective planar case, and we basically follow the same strategy as in \cite{proj2024}. The main difficulty over the projective planar case is that many more exceptional cases exist for the torus. This forces us to perform more involved computer testing. 
We also utilize a lemma proved in Section \ref{sect:6,7-cut}\footnote{The precise statement is the following: if $C$ is a contractible cycle of length 6 or 7 in the triangulation $G'$ (the dual graph of the cubic graph $G$) and the number of vertices strictly contained in the disk $D$ bounded by $C$ is at least 5, we can find another reducible configuration contained in $D\setminus C$.}, which is similar to the projective planar case, but since there are many more discharging rules, the proof is much more involved. 
 
For some cases that do not pass computer verification, we provide proof for each configuration using additional criteria. There are, however, some differences over the projective planar case. We may have to enumerate contractions, not just a single possible way to reduce, but many ways to reduce, and in the end, at least one contraction should work. Moreover, we observe that $G$ is of order at least 38 by the result of \cite{brinkmann2013generation} (who generated and examined all snarks of order at most 36). It turns out that only certain snarks (dot products of copies of the Petersen graph) can be embedded in the torus. This means that the dual graph $G'$ is of order at least 19. This indeed helps a lot.  
The full proof is given in Section~\ref{sect:safecont}.

For the third issue, when dealing with reducibility, we must confirm that the outskirt of the configuration is induced, meaning that there should be no edges connecting two vertices of the configuration.
In the planar case, this can be achieved fairly easily. In our proof (and also in the projective planar case), there is a short-length cycle that may be \emph{non-contractible}. The minimum length of such a cycle is called the \emph{representativity} (also called the \emph{face-width}) of the embedded graph. For the projective planar case, the proof handles the cases of representativity up to 5. The proof in \cite{proj2024} uses a combination of discharging methods and reducibility checking, which also involve some ideas that are developed in \cite{doublecross}. 

On the other hand, for the torus case, the methods for the projective planar case would not work because we have to deal with the annulus (so we have to deal with two disks, instead of just one disk for the planar/projective planar case). Therefore we have to develop completely new ideas. 
We first deal with the representativity-1 case, where we have an edge $e\in E(G)$ such that $G-e$ is planar\footnote{In this case, $G$ is an apex graph. If we would be able to assume that every apex graph is 3-edge-colorable, this case would not be relevant.}. Now, assume that the representativity of the embedding of $G$ is at least two.  In this case, $G$ may contain multiple edges, but they would not form a face of size two. So we can still apply our discharging rules to the embedding because it is a 2-cell embedding. 

Using our discharging rule, roughly, we show the following:

\begin{quote}
   In the dual graph $G'$ of a cubic graph $G$ embedded in the torus, if there is a vertex $v$ that ends up with a positive final charge, there is a subgraph $W$ that contains either $v$ or its neighbor, such that reducing $W'$ (which corresponds to $W$ in $G'$) from $G$ results in a 3-edge-colorable cubic graph. 

   Here $W$ can be embedded either in a disk or in an annulus. 
\end{quote}

For the proof of 4CT, reducible configurations are required to be induced, however, in our case above, we allow the configuration $W$ to ``wrap-around'' the torus. Indeed what we do is as follows: as mentioned above, we first obtain approximately 14000 reducible configurations that are ``induced'' (i.e., none of the configurations is allowed to be ``wrap-around''). Then, for each configuration, we enumerate ALL possible ways to ``wrap-around'' (there are approximately 50 ways for each reducible configuration) and test whether or not they are all reducible (so we test more than half a million configurations).  Fortunately, they are all reducible. Therefore, if there is a vertex of positive final charge, we are done.

We now come to the fourth issue. It may happen that there is no vertex with positive final charge. Since the total charge is 0, this would mean that all vertices have final charge equal to 0.
This means that we cannot find a reducible configuration in the conventional way (i.e. a vertex of positive final charge implies the appearance of a reducible configuration).
We have to find a reducible configuration even if all the vertices of $G'$ end up with a final charge 0 (which we refer to as the \emph{flat} case). 
% The flat case would not happen for the planar case and for the projective planar case. So 
We need a lot of extra work here.
We show that if the degree of a vertex is at least ten, a vertex of final charge zero implies the appearance of a reducible configuration. Also, we can characterize all cases where all vertices in the third neighborhood\footnote{This neighborhood is called a \emph{cartwheel} in \cite{RSST}.} of a vertex of degree at most nine have final charge zero. Using this, we show that a reducible configuration appears if vertices of degree nine whose final charge is zero are concentrated (see Lemma \ref{lem:charge0-concentrated}). We show an example in Figure \ref{fig:example-combine-graph}. In these cases, we may not have reducible configurations, but if two such vertices $v_0,v_1$ are adjacent, we can find a reducible configuration in the union of their third neighborhoods. In doing so, we need many new reducible configurations. Using these results, we find reducible configurations even in the flat case. 

\begin{figure}[htbp]
    \centering
    \includegraphics[width=13.5cm]{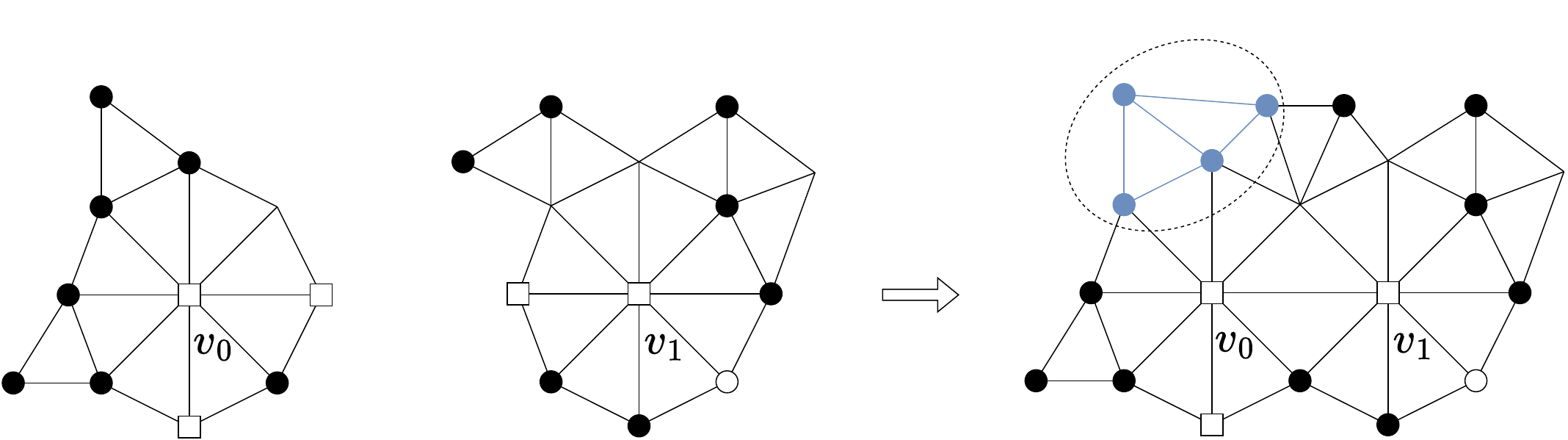}
    \caption{Two configurations around $v_0,v_1$ where the final charge is zero. If $v_0,v_1$ are adjacent, then their union contains a reducible configuration, shown as the blue-colored subgraph surrounded by a dotted circle.}
    \label{fig:example-combine-graph}
\end{figure}

\subsection{Some technical results and organization of the paper}
\label{org}\showlabel{org}

In addition to the discharging method and the reducibility testing, which are detailed in Sections \ref{sect:reducible configurations} and \ref{sect:discharging}, respectively, we utilize the following technical results.

\begin{itemize}
    \item[(a)]
    Let $C$ be a contractible cycle of length $l$ in the dual graph $G'$, where $l\in\{6, 7\}$ and the disk $D$ bounded by $C$ contains at least $l-2$ vertices. Then there is a reducible configuration that is strictly contained in the interior of $D$. 

   This is shown in Sections \ref{sect:smaller} and \ref{sect:6,7-cut}.
   \item[(b)]
 We deal with the ``flat'' case in Section \ref{sect:flatproof}. 
\end{itemize}

The conclusion in (a) is the same as that in \cite{proj2024}, but the details are different because we have more discharging rules, which makes the proof harder. The importance of (a) is that 
if the assumption of (a) is satisfied,  we can find another reducible configuration $J'$ in the disk $D$ bounded by $C$ in $G$, and we manage to show that by making a reduction using $J'$, the resulting graph obtained from $G$ is still 3-edge-colorable. More details will be given in Sections \ref{sect:smaller} and \ref{sect:6,7-cut}.  This means that if $C$ is a separating cycle of length exactly $l$ for $l=6, 7$ in $G'$, we may assume that the component of $G'-C$ strictly inside the disk bounded by $C$ has at most $l-3$ vertices.

We have around 2000 reducible configurations that need contraction of two edges or more. 
For each of these configurations, we provide a case-by-case analysis using results in Section \ref{sect:safecont} to prove that after reduction, the resulting cubic graph of $G$ still becomes 3-edge-colorable. We use computer programs to cover most of the cases.
The algorithm used for the analysis is provided in Appendix \ref{sect:code}.

This paper is organized as follows.
In Section \ref{sect:minimal} we shall introduce the main notation. 
The discharging method and the reducibility check are detailed in Sections \ref{sect:reducible configurations} and \ref{sect:discharging}, respectively. 
As pointed out above, we have to deal with the flat case. This is given in Section \ref{sect:flatproof}. 
Finally, in Sections \ref{sect:mainproof} and \ref{sect:algorithm}, we give the proof of Theorem \ref{mainth} and our quadratic-time algorithm for Theorem \ref{algo}, respectively.

Some of the main conclusions of this work are based on extensive computation. This involves reducibility of configurations and the ``unavoidability" (the discharging part). All computations have been programmed and run independently at least twice. 
Pseudocodes of programs used in our computer verification are included in Section \ref{sect:code}.

% ====================================================
\section{A minimal counterexample}
\label{sect:minimal}
\showlabel{sect:minimal}
% ====================================================

\begin{dfn}
    Let $G$ be a cubic graph. An edge-cut $F \subset E(G)$ is a \emph{cyclic cut} if at least two components of $G - F$ contain cycles. The graph is \emph{cyclically $k$-edge-connected} if no cyclic cut of size less than $k$ exists.
\end{dfn}

Let $\mathcal{T}_1$ be the set of all simple cubic bridgeless graphs embeddable in the torus. 
For $G \in \mathcal{T}_1$, the replacement of any cyclic 3-edge-cut by a vertex is called a \emph{3-edge-cut reduction}.
Similarly, the replacement of any cyclic 2-edge-cut by an edge is called a \emph{2-edge-cut reduction}.
These two reductions were introduced in the projective planar theorem \cite{proj2024}. In this paper, we introduce a third reduction called the 4-edge-cut reduction.

\begin{dfn}
    Let $G$ be a cubic graph and let $F$ be an edge-cut of size exactly four.
    Let $I_4$ be the subcubic graph obtained from the Petersen graph by removing two adjacent vertices, see Figure \ref{fig:4-edge-cut-reduction}.\footnote{This is an \emph{annular island} of ring size 4, which we will define later in Section \ref{sect:reducible configurations}.}
    If one of the connected components of $G-F$ is isomorphic to $I_4$, let us remove $I_4$ and add two edges, $\hat{e}_1$ connecting the endpoints of $e_1, e_1'$, and $\hat{e}_2$ connecting the endpoints of $e_2, e_2'$.
    (Note that the underlying surface that $I_4$ is embedded is annular, so the two newly added edges can be drawn in the same surface as $G$ without $\hat{e}_1$ and $\hat{e}_2$ crossing each other.)
    This operation is called a \emph{4-edge-cut reduction}.
\end{dfn}

\begin{figure}[htbp]
  \centering
  \includegraphics[width=6cm]{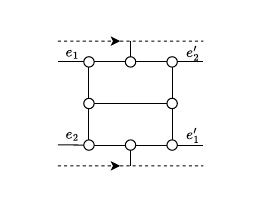}
  \caption{The graph $I_4$ as a subgraph of a toroidal graph, on which we can make a 4-edge-cut reduction. After deleting it, $e_1,e_1'$ and $e_2,e_2'$ can be connected in the skeleton of the removed $I_4$ without crossing each other.}
  \label{fig:4-edge-cut-reduction}
\end{figure}

We refer to these reductions as \emph{low edge-cut reductions}.
Since there are exactly two connected components separated by a 2-edge or 3-edge or 4-edge-cut, there are two possible graphs that are obtained by making a low edge-cut reduction. 

\begin{lem}
    A cubic graph $G \in \mathcal{T}_1$ is $3$-edge-colorable if and only if any cubic graph obtained from $G$ by zero or more low edge-cut reductions is 3-edge-colorable.
    \label{low-red}
\end{lem}

The cases for 2-edge-cut and 3-edge-cut reductions are obvious. The 4-edge-cut reduction is just the inverse of making the dot product with the Petersen graph, and it is well known that a dot product of a cubic graph $G_1$ with $P_{10}$ is 3-edge-colorable if and only if $G_1$ is 3-edge-colorable. 

Let $\mathcal{T}_0$ be a set of those graphs $G \in \mathcal{T}_1$ that can be reduced to a graph isomorphic to $P_{10}$ after a sequence of low edge-cut reductions. Graphs in $\mathcal{T}_0$ are said to be \emph{Petersen-like graphs}.
We show several examples of Petersen-like graphs. 
All of the examples are obtained by the inverse operation of a 4-edge-cut reduction, which we call \textit{4-edge-cut expansions}.
Note that a graph obtained from a toroidal graph by 4-edge-cut expansions is not necessarily toroidal, but here we only think about graphs in which the resulting graph is also toroidal.

The toroidal graph obtained from the Petersen graph by applying 4-edge-cut expansion only once is unique up to isomorphism. 
This unique graph is known as the \textit{Blanuša snark} \cite{blanuvsa1946problem}, being the second smallest snark.

Applying 4-edge-cut expansions to the Blanuša snark gives us two types of snark families: the Blanuša-H$k$ ($k \geq 1$) family and the Blanuša-V$k$ ($k \geq 1$) family. 
The Blanuša-H$k$ snark is obtained by applying 4-edge-cut expansions $k$ times horizontally, and the Blanuša-V$k$ snark is obtained by applying 4-edge-cut expansions $k$ times vertically\footnote{The embedding of Blan\v{u}sa snark 
in the first figure in Figure \ref{fig:blanusa-embedding} is the basis for vertical and horizontal alignment.}.
In this paper, we call all snarks in the Blanuša-H$k$ and Blanuša-V$k$ family as well as the Blanuša snark itself the \emph{Blanuša snark family}, and we call H or V the \emph{type} and $k$ the \emph{level} of the Blanuša snark.
The original Blanuša snark has no type and has level~$0$.
It is known that any two members of the Blanuša snark family with different type or level are nonisomorphic.
The embeddings of the Petersen graph, the Blanuša, Blanuša-H$1$, Blanuša-V$1$, and Blanuša-V2 snark, as well as their duals on the torus, are shown in Figures \ref{fig:petersen-embedding}--\ref{fig:blanusaV2-embedding}.
We show all embeddings of each of these snarks (up to its automorphisms) since we need to investigate the properties of any embedding of these snarks later in this paper.\footnote{For Blanuša-V$k$ ($k \geq 3$), Blanuša-H$l$ ($l \geq 2$), the flexibility of embeddings only depends on the difference of how four edges of (2,2)-annulus-cut are embedded. This can be easily shown by induction on $k, l$. For safety, we enumerate all the embeddings of the snarks up to $k \leq 3, k \leq 3$ by computer. See Subsection \ref{subsect:avoidpet} and Table \ref{tab:neighbor-degree-pattern}.}

\begin{figure}[htbp]
  \centering
  \includegraphics[width=0.55\textwidth]{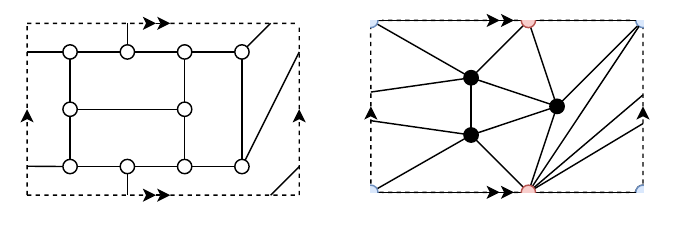}
  \caption{Embeddings of the Petersen graph and its dual in the torus.}
  \label{fig:petersen-embedding}
\end{figure}

\begin{figure}
    \centering
    \includegraphics[width=0.62\textwidth]{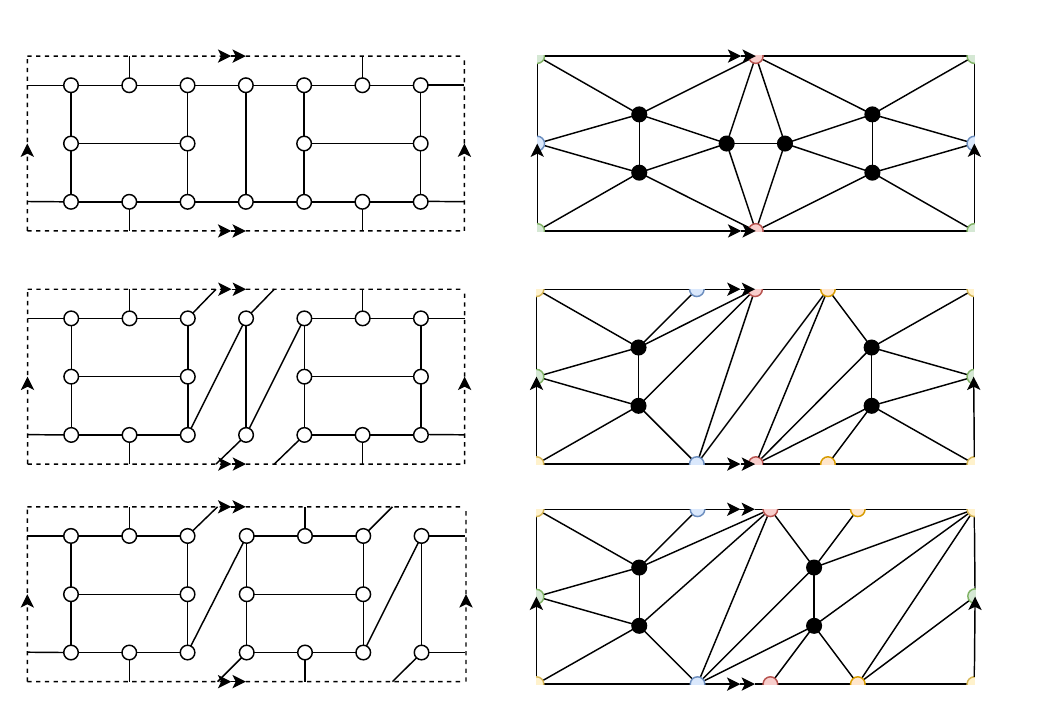}
    \caption{Embeddings of the Blanuša snark and its dual in the torus.}
    \label{fig:blanusa-embedding}
\end{figure}

\begin{figure}
    \centering
    \includegraphics[width=0.8\textwidth]{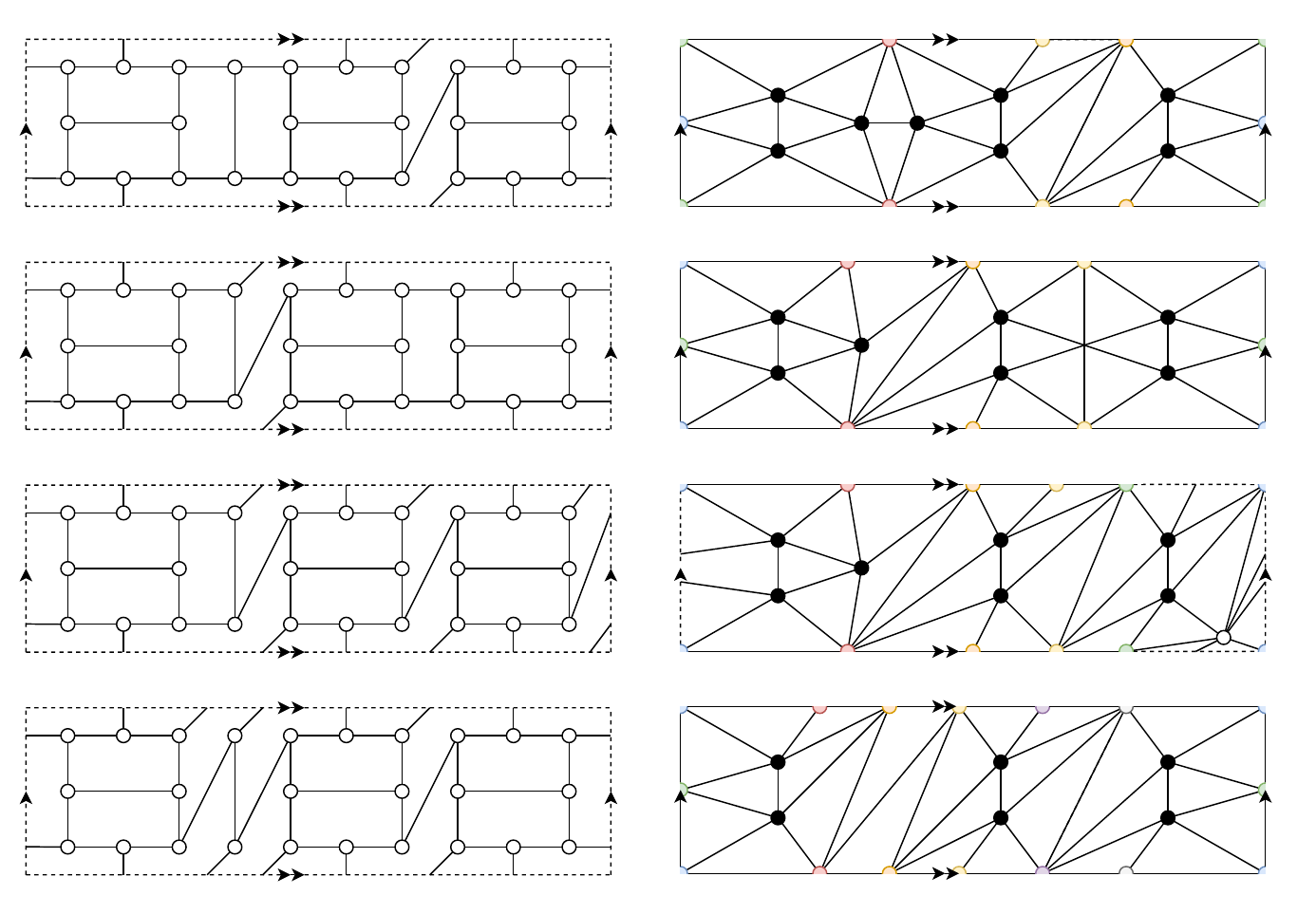}
    \caption{Embeddings of the Blanuša-H1 snark and its dual in the torus.}
    \label{fig:blanusaH1-embedding}
\end{figure}

\begin{figure}[htbp]
  \centering
  \includegraphics[width=0.5\textwidth]{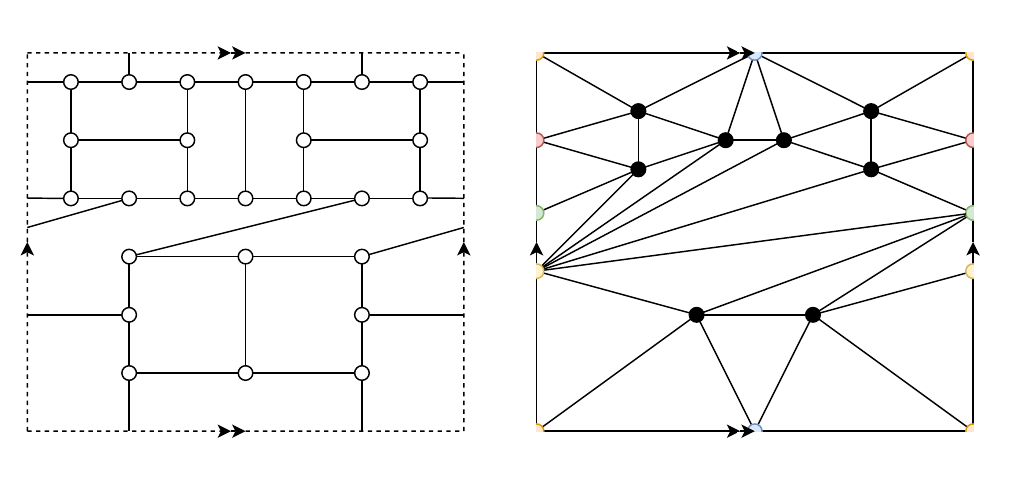}
  \caption{Embeddings of the Blanuša-V1 snark and its dual in the torus.}
  \label{fig:blanusaV1-embedding}
\end{figure}

\begin{figure}
    \centering
    \includegraphics[width=0.5\linewidth]{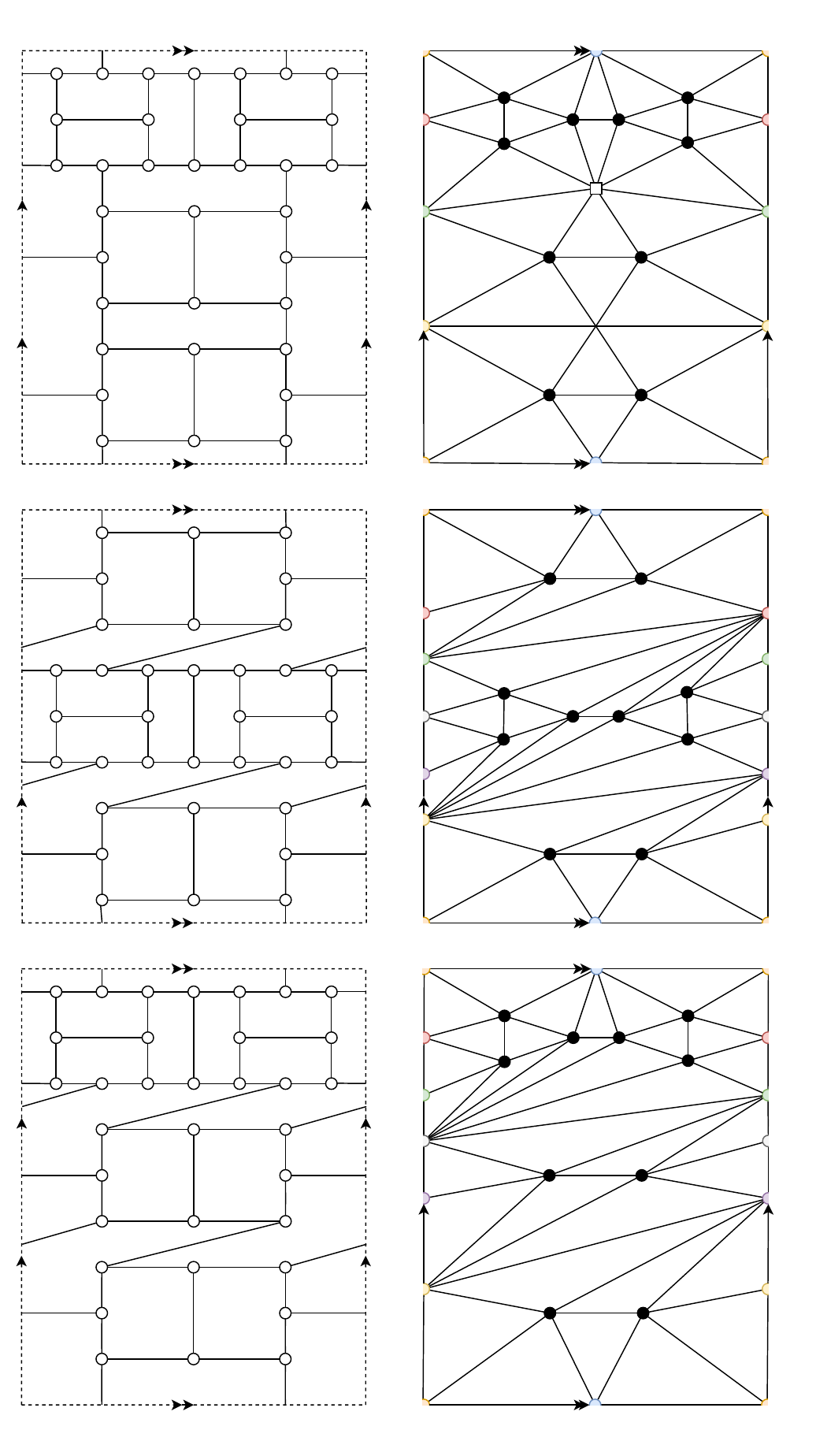}
    \caption{Embeddings of Blanuša-V2 snark and its dual in the torus.}
    \label{fig:blanusaV2-embedding}
\end{figure}

Because $P_{10}$ is not 3-edge-colorable, Lemma \ref{low-red} implies that graphs in $\mathcal{T}_0$ are not 3-edge-colorable. 
These are the obvious non-3-edge-colorable toroidal cubic graphs, and no other examples are known. 
In this paper, we prove that these are the only non-3edge-colorable cubic graphs that can be embedded in the torus. In other words, any graph in $\mathcal{T}_1 \setminus \mathcal{T}_0$ can be 3-edge-colored. While trying to prove this, we consider a hypothetical counterexample and discuss how it should look like until finally reaching a contradiction.

\begin{dfn}
    A 2-connected cubic graph $G \in \mathcal{T}_1\setminus \mathcal{T}_0$ is called a \emph{counterexample} if the graph does not admit a 3-edge-coloring. A counterexample is \emph{minimal} if there are no counterexamples of smaller order in $\mathcal{T}_1 \setminus \mathcal{T}_0$.
\end{dfn}

Our main result, Theorem \ref{mainth}, is equivalent to saying that there are no counterexamples in $\mathcal{T}_1 \setminus \mathcal{T}_0$.

The following lemmas, whose proof will be given in Appendix (section \ref{sect:below5cuts}), hold.

\begin{lem}\label{minc}\showlabel{minc}
    A minimal counterexample $G \in \mathcal{T}_1 \setminus \mathcal{T}_0$ is cyclically 5-edge-connected. Furthermore, for any cyclic cut $F$ of size $5$, one of the connected components of $G - F$ must be a 5-cycle.
\end{lem}

\begin{dfn}\label{dfn:(k,l)-annulus-cut}\showlabel{dfn:(k,l)-annulus-cut}
    A $(k,l)$-annulus-cut in an embedded graph $G$ in the torus is an edge cut $\{e_1,...,e_k,e_{k+1},...,e_{k+l}\}$ such that there exists a non-contractible simple closed curve that intersects only $e_1,...,e_k$ in $G$, and the same for $e_{k+1},...,e_{k+l}$.
\end{dfn}

We also prove the following two lemmas whose proof will be given in  Appendix (section \ref{sect:below5cuts}). 

\begin{lem}\label{lem:(2,2)-annulus-cut}\showlabel{lem:(2,2)-annulus-cut}
    A minimal counterexample $G \in \mathcal{T}_1 \setminus \mathcal{T}_0$ has no cyclic (2, $\leq$ 2)-annulus-cut.
    If a cyclic (2, 3)-annulus-cut exists in a minimal counterexample $G \in \mathcal{T}_1 \setminus \mathcal{T}_0$, one of the connected components of the remaining graph obtained by removing the annulus-cut is a 5-cycle.
\end{lem}

We also need the following result.
\begin{lem}\label{lem:(1,3)-annulus-cut}\showlabel{lem:(1,3)-annulus-cut}
    A minimal counterexample $G \in \mathcal{T}_1 \setminus \mathcal{T}_0$ has no cyclic (1, $\leq$ 3)-annulus-cut.
\end{lem}

Note that we cannot get rid of $(1, 4)$-annulus-cut, but we can actually handle the representativity one case separately, see Lemma \ref{lem:rep1}.

These lemmas imply the following:
\begin{lem}\label{lem:notwoshort}\showlabel{lem:notwoshort}
A minimal counterexample $G \in \mathcal{T}_1 \setminus \mathcal{T}_0$ has no two noncontractible curves $C_1, C_2$ hitting exactly two vertices, such that $C_1, C_2$ bound the two cylinders that both contain at least one vertex (that is, there is one vertex neither on $C_1$ nor on $C_2$ but in the cylinder)
\end{lem}

Let $G$ be a minimal counterexample. 
Throughout the paper, we will interchangeably consider the cubic graph $G \in \mathcal{T}_1 \setminus \mathcal{T}_0$ and its dual graph $G'$ in our statements and proofs. 
Since $G$ is cubic and nonplanar, $G'$ is a triangulation in the torus. The graph $G'$ is loopless since otherwise, $G$ would have a cutedge.

For other definitions and notations used in the paper, we refer to Diestel's book \cite{Diestel_book} for basic graph theory and to the monograph \cite{MT} for notions in topological graph theory.

% ====================================================
\section{Reducible configurations}
\label{sect:reducible configurations}
\showlabel{sect:reducible configurations}
% ====================================================

In this section, we start by considering the dual graph $G'$, which is a triangulation in the torus. 
Let us first recall basic definitions of configurations and reducibility used in \cite{RSST}.
A \emph{near-triangulation} is a plane graph whose faces are triangles with the exception of the infinite face. 

\begin{dfn}\label{dfn:conf}\showlabel{dfn:conf}
    A \emph{configuration} $K$ consists of a near-triangulation $G(K) = (V(K), E(K))$, and a map $\gamma_K \colon V(K) \rightarrow \mathbb{Z}_{+}$ satisfying the following conditions.
    \begin{enumerate}
        \renewcommand{\labelenumi}{(\roman{enumi})}
        \item For every vertex $v \in V(K)$, $G(K) - v$ has at most two components, and if there are two, then $\gamma_K(v) = \deg_{G(K)}(v) + 2$.
        \item For every vertex $v \in V(K)$, if $v$ is not incident with the infinite face, then $\gamma_K(v) = \deg_{G(K)}(v)$, and otherwise $\gamma_{K}(v) > \deg_{G(K)}(v)$. In either case, $\gamma_K(v) \geq 5$.
        \item $K$ has ring-size $\geq 2$ where \emph{ring-size} is $\sum_{v}(\gamma_K(v) - \deg_{G(K)}(v) - 1)$, and the summation runs over all vertices $v \in V(K)$ that are on the infinite face and for which $G(K)-v$ is connected.
    \end{enumerate}
\end{dfn}

In what follows, we will have many figures showing configurations. In all of them, we will use the Heesch notation \cite{Heesch69} designating vertex degrees. To help the reader, we collect the vertex-shapes in Figure \ref{fig:shapes}. We want to represent the degree of at least 11 in a few cases. Especially, in Figure \ref{fig:confs_degree56}, the degree of 12 is represented by describing the number near the vertex.

\begin{figure}[htbp]
  \centering
  \includegraphics[width=0.85\columnwidth]{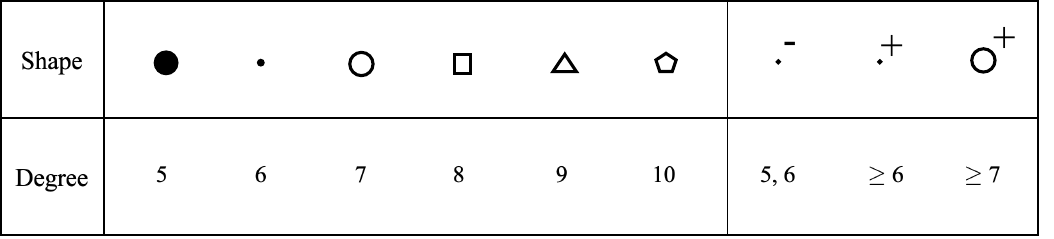}
  \caption{Shapes used to designate vertices of specific degrees.}
  \label{fig:shapes}
\end{figure}

\begin{dfn}\label{dfn:appear}\showlabel{dfn:appear}
    Let $G'$ be a triangulation in some surface $\Sigma$. A configuration $K$ \emph{appears} in $G'$ if 
    \begin{enumerate}
        \renewcommand{\labelenumi}{(\roman{enumi})}
        \item $G(K)$ is an induced subgraph of $G'$.
        \item Every finite face of $G(K)$ is also a face of $G'$.
        \item Every vertex $v \in V(K)$ satisfies $\gamma_K(v) = \deg_{G'}(v)$.
        \item There is a disk $D\subset \Sigma$ containing $G(K)$ in its interior, and the orientation of each oriented finite face of $G(K)$ is the same in $D$ (either all are clockwise or all are counterclockwise). 
    \end{enumerate}
\end{dfn}

\begin{dfn}\label{dfn:ring}\showlabel{dfn:ring}
    Let $K$ be a configuration, and $R = (V(R), E(R))$ be a cycle whose length is equal to the ring-size of $K$ and is disjoint from $G(K)$. A near-triangulation $S$ is a \emph{free completion} of K with \emph{ring} $R$ if
    \begin{enumerate}
        \renewcommand{\labelenumi}{(\roman{enumi})}
        \item $R$ is a cycle in $S$ that bounds the infinite face of $S$.
        \item $G(K)$ is a subgraph of $S$, induced on $V(S)\setminus V(R)$, i.e. $G(K) = S - V(R)$, every finite region of $G(K)$ is a finite region of $S$, and the infinite region of $G(K)$ includes $R$ and the infinite region of $S$.
        \item Every vertex $v$ of $S$ not in $V(R)$ has degree $\gamma_K(v)$ in S.
    \end{enumerate}
%    $R$ is called a \emph{ring} of $K$.
\end{dfn}

\begin{dfn}\label{dfn:island}\showlabel{dfn:island}
  An \emph{island} $I$ is a plane graph with the following properties.
  \begin{enumerate}
    \renewcommand{\labelenumi}{(\roman{enumi})}
    \item $I$ is 2-connected.
    \item Every vertex of $I$ has degree two or three.
    \item Every vertex of degree two is incident with the infinite region.
  \end{enumerate}
  An island \emph{appears} in a cubic graph $G \in \mathcal{T}_1$ if $I$ is a subgraph of $G$, and every finite face of $I$ is also a face of $G$.
\end{dfn}

For any configuration $K$, the inner dual\footnote{The \emph{inner dual} is obtained from the dual graph by deleting the vertex that corresponds to the infinite face.} of the free completion of $K$ is an island. This island is called the \emph{island of $K$} and will be denoted by $I(K)$. 

% match / signed match
\begin{dfn}\label{dfn:match}\showlabel{dfn:match}
  Let $k$ be a positive integer. 
  A \emph{match} is an unordered pair $\{a, b\}$ of distinct integers in $\{1,\dots,k\}$.
  Two matches $\{a,b\}$ and $\{c,d\}$ with $1\le a<c<b<d\le k$ are said to \emph{overlap}.
  A \emph{signed match} is a pair $(m, \mu)$ where $m$ is a match and $\mu = \pm 1$.
\end{dfn}

% matchings / signed matchings
\begin{dfn}\label{dfn:matching}\showlabel{dfn:matching}
    A \emph{matching} is a set $M$ of matches such that any two distinct matches $\{a_1, b_1\}, \{a_2, b_2\} \in M$ is disjoint, i.e. $\{a_1, b_1\} \cap \{a_2, b_2\} = \varnothing$.
    Let $M$ be a matching and $\chi_M$ be a mapping $M \to \{-1, +1\}$.
    The set $M' = \{(m, \chi(m)) \mid m \in M\}$ is called the \emph{signed matching} corresponding to $M$, and $\chi_M$ is called the \emph{sign function}.
    We denote $E(M)$ as the union of all matches in $M$ where $M$ is a matching, and we denote $E(M')$ as the union of all matches in $M$ where $M'$ is a signed matching corresponding to $M$.
\end{dfn}

% planar
\begin{dfn}\label{dfn:planar matching}\showlabel{dfn:planar matching}
  A \emph{planar matching} is a matching $M$ such that no two different matches in $M$ overlap.
  A \emph{planar signed matching} is a signed matching corresponding to some planar matching $M$.
\end{dfn}

% toroidal
\begin{dfn}\label{dfn:toroidal-matching}\showlabel{dfn:toroidal-matching}
    A \emph{toroidal matching} is a matching $M$ with one of the following properties.
    
  \begin{enumerate}
      \renewcommand{\labelenumi}{(\roman{enumi})}
      \item $M$ is a planar matching.
      \item There exist two overlapping matches $\{a_1, b_1\}, \{a_2, b_2\} \in M$ ($a_1 < a_2 < b_1 < b_2$) such that, the following matching $M'$ is a planar matching.
      \begin{align*}
        M' &= \{\{f(a), f(b)\} \mid \{a, b\} \in M \setminus \{\{a_1,b_1\},\{a_2,b_2\}\}\} \\
        \textrm{where } f(x) &= \begin{cases}
            x - a_2 + b_2 - a_1 - 3 + 2(|M| - 2) & (x < a_1)\\
            x - a_2 + b_2 - b_1 - 1 & (a_2 < x < b_1)\\
            x - b_1 & (b_1 < x < b_2)\\
            x - a_1 + b_2 - a_2 - 2 & (a_1 < x < a_2)\\
            x - a_2 + b_2 - a_1 - 3 & (b_2 < x)\\
        \end{cases}
      \end{align*}
  \end{enumerate}

  A \emph{toroidal signed matching} is a signed matching corresponding to some toroidal matching $M$.
\end{dfn}

The mapping $f$ is a way to convert matchings embedded in the torus to the plane. (See Figure \ref{fig:toroidal-kempe}.)

\begin{figure}[htbp]
    \centering
    \includegraphics[width=0.5\linewidth]{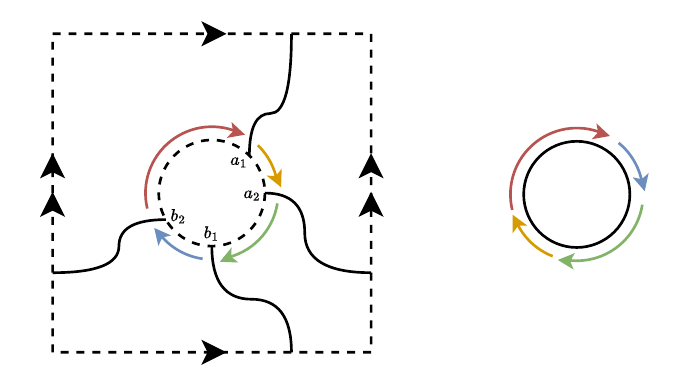}
    \caption{The mapping $f$ converts the left toroidal matching to the right planar matching.}
    \label{fig:toroidal-kempe}
\end{figure}

Let a coloring $\kappa$ be a mapping from $\{1,\dots,k\}$ to $\{0,1,2\}$.
The coloring $\kappa$ satisfies \emph{the parity condition} if $|\kappa^{-1}(0)| \equiv |\kappa^{-1}(1)| \equiv |\kappa^{-1}(2)| \pmod 2$. 
We will only consider colorings that satisfy the parity condition since every edge-coloring of edges in a cut satisfies the parity condition.

\begin{dfn}\label{dfn:fit}\showlabel{dfn:fit}
  Let $\kappa: \{1,\dots,k\} \rightarrow \{0, 1, 2\}$ be a coloring that satisfies parity. For $\theta \in \{0, 1, 2\}$, $\kappa$ is said to \emph{$\theta$-fit} a signed matching $M$ if
  \begin{enumerate}
        \renewcommand{\labelenumi}{(\roman{enumi})}
        \item $E(M) = \{e \mid 1 \leq e \leq k, \kappa(e) \neq \theta\}$
        \item For any $(\{e, f\}, \mu) \in M$, $\kappa(e) = \kappa(f)$ if and only if $\mu = 1$.
  \end{enumerate}
\end{dfn}

% consistent set
\begin{dfn}\label{dfn:consist}\showlabel{dfn:consistent}
    Let $\mathcal{M}$ be some set of matchings.
    A set $\mathscr{C}$ of colorings $\{1,\dots,k\} \rightarrow \{0,1,2\}$ that satisfy parity is \emph{consistent over} $\mathcal{M}$ if for every $\kappa \in \mathscr{C}$ and every $\theta \in \{0,1,2\}$ there is a signed matching $M'$ corresponding to a matching $M \in \mathcal{M}$ and a sign function $\chi_M$ such that $\kappa$ $\theta$-fits $M'$, and $\mathscr{C}$ contains every coloring that $\theta$-fits $M'$.
    Especially, when $\mathcal{M}$ is the set of all planar matchings, $\mathscr{C}$ is said to be consistent on the plane, and when $\mathcal{M}$ is the set of all toroidal matchings, $\mathscr{C}$ is said to be consistent in the torus.
\end{dfn}

% D-reducible
\begin{dfn}\label{dfn:Dred}\showlabel{dfn:Dred}
  Let $I$ be an island, and let $\mathcal{M}$ be a set of matchings.
  We add a new incident edge to each vertex of degree 2 in $I$ and denote the set of all added edges by $R$. 
  Let $\mathscr{C}^{\ast}$ be the set of all colorings $\{1,\dots,|R|\} \rightarrow \{0,1,2\}$ that satisfies parity.
  Let $\mathscr{C_0}$ be the set of all $3$-edge-colorings of $R$ that are restrictions of $3$-edge-colorings of $I \cup R$. 
  Let $\mathscr{C}'$ be the maximal consistent subset of $\mathscr{C}^{\ast} - \mathscr{C_0}$ over $\mathcal{M}$. 
  The island $I$ is said to be \emph{D-reducible} over $\mathcal{M}$ if $\mathscr{C}' = \emptyset$. 
  We also call the configuration $K$ \emph{D-reducible} over $\mathcal{M}$ if the corresponding island $I(K)$ is D-reducible.
\end{dfn}

\begin{dfn}
    Let $G$ be a subcubic graph and $F$ be a set of edges in $G$ that satisfy the following conditions.
    \begin{itemize}
        \item The endpoints of $f \in F$ have degree 3.
        \item For all $v \in V(G)$, the number of edges adjacent to $v$ in $F$ is not $2$.
    \end{itemize}
    Let $G \dotdiv F$ be the subcubic graph obtained from $G$ by deleting the edges in $F$ and suppressing its endpoints of degree 2.
\end{dfn}

Note that if $G$ is a cubic graph, $G\dotdiv F$ is also a cubic graph.

\begin{dfn}\label{dfn:Cred}\showlabel{dfn:Cred}
  Let $I$ be an island and let $X \subseteq E(I)$ be such that none of the vertices in $I$ is incident with exactly two edges of $X$. 
  Let $I' = I \dotdiv X$ .
  Let $\mathscr{D}$ be all restrictions to $R$ of $3$-edge-colorings of $I' \cup R$.
  $I$ is \emph{C-reducible} if
    $\mathscr{D} \cap \mathscr{C}' = \emptyset$.

  The edges in $X$ are called \emph{contraction edges}\footnote{Although the edges in $X$ are not contracted, but rather deleted in $I$, we use the term ``contraction" following previous work. 
  This is because the term originated from definitions using the near-triangulation $K$ rather than the cubic dual $I$.} of $I$, and $X$ is denoted by $c(I)$. 
  The size of $c(I)$ is called the \emph{contraction size} of $I$.
  
  A configuration $K$ is also \emph{C-reducible} if the corresponding island $I(K)$ is C-reducible.
  In this case, we call the edges of $K$ corresponding to the edges in $c(I)$ the \emph{contraction edges} and denote them by $c(K)$.
\end{dfn}

Let $I$ be a reducible island in a minimal counterexample $G$.
By the definition of reducibility, if we can 3-edge color $G \dotdiv c(I)$, we can also 3-edge color $G$.
This can be constructively proved by a method known as \emph{Kempe changes} \cite{AppelHaken89}\cite{RSST}. Hence, we omit the proof here since it is identical to previous works for planar graphs (note that reducibility implies that Kempe chains are enough to extend a 3-edge coloring of $G$ from that of $G \dotdiv c(I)$).

This means that contracting the edges in $c(I)$ will not affect the non-3-edge colorability of a graph $G$ in which $I$ is a subgraph.
This implies the following lemma.

\begin{lem}\label{lem:reducing}\showlabel{lem:reducing}
    Let $G$ be a minimal counterexample embedded in the torus $\Sigma$. 
    Suppose that $C$ is a circuit of $G$ bounding a closed disk $\Delta$ such that the subgraph of $G$ formed by the vertices and edges drawn in $\Delta$ is a reducible island $I$. 
    Then $G \dotdiv c(I)$ either has a bridge or belongs to $\mathcal{T}_0$.
\end{lem}

Now, we introduce a set of reducible configurations $\mathcal{K}$, which we have constructed by hand and verified each of its reducibility by computer programs.
There are more than 14000 reducible configurations in $\mathcal{K}$, and the reader can find all the reducible configurations (including contraction edges for $C$-reducible configurations) in the aforementioned Google Drive. 
The set $\mathcal{K}$ \footnote{Note that $\mathcal{K}$ consists of normal configurations and does not contain any annular configuration or any wrapping of a normal configuration ("normal configurations and wrapping" are defined later).} also happens to be an unavoidable set of reducible configurations in the torus.
We prove unavoidability by the discharging method later in this paper; see lemma \ref{T(v)>0}.
In Table \ref{tab1}, we show the contraction size of each configuration of $\mathcal{K}$. The contraction size being zero means D-reducible.
Let us remark the following.

\begin{itemize}
    \item There are 170 configurations that need to contract at least five edges. 
    Indeed, there are $90, 59, 16, 3, 1, 1$ configurations where the contraction edge size is 5, 6, 7, 8, 9, and 10, respectively. See Table \ref{tab1}. 
    \item There are around 1000 configurations that contain pairs of vertices at a distance of exactly five (although none of our configurations contain pairs of vertices at a distance of six or more). 
\end{itemize}
\begin{table}[htbp]
    \centering
    \caption{The number of configurations that need the contraction size $k$ $(0 \leq k \leq 10)$ in $\mathcal{K}$.}\label{tab1}
    \begin{tabular}{cccccccccccccc}
        \toprule
         contraction size & 0 & 1 & 2 & 3 & 4 & 5 & 6 & 7 & 8 & 9 & 10 & 11$_{\geq}$ & total \\
        \midrule
         \# of configurations & 4777 & 7441 & 530 & 975 & 384 & 90 & 59 & 16 & 3 & 1 & 1 & 0 & 14276 \\
        \bottomrule
    \end{tabular}
\end{table}

\subsection{General configurations}\label{non-induced}

The definitions of \emph{configurations} and \emph{islands} stated above are derived from previous work.
However, in later sections, we need to use configurations that are not planar and even configurations with two or more boundaries. % Probably Chapter 7 and 8?
To account for this, we introduce a more general type of configurations which we define below.

\begin{dfn}
    A \emph{general configuration} $K$ consists of a graph $G(K) = (V(K), E(K))$ that is 2-cell embedded in some closed surface $\mathbb{S}$, a set of faces\footnote{Face boundaries here need not be cycles; just closed facial walks.} of $G(K)$ which we denote by $\mathcal{Q}(K)$, and a map $\gamma_K: V(K) \to \mathbb{Z}_{+}$ satisfying the following conditions:
    \begin{enumerate}
        \item For every vertex $v\in V(K)$, there are at most two faces in $\mathcal{Q}(K)$ that are incident with $v$. If $v$ is incident with $Q\in \mathcal{Q}(K)$, we denote by $w_Q(v)$ the number of times the boundary of $Q$ (considered as a closed facial walk) touches $v$. (For convenience, if $v$ is not incident with $Q$, we set $w_Q(v)=0$.)
        \item For every vertex $v\in V(K)$, if $v$ is not incident with a face in $\mathcal{Q}(K)$, then $\gamma_K(v) = \text{deg}_{G(K)}(v)$, and otherwise $\gamma_K(v) \geq \deg_{G(K)}(v) + \sum_{Q\in \mathcal{Q}(K)} w_Q(v)$.
        In either case, $\gamma_K(v) \geq 5$.
        \item For every face $Q \in \mathcal{Q}(K)$, there is a value $r(Q)$ called the \emph{ring size}, defined by the following formula: 
%        $\sum_{v} (\gamma_K(v) - \deg_{G(K)}(v) - (2 - w_Q(v)))$, 
        $$ r(Q) = \sum_{v} (\gamma_K(v) - \deg_{G(K)}(v) - w_Q(v)),$$ 
        where the summation runs over all vertices $v \in V(K)$ that are incident with $Q$.
        \item All faces that are of size greater than 3 must be in $\mathcal{Q}(K)$.
    \end{enumerate}
\end{dfn}

We can see that a (normal) configuration is the special case of general configurations, where $\mathcal{Q}(K)$ consists of only one face (the outer face boundary). 

We also define \emph{free completions} and \emph{rings} for general configurations.

\begin{dfn}
    Let $K$ be a general configuration and let $\mathcal{R}(K)$ be a set of cycles, in which for each cycle $R \in \mathcal{R}(K)$ there exists a corresponding face $Q_R \in \mathcal{Q}(K)$ such that $|V(R)| = r(Q_R)$.
    A graph $S$ embedded in the torus is a \emph{free completion} of $K$ with \emph{ring set} $\mathcal{R}(K)$ if the following conditions are satisfied.

    \begin{enumerate}
        \item $R$ is a subgraph of $S - G(K)$, and for every vertex of $r$ of $R$, there is a vertex adjacent to $r$, which is in the boundary of the corresponding face $Q_R$.
        \item $G(K)$ is a subgraph of $S$, induced on $V(S) \setminus V(R)$.
        \item Every $v \in V(S)$ not in $V(R)$ has degree $\gamma_K(v)$ in $S$.
    \end{enumerate}
\end{dfn}

We can again see that this definition coincides with the original definition of free completions when $K$ is a normal configuration (with the ring set $\mathcal{R}(K)$ containing the ring $R$).

When the surface $\mathbb{S}$ is a sphere and $\mathcal{R}(K)$ contains exactly two faces, we call $K$ an \emph{annular configuration}. 
An \emph{annular island} is obtained by taking the dual of a free completion of an annular configuration and removing vertices that correspond to $R \in \mathcal{R}(K)$.
It corresponds to the fact that a normal island is obtained by taking the inner dual of a free completion of a normal configuration.

For annular configurations, we define reducibility using the following \emph{annular matchings}. % Is this really a good definition? Maybe what we all need is a general ``reducibility" definition for all general configurations having $|\mathcal{R}(K)| > 1$ ??

% We moved annular matchings here, because we only need them in GENERAL configurations.
\begin{dfn}\label{dfn:annular-matching}\showlabel{dfn:annular-matching}
  Let $k$ be a positive integer and let $0 < a < b$ be positive integers.
  Let $M$ be a matching with \textbf{one of} the following properties.
  
  \begin{enumerate}
      \renewcommand{\labelenumi}{(\roman{enumi})}
      \item There are no matches $\{i, j\} \in M, i < j$ such that $i < a < j$, and the two disjoint matchings $M_1 = \{\{i,j\} \in M \mid i < j < a\}$ and $M_2 = \{\{i,j\} \in M \mid a < i < j\}$ are both planar matchings. 
      \item There exists a match $\{a', b'\} \in M, a' < a < b'$ such that, the matching $M' = \{\{f(i), f(j)\} \mid \{i, j\} \in M, \{i,j\} \neq \{a', b'\}\}$ is a planar matching.
      Here, $f$ is a mapping defined as below.
      \begin{align*}
        f(x) &= \begin{cases}
            x & (x < a') \\
            x - a' + b - a + a'- 2 & (a' < x < a)\\ 
            x - a + a' + b - b' - 1 & (a \leq x < b') \\
            x - b' + a' - 1 & (b' < x)
        \end{cases}
    \end{align*}
  \end{enumerate}

  We call such $M$ a \emph{annular matching} of size $(a, b-a)$.
\end{dfn}

\begin{dfn}
    Let $K$ be a general configuration with $\mathcal{R}(K)$ containing two rings $R_1, R_2$.
    We label the edges in $R_1$ with numbers $\{1, \cdots, r_1\}$ and edges in $R_2$ with numbers $\{r_1 + 1, \cdots, r_1 + r_2\}$, where $r_1 = |R_1|$ and $r_2 = |R_2|$.
    Labelings are done in a direction such that the two disks bounded by $R_1$ and $R_2$ preserve the surface's orientability.
    
    Then, with a coloring $\kappa: \{1, \cdots, r_1 + r_2\} \to \{0,1,2\}$, we define \emph{$\theta$-fitting} with the above edge/label relationship.
    The notions of \emph{consistency}, \emph{D-reducibility}, \emph{C-reducibility} are defined in the same way as for normal configurations.
\end{dfn}

% =================================================
\subsection{Weak appearance of configurations}
\label{sect:weak-appear}\showlabel{sect:weak-appear}
% =================================================

The definition of the appearance of configurations can also be expanded to general configurations.
We will define another term, ``weak appearance," for use in later sections.

\begin{dfn}
    Let $K$ be a general configuration with free completion $S$, and let $G'$ be a triangulation.
    $K$ is said to \emph{appear} in $G'$ if there exists a mapping $f: V(S) \to V(G')$ such that
    \begin{enumerate}
      \renewcommand{\labelenumi}{(\roman{enumi})}
        \item There is no $v, u \in V(S)$ where at least one of $v, u$ is in $V(K)$ such that $v \neq u$ and $f(v) = f(u)$.
        \item For any $v \in V(K)$, $\gamma_K(v) = \text{deg}_{G'}(f(v))$.
        \item For any $v, u, w \in V(K)$ such that the three vertices are vertices of a triangle face in $G(K)$, the three vertices $f(v), f(u), f(w)$ are vertices of a triangle face in $G'$.
    \end{enumerate}
    Furthermore, $K$ is said to be \emph{weakly appearing} in $G'$ when condition (i) is dropped.
\end{dfn}

We also define a term closely related to weak appearance as follows.

\begin{dfn}
    Let $K$ and $K'$ be general configurations.
    $K'$ is said to be a \emph{wrapping} of $K$ if there exists a mapping $f: V(K) \to V(K')$ such that
    \begin{enumerate}
      \renewcommand{\labelenumi}{(\roman{enumi})}
        \item For any $v \in V(K)$, $\gamma_K(v) = \gamma_{K'}(f(v))$.
        \item For any $v, u, w \in V(K)$ such that the three vertices are vertices of a triangle face in $G(K)$, the three vertices $f(v), f(u), f(w)$ are vertices of a triangle face in $G(K')$.
    \end{enumerate}
\end{dfn}

By definition, if a general configuration $K'$ is a wrapping of another general configuration $K$, and $K'$ appears in $G'$, $K$ weakly appears in $G'$.

Why do we need these terms?
As discussed in this section, the appearance of a reducible configuration $K$ in $G'$ results in a reduction of $G'$.
However, the \emph{discharging} method used in Section \ref{sect:discharging} only asserts the weak appearance of some reducible configuration $K$, which is not enough to warrant a reduction of $G'$.

This is when the wrappings of $K$ come into play.
Even if we can only assert the weak appearance of $K$, if we find out that for any wrapping of $K$ (say $K'$) is also reducible, we can use $K'$ instead for finding a reduction of $G'$.

More about this is discussed in Section \ref{sect:lowrep}.

% =================================================
\subsection{Distance five in reducible configurations}
\label{sect:dist5}\showlabel{sect:dist5}
% =================================================

In the proof of the 4CT in \cite{RSST}, there are 633 reducible configurations, all of which are of diameter at most four. 
On the other hand, in our set of reducible configurations $\mathcal{K}$, there are many configurations whose diameter is five (although we do not use any reducible configuration of diameter six or more). 

If a reducible configuration $K$ in a planar graph is of diameter at most four, $K$ is an induced subgraph (or otherwise, there is a separating cycle of order at most five with both sides having at least two vertices, which would not exist in a minimal counterexample). 
This is also the case for the doublecross. 
However, in the torus case, we have to deal with the case when a reducible configuration is not induced. 

As with the sections above, $G'$ is a triangulation that is embedded in the torus, and $K$ is a reducible configuration in $G'$. 
Let us assume that there are two vertices $u, v \in V(K)$ that are at distance five in $K$, and there is an edge $uv \in E(G')$. 
Then we obtain a cycle $C$ of length exactly six, and %they satisfy one of the following:
assume that $C$ is a contractible cycle in $G'$. 
    
    In the cubic graph $G$, $C$ corresponds to a 6-cut, which yields a contractible curve of order exactly six. 

Using Lemma \ref{minc} and assuming Lemma \ref{lem:6,7-cut}, which we will prove in Section \ref{sect:6,7-cut}, we get the following fact.

\begin{fact}\label{fact:cont-cycle}\showlabel{fact:cont-cycle}
    If $G'$ is a minimal counterexample, then $G'$ contains no cycles with any of the following properties: \begin{itemize}
        \item A contractible cycle of length at most four that divides $G'$ into two nonempty components.
        \item A contractible cycle of length five that divides $G'$ into two components of order at least two.
        \item A contractible cycle of length six that divides $G'$ into two components of order at least four.
    \end{itemize}
\end{fact}

To handle the above case, we show the following.
% \textbf{COMPUTER: cycle6 imply cycle5}
\begin{lem}\label{pairs6}\showlabel{pairs6}
    Let $C$ be a cycle of length $6$ in $G'$, as discussed above. Then, it satisfies Fact \ref{fact:cont-cycle}, and thus cannot appear in a minimal counterexample.
\end{lem}

The proof of Lemma \ref{pairs6} is the same as that of Lemma 3.16 in \cite{proj2024} except that we only care about a contractible cycle here (the noncontractible cycle case will be handled in Section \ref{sect:lowrep}). We use the same computer program. In almost all cases, we conclude that the cycle would contradict Fact \ref{fact:cont-cycle}. However, there are 23 cases that we cannot immediately conclude. We show them in Figure \ref{fig:dist5}. For the remaining 22 cases, we confirm that there are separating 4-cycles, which would still contradict Fact \ref{fact:cont-cycle}. For one case drawn in the second column from the left in the last row, another reducible configuration C(2)\footnote{C(2) is the second configuration drawn in Figure \ref{fig:confs_degree56}, which is explained later.} appears, so we do not care about the case. Therefore, Lemma \ref{pairs6} holds.

\begin{figure}[htbp]
  \centering
  \includegraphics[width=15cm]{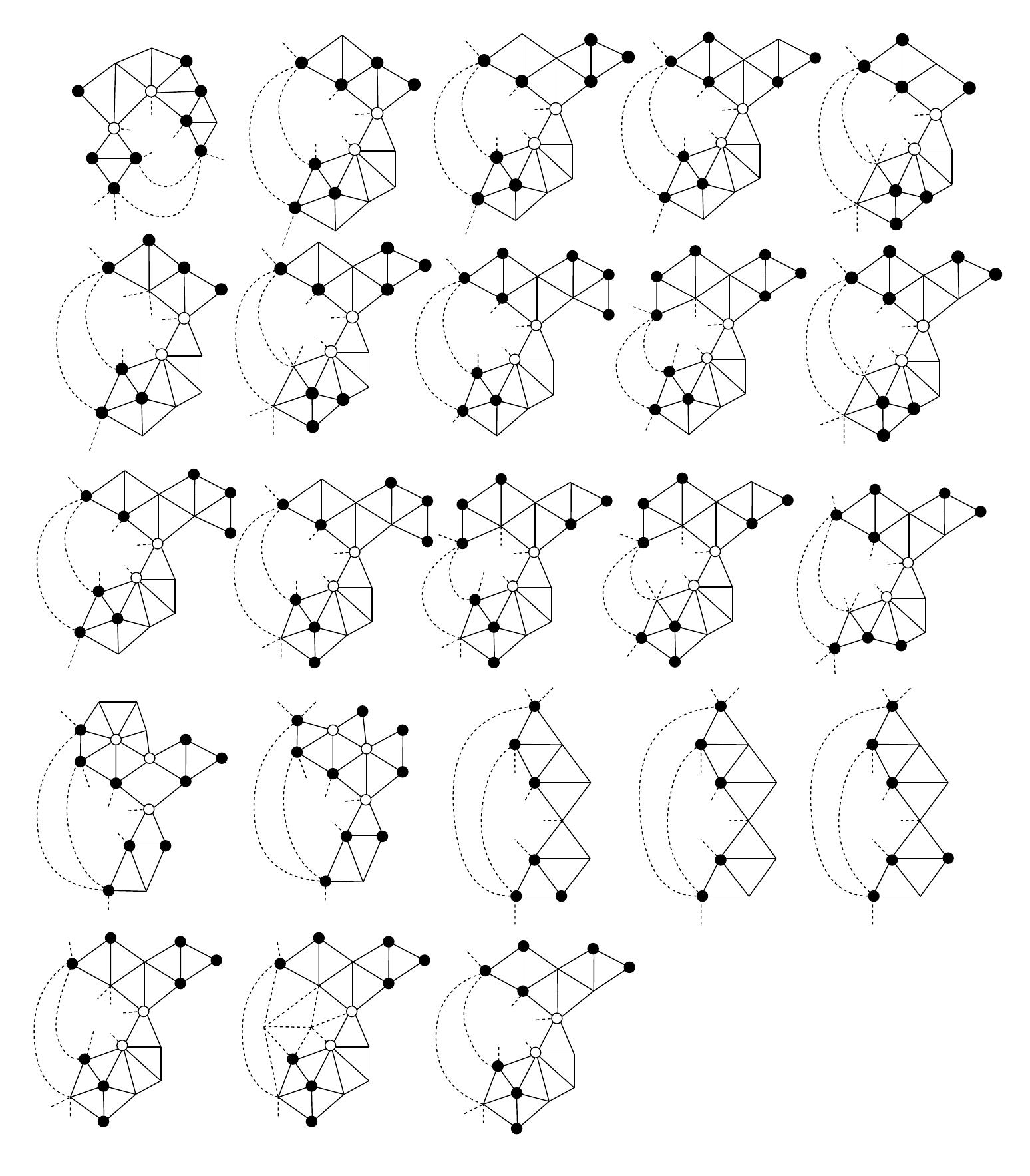}
  \caption{The 23 cases that remain after computer-check in the proof of Lemma \ref{pairs6}. The dotted lines represent edges when we assume two vertices of distance five are adjacent.}
  \label{fig:dist5}
\end{figure}
 
% ====================================================
\section{Discharging process}
\label{sect:discharging}\showlabel{sect:discharging}
% ====================================================

% TODO: T_0 (initial charge), T(final charge)
Let $G'$ be an internally 6-connected triangulation in the torus\footnote{We assume that the representativity of $G'$ is at least two. Note that it could be exactly two, so there could be multiple edges. But the embedding is 2-cell, so we can still apply the rule below}. We start by assigning to each vertex $v \in V(G')$ the value 
$$
    T_0(v) := 10(6 - d(v)),
$$
which we call the \emph{initial charge at $v$}. It follows by Euler's formula (see Lemma \ref{lem:120} below) that 
$$
    \sum_{v\in V(G')} T_0(v) = 0.
$$
Now we redistribute the charge among vertices by applying the \emph{discharging rules} that are shown in Section \ref{sect:rule} in the appendix. There are 201 rules (compared to 33 rules for the proof of the 4CT and that for the doublecross case).

Discharging rules are formally defined as follows.

\begin{dfn}\label{dfn:rule}\showlabel{dfn:rule}
  A \emph{rule} is six-tuple $R = (G(R), \beta_R, \delta_R, r(R), s(R), t(R))$, so that
  \begin{enumerate}
      \renewcommand{\labelenumi}{(\roman{enumi})}
      \item $G(R) = (V(R), E(R))$ is a near-triangulation, and for each $v \in V(R)$, $G(R) - v$ is connected.
      \item $\beta_R \colon V(R) \rightarrow \mathbb{N}$, $\delta_R \colon V(R) \rightarrow \mathbb{N} \cup \{\infty\}$, such that $5 \leq \beta_R(v) \leq \delta_R(v)$ for each $v \in V(R)$.
      \item $r(R)$ is a positive integer.
      \item $s(R), t(R) \in V(R)$ are distinct adjacent vertices.
  \end{enumerate}
  A configuration $K$ \emph{obeys} the rule $R$ if $G(K) = G(R)$ and every vertex $v \in V(K)$ satisfies $\beta_R(v) \leq \gamma_K(v) \leq \delta_R(v)$.
\end{dfn}

A rule is a near-triangulation in which each vertex $v$ is given a range of possible degrees, the interval $[\beta_R(v), \delta_R(v)]$. The value $r(R)$ represents the amount of charge that is sent from the vertex $s(R)$ to its neighbor $t(R)$. We have $r(R) = 1$ or $r(R)=2$ for all our rules. We describe the degree range $[\beta_R, \delta_R]$ by the same convention as in configurations. There are three cases. When $\beta_R(v) = \delta_R(v)$, we describe $\beta_R(v)$ by one of the shapes shown in Figure \ref{fig:shapes}. When $5 = \beta_R(v) < \delta_R(v)$, we describe $\delta_R(v)$ as the vertex shape and set $-$ next to the vertex. When $\beta_R(v) < \delta_R(v) = \infty$, we describe $\beta_R(v)$ as the vertex shape and set $+$ next to the vertex. All our rules have one of these three types of degree ranges. 

The set of 201 rules exhibited in Figure \ref{fig:rules} in the appendix is denoted by $\mathcal{R}$. 
An edge with an arrow represents the direction in which a charge moves. Thus, $s(R), t(R)$ are two vertices that are incident with an edge with an arrow.
The number of arrows in the figure represents $r(R)$.

We calculate the amount of charge sent along an edge $st$ in $G'$ by checking which rules can be used on this edge. For a rule $R\in \mathcal{R}$, $r(R)$ gives the amount of charge to send, $s(R)$ is the vertex that sends the charge, and $t(R)$ is the vertex that receives charge. We \textit{apply a rule $R$} when a configuration $K$, which obeys $R$ weakly appears in $G'$. Note that each\footnote{The very first rule (R1) is the only exception that is used only once.} rule $R\in \mathcal{R}$ can be applied at the edge $st$ in two ways, where the triangles in $G(R)$ have the same orientation as in $G'$, and when they have opposite orientation.\footnote{
We just consider the orientation of triangles locally around the edge $st$.}

\begin{dfn}\label{dfn:amount}\showlabel{dfn:amount}
  For adjacent vertices $u, v \in V(G')$, we define $\phi(u,v)$ as the sum of the values $r(R)$ over all rules $R\in \mathcal{R}$ that are applied with $s(R) = u$, and $t(R) = v$. This is \emph{the amount of charge sent from $u$ to $v$}.
\end{dfn}

\begin{dfn}\label{dfn:finalcharge}\showlabel{dfn:finalcharge}
    For a vertex $u \in V(G')$, we set
    \[
      T(u) := T_0(u) + \sum_{v \sim u} (\phi(v, u) - \phi(u,v)).
    \]
    The obtained value is the \emph{final charge} at $u$.
\end{dfn}

The discharging method uses the following easy observation.

\begin{lem}
\label{lem:120}\showlabel{lem:120}
$\sum_{u \in V(G')} T(u) = 0$.
\end{lem}

\begin{proof}
By applying any discharging rule, the total sum of all charges remains the same. Thus, $\sum_{u \in V(G')} T(u) = \sum_{u \in V(G')} 10(6-d(u))$. If $n$ is the number of vertices of $G'$, then Euler's formula for the triangulation $G'$ implies that $\tfrac12 \sum_{u\in V(G')}d(u) = |E(G')| = 3n$, and thus we have:
 \[
 \sum_{u \in V(G')} T(u) = \sum_{u \in V(G')} 10(6-d(u)) = 60n - 10\sum_{u \in V(G')}d(u) =
 60n-60n = 0.
 \]
\end{proof}

Lemma \ref{lem:120} implies that either (i) there is a vertex $u$ with $T(u) > 0$, or (ii) $T(u) = 0$ for each vertex $u$ in $V(G')$. Known proofs of the 4CT show that for any vertex $u$ with $T(u)>0$, there is a reducible configuration in the vicinity of $u$. This yields a contradiction to the assumption that we have a minimum counterexample. 

We have a similar result on the torus, but there is some difference.

%{\bf below, we have to show the torus case}
% \textbf{COMPUTER: discharging 7 $\leq$ degree $\leq$ 11}
\begin{lem}\label{T(v)>0}\showlabel{T(v)>0}
  Let $G$ be a minimal counterexample with representativity at least two. Let $G'$ be the dual of $G$. If there is a vertex $v$ of $V(G')$ with $T(v) > 0$, a configuration that is isomorphic to one of the reducible configurations of $\mathcal{K}$ weakly appears.    
\end{lem}

We show Lemma \ref{T(v)>0} by computer check, but we can show this lemma without the help of computers when $d_{G'}(v)=5$ or $d_{G'}(v)=6$ (Lemma \ref{lem:deg5deg6}) or $d_{G'}(v) \geq 12$ (Lemma \ref{lem:12}). For other cases, computer-free proof would be too long. Note also that the representativity one case is handled separately, see Lemma \ref{lem:rep1}. 
\subfile{tikz/confs}

\begin{lem}
  \label{lem:deg5deg6}
  \showlabel{lem:deg5deg6}
  Let $G' = (V, E)$ be an internally 6-connected triangulation. Let $v \in V$ be a vertex of degree 5 or 6.
  Suppose that no configurations shown in Figure $\ref{fig:confs_degree56}$ weakly appear in $G'$. Then, $T(v) = 0$.
\end{lem}

The proof of Lemma \ref{lem:deg5deg6} is the same as that given in \cite{proj2024}. 
Indeed the set of configurations in Figure \ref{fig:confs_degree56} includes the configurations used in hand-written proofs of \cite{proj2024}. The rules in $\mathcal{R}$ such that the degree of a vertex that receives or sends charge is at most 6 (i.e. $\delta_R(s(R)) \leq 6$ or $\delta_R(t(R)) \leq 6$) is the same as that of rules used in \cite{proj2024}. Hence, the proof of Lemma \ref{lem:deg5deg6} is the same as that of Lemma about the final charge of degree $5$ or $6$ in \cite{proj2024}.
We denote a configuration in Figure \ref{fig:confs_degree56} by C(i) $(1 \leq i \leq 29)$ in the natural order, from top to bottom in a row and from left to right in a column.
The bold lines in Figure \ref{fig:confs_degree56} represents contraction edges used for C-reducibility.

% ================================================
\subsection{Maximum charge that a vertex can send}
\label{sect:maxcharge}
\showlabel{sect:maxcharge}
% ================================================

%{\bf below, we have to show the torus case}

Let $n$ be a positive integer. We want to figure out the case that a vertex sends charge $n$ to a vertex of degree at least $7$. We calculate it by combining a set of rules of $\mathcal{R}$ such that the edge where charge moves coincide, and then by checking that none of our reducible configurations of $\mathcal{K}$ weakly appears. We use computer-check to calculate it. The pseudocode of this algorithm is the same as the algorithm described in \cite{proj2024}. 
We now show the following.

% \textbf{COMPUTER: charge send}
\begin{lem}\label{lem:vsends}\showlabel{lem:vsends}
Assume that none of our reducible configurations in the set $\mathcal{K}$ weakly appears in $G'$. Then, any vertex $v \in V(G')$ can send charge at most $6$ to any one of its neighbors.   
If it sends charge $6$ ($5$ resp.), then we have one of the cases shown in Figure \ref{fig:send6} (Figure \ref{fig:send5}, resp.) in the appendix. 
If a vertex of degree $5$ sends charge $4$, we have one of the cases in Figure \ref{fig:deg5send4} in the appendix.
\end{lem}

We denote the case in row $i$ and column $j$ in Figure \ref{fig:send6} by send6($i$, $j$). We use the same notation to denote the case in Figure \ref{fig:send5} by send5($i$, $j$) and Figure \ref{fig:deg5send4} by send4($i$, $j$), respectively. 

Using Lemma \ref{lem:vsends}, we prove the following.

\begin{lem}\label{lem:edgesends}\showlabel{lem:edgesends}
Assume that none of our reducible configurations in the set $\mathcal{K}$ weakly appears in $G'$. Let $u$ be a vertex of degree at least $7$ in $G'$ and $v_1, v_2, v_3$ be three consecutive neighbors of $u$ 
%when neighbors of $u$ are 
listed in the clockwise order in the embedding of $G'$.
If $\phi(v_1, u) = 6$, $\phi(v_2, u) \leq 5$ and $\phi(v_3, u) \leq 4$ hold. Moreover, if all the equalities hold ($\phi(v_2, u) = 5, \phi(v_3, u) = 4$), we have one of the cases shown in Figure \ref{fig:send6,5,4}.
\end{lem}
\begin{proof}
    When the degree of $v_1$ is 7 (that corresponds to send6(2,4), or send6(2,5)), $\phi(v_2, u) \leq 4$ holds by Lemma \ref{lem:vsends}.
    
    We consider the case when degree of both $v_1$ and $v_2$ is 6 (that corresponds to send6(1,4), send6(2,1)).
    First, we consider the case when $v_1$ sends charge 6 in send6(1, 4). 
    The left figure in Figure \ref{fig:proof_lem:edgesends} represents this situation. 
    The degree of $v_4$ is at least 7, otherwise, C(2), or C(3) weakly appear. 
    The degree of $v_3$ is not 5, otherwise, C(2) weakly appears. 
    When the degree of $v_3$ is 6, $\phi(v_2, u)$ is at most 2, otherwise one of C(17), C(16), C(24) weakly appears in the following three cases $d(v_5)=5$, or $d(v_5)=6,d(v_6)=5$, or $d(v_5)=6,d(v_6)=6,d(v_7)=5$ respectively.
    When the degree of $v_3$ is at least 7, $\phi(v_2, u)$ is at most 1 by the rules we used.
    
    Second, we consider the case when $v_1$ sends charge 6 in send6(2, 1). 
    The right figure in Figure \ref{fig:proof_lem:edgesends} represents this situation. 
    The degree of $v_4$ is at least 7 since otherwise C(14) or C(15) weakly appears.
    The value $\phi(v_2, u)$ is at most 3, since otherwise one of C(4), C(25) weakly appears in the following two cases $d(v_5)=5$, or $d(v_5)=6, d(v_6)=5$ respectively.
    
    It remains to handle the case when the degree of $v_1$ is 6 and the degree of $v_2$ is 5 (that corresponds to send6(1, $i$) ($1 \leq i \leq 5$), send6(2, $j$) ($1 \leq j \leq 3$). When $\phi(v_2, u) \geq 5$, the only possible case is that $v_2$ sends charge 5 in send5(1,2) by Lemma \ref{lem:vsends}, so $\phi(v_2, u) \leq 5$ holds. The equality holds only when $v_1$ sends charge 6 in send6(1,2), send6(1,3), or send6(2,2).
    The degree of $v_3$ is 5 in every case. Let $w, x, y$ be neighbors of $v_3$ such that $u, v_2, x, y, z$ are listed in the clockwise order in the embedding of $G'$. If the degree of $y$ is 5, C(2) weakly appears. If the degree of $y$ is 6 and the degree of $z$ is 5, C(17) weakly appears. Hence $\phi(v_3, u) \leq 4$ by Lemma \ref{lem:vsends} (Figure \ref{fig:send5}). Moreover, the equality holds only when  
    $\phi(v_3, u)$ is exactly $4$ in send4(1,4), or send4(2,1). However C(19) weakly appears in send4(2,1), so send4(1,4) is the only case. Hence, if $\phi(v_3, u)$ is exactly $4$, all cases are enumerated in Figure \ref{fig:send6,5,4}.
\end{proof}
\begin{figure}
    \centering
    \includegraphics[height=4cm]{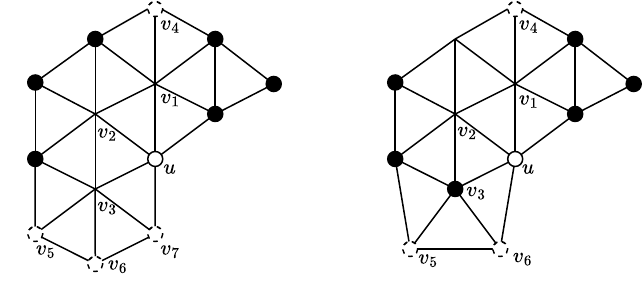}
    \caption{The figures used in the proof of Lemma \ref{lem:edgesends}}
    \label{fig:proof_lem:edgesends}
\end{figure}
\tikzset{deg5/.style={thick, circle, draw, fill=black, inner sep=1.5pt,}}
\tikzset{deg6/.style={thick, circle, draw, fill=black, inner sep=0pt,}}
\tikzset{deg7/.style={thick, circle, draw, fill=white, inner sep=2pt,}}
\tikzset{deg8/.style={thick, rectangle, draw, fill=white, inner sep=2pt,}}
\tikzset{deg9/.style={thick, regular polygon, regular polygon sides=3, rotate=180, draw, fill=white, inner sep=1pt,}}
\tikzset{deg10/.style={thick, regular polygon, draw, fill=white, inner sep=2pt,}}
\tikzset{->-/.style={decoration={
    markings,
    mark=at position .6 with {\arrow{>}}}, postaction={decorate}}}
\tikzset{->>-/.style={decoration={
    markings, 
    mark=at position .5 with {\arrow{>}};, 
    mark=at position .6 with {\arrow{>}};}, postaction={decorate}}}

\begin{figure}[htbp]
\centering
\begin{minipage}{0.3\linewidth}
\centering
\captionsetup{width=.95\linewidth}
\begin{tikzpicture} [baseline=1.5cm]
    \node [deg6] at (2.151, 1.5) (v0) {};
    \node [deg7] at (2.151, 0.5) (v1) {};
    \node [below = 0.15 cm of v1, anchor=center] (v1+) { $+$ };
    \node [deg5] at (1.285, 1.0) (v2) {};
    \node [deg5] at (1.285, 2.0) (v3) {};
    \node [deg5] at (3.017, 1.0) (v4) {};
    \node [deg5] at (3.017, 2.0) (v5) {};
    \node [deg5] at (0.3, 1.274) (v6) {};
    \node [deg6] at (4.002, 1.274) (v7) {};
    \node [deg5] at (4.002, 0.5) (v8) {};
    \node [deg5] at (3.017, 0) (v9) {};
    \draw [->-] (v0) -- (v1);
    \draw [->-] (v4) -- (v1);
    \draw [->-] (v8) -- (v1);
    \foreach \u / \v in {v0/v1, v0/v2, v0/v3, v0/v4, v0/v5, v1/v2, v1/v4, v1/v9, v2/v3, v4/v5, v3/v6, v2/v6, v3/v6, v4/v7, v5/v7, v1/v8, v4/v8, v7/v8, v8/v9}
        \draw (\u) -- (\v);
\end{tikzpicture}
\end{minipage}
\begin{minipage}{0.3\linewidth}
\centering
\captionsetup{width=.95\linewidth}
\begin{tikzpicture} [baseline=1.5cm]
    \node [deg6] at (2.151, 1.5) (v0) {};
    \node [deg7] at (2.151, 0.5) (v1) {};
    \node [below = 0.15 cm of v1, anchor=center] (v1+) { $+$ };
    \node [deg5] at (1.285, 1.0) (v2) {};
    \node [deg5] at (1.285, 2.0) (v3) {};
    \node [deg5] at (3.017, 1.0) (v4) {};
    \node [deg5] at (3.017, 2.0) (v5) {};
    \node [deg6] at (0.3, 1.274) (v6) {};
    \node [deg6] at (4.002, 1.274) (v7) {};
    \node [deg5] at (4.002, 0.5) (v8) {};
    \node [deg5] at (0.3, 0.5) (v9) {};
    \node [deg5] at (3.017, 0) (v10) {};
    \draw [->-] (v0) -- (v1);
    \draw [->-] (v4) -- (v1);
    \draw [->-] (v8) -- (v1);
    \foreach \u / \v in {v0/v1, v0/v2, v0/v3, v0/v4, v0/v5, v1/v2, v1/v4, v1/v10, v2/v3, v4/v5, v3/v6, v2/v6, v3/v6, v4/v7, v5/v7, v1/v8, v4/v8, v7/v8, v1/v9, v2/v9, v6/v9, v8/v10}
        \draw (\u) -- (\v);
\end{tikzpicture}
\end{minipage}
\begin{minipage}{0.3\linewidth}
\centering
\captionsetup{width=.95\linewidth}
\begin{tikzpicture} [baseline=1.5cm]
    \node [deg6] at (2.151, 1.5) (v0) {};
    \node [deg7] at (2.151, 0.5) (v1) {};
    \node [below = 0.15 cm of v1, anchor=center] (v1+) { $+$ };
    \node [deg5] at (1.285, 1.0) (v2) {};
    \node [deg6] at (1.285, 2.0) (v3) {};
    \node [deg5] at (3.017, 1.0) (v4) {};
    \node [deg5] at (3.017, 2.0) (v5) {};
    \node [deg5] at (0.3, 1.274) (v6) {};
    \node [deg6] at (4.002, 1.274) (v7) {};
    \node [deg5] at (4.002, 0.5) (v8) {};
    \node [deg5] at (0.3, 0.5) (v9) {};
    \node [deg5] at (3.017, 0) (v10) {};
    \draw [->-] (v0) -- (v1);
    \draw [->-] (v4) -- (v1);
    \draw [->-] (v8) -- (v1);
    \foreach \u / \v in {v0/v1, v0/v2, v0/v3, v0/v4, v0/v5, v1/v2, v1/v4, v1/v10, v2/v3, v4/v5, v3/v6, v2/v6, v3/v6, v4/v7, v5/v7, v1/v8, v4/v8, v7/v8, v1/v9, v2/v9, v6/v9, v8/v10}
        \draw (\u) -- (\v);
\end{tikzpicture}
\end{minipage}
\caption{Three consecutive neighbors $v_1,v_2,v_3$ of $u$ satisfy $\phi(v_1,u) = 6, \phi(v_2,u) = 5, \phi(v_3, u) = 4$ where an arrow in the center represents the edge from $v_1$ to $u$ and two arrows in the right side represent the edges from $v_2, v_3$ to $u$.}
\label{fig:send6,5,4}
\end{figure}
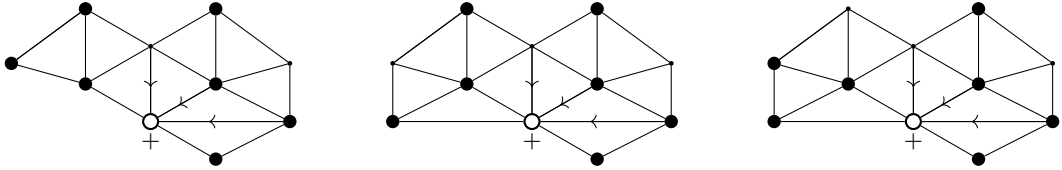

\begin{lem}\label{lem:sendavg5+1}\showlabel{lem:sendavg5+1}
    Assume that none of our reducible configurations in the set $\mathcal{K}$ weakly appears in $G'$.
    Let $u$ be a vertex of degree at least 7 in $G'$, and $v_1, v_2, ..., v_n (1 \leq n \leq d_{G'}(u))$ be consecutive neighbors of $u$ listed in the clockwise order in the embedding of $G'$.
    The sum of $\phi(v_i, u) (1 \leq i \leq n)$ is at most $5n + 1$.
\end{lem}
\begin{proof}
    This proof is by induction on $n$.
    When $n \leq 3$, the claim holds by Lemma \ref{lem:vsends}, \ref{lem:edgesends}.
    We assume $n \geq 4$. Suppose for a contradiction that a counterexample to the claim exists. 
    By Lemma \ref{lem:edgesends}, either of $\phi(v_1, u) \leq 5$, or $\phi(v_1, u) = 6, \phi(v_2, u) \leq 4$, or $\phi(v_1, u) = 6, \phi(v_2, u) = 5, \phi(v_3, u) \leq 4$ holds. In every case, the remaining vertices contradict the induction hypothesis, so the claim holds.
\end{proof}

Using these lemmas, we obtain the following result.
\begin{lem}\label{lem:12}\showlabel{lem:12}
Assume that none of our reducible configurations in the set $\mathcal{K}$ weakly appears in $G'$.
Any vertex $v$ of degree at least 12 has final charge less than $0$.
\end{lem}

\begin{proof}
    We fix the clockwise order of neighbors of $v$ in the embedding of $G'$. When we say the ``next" neighbor, it means the next neighbor in this clockwise order.
    When a neighbor of $v$ sends charge 6 to $v$, the next neighbor sends charge at most 4, or the next neighbor sends charge 5 and the next after next neighbor sends charge at most 4 by Lemma \ref{lem:edgesends}. Hence, there is one neighbor that sends charge at most 4 when a neighbor of $v$ that sends a charge of $6$ so the average amount of charges the neighbors of $v$ send to $v$ is at most $5$. 
    This implies that the final charge of $v$ is less than 0 ($T(v) < 0$) if the degree of $v$ is at least 13 ($10(6-d)+5d=60-5d<0$). If the degree of $v$ is 12, the final charge can become exactly 0. We show the case when the degree of $v$ is exactly 12 in the rest of the proof. We denote $u_i$($0 \leq i < 12$) as neighbors of $v$ such that $u_{i+1}$ is the next neighbor of $u_{i}$ ($0 \leq i < 12$, $u_i (i \geq 12)$ represents $u_{i \bmod 12}$).

    \begin{itemize}
        \item [(1)] If $\phi(u_i, v)=\phi(u_{i+1}, v)=5$ and the average amount of charges the neighbors of $v$ send to $v$ is 5, then for each $0 \leq i < 12$, $\phi(u,v) = 5$ holds.
    \end{itemize}
    Suppose for a contradiction that a counterexample of (1) exists.
    Let $i, j$ be two different indexes such that one of the followings  holds:
    \begin{itemize}
        \item $i < j$, $\phi(u_{i-1}, v) \neq 5, \phi(u_{j+1}, v) \neq 5, \phi(u_k, v) = 5$ ($i \leq k \leq j$), or 
        \item $i > j$, $\phi(u_{i-1}, v) \neq 5, \phi(u_{j+1}, v) \neq 5, \phi(u_k, v) = 5$ ($i \leq k < 12$ or $ 0 \leq k \leq j$).
    \end{itemize}
    A counterexample of (1) has such indexes. We show only the first case (the proof of the second case is the same.). By Lemma \ref{lem:edgesends}, $\phi(u_{i-1}, v) < 5, \phi(u_{j+1}, v) < 5$. By Lemma \ref{lem:sendavg5+1}, the vertices $v_k$ ($j+2 \leq k < 12$ or $0 \leq k \leq i-2$) send at most $5(9-j+i)+1$ (average charge $5$ plus one). Hence, the claim (1) holds.

    \begin{itemize}
        \item [(2)] If $\phi(u_i, v)=6, \phi(u_{i+1}, v)=5, \phi(u_{i+2}, v)=4$, then the average amount of charges the neighbors of $v$ send to $v$ is less than 5.
    \end{itemize}
    If $\phi(u_i, v)=6, \phi(u_{i+1}, v)=5, \phi(u_{i+2}, v)=4$, degree of $u_{i+3}$ is 5 by Lemma \ref{lem:edgesends} (Figure \ref{fig:send6,5,4}). The amount of charges a vertex of degree 5 sends is at most 5 by Lemma \ref{lem:vsends}. If $\phi(u_{i+3}, v) < 5$, claim (2) holds by (1). 
%    considering the above proof. 
    Moreover, if $\phi(u_{i+3}, v) = 5$ and $\phi(v_{i+4}, v) < 5$, claim (2) holds by the same reason. By claim (1), neither $\phi(u_{i+3}, v) = 5$ nor $\phi(u_{i+4}, v) = 5$. Hence the only possible case that contradicts claim (2) is $\phi(u_{i+3}, v) = 5$ and $\phi(u_{i+4}, v) = 6$. 
    However, C(17) weakly appears by considering $\phi(u_{i+2}, v)=4, \phi(u_{i+3}, v)=5, \phi(u_{i+4}, v)=6$ and Lemma \ref{lem:edgesends} (Figure \ref{fig:send6,5,4}). Hence claim (2) holds.

\medskip

    If the average amount of charges the neighbors of $v$ send to $v$ is exactly 5, by combining claims (1) and (2), one of the followings   happens: (i) $\phi(u_i, v) = 5$ for each $0 \leq i < 12$, or (ii) $\phi(u_i, v) = 6$ ($i \bmod 2 = 0$), $\phi(u_j, v) = 4$ ($i \bmod 2 = 1$).

    First, we handle case (i). 
    We apply Lemma \ref{lem:vsends} (especially Figure \ref{fig:send5}).
    The degree of $u_i$ is 5 or 6 or 7. If the degree of $u_i$ is 7 and the degree of $u_{i+1}$ is 5 or 6, then $\phi(u_{i+1}, v)$ is not 5. None of degrees of $u_i, u_{i+1}$ is 7. Hence, all degrees of $u_i$ ($0 \leq i < 12$) are 5 or 6. 

    Suppose all degrees of $u_i$ are 6. It is possible when all $u_i$ ($0 \leq i < 12$) sends charges of 5 in send5(6, 1) since for any other cases described in Figure \ref{fig:send5}, there exists a vertex of degree 5 that is adjacent to $v$. However, C(17) weakly appears in this case. Hence, the degree of some $u_i$ is 5.

    Suppose some $u_i$ is in send5(1, 2). The degree of either $u_{i+1}$ or $u_{i-1}$ is 5. Assume without loss of generality that the degree of $u_{i+1}$ is 5. Let $w, x, y$ be neighbors of $u_{i+1}$ such that $v, u_i, w, x, y$ are listed in clockwise order of the embedding of $G'$. The degree of $w$ is 6. If the degree of $x$ is 5, C(2) weakly appears. If the degree of $x$ is 6 and the degree of $y$ is 5, C(17) weakly appears. Hence $\phi(u_{i+1}, v)$ cannot be 5.

    Suppose some $u_i$ is in send5(1, 3). Assume without loss of generality that both degrees of $u_{i+1}, u_{i+2}$ are 5 and the degrees of neighbors of $u_{i+1}$ except $v, u_i, u_{i+2}$ are 6. By the condition of degrees, $u_{i+1}$ must be in send5(1, 3). However, C(20) weakly appears in this case.

    When some $u_i$ is in send5(1, 1), degree of both $u_{i-1}$ and $u_{i+1}$ is 5. Hence, both $u_{i-1}$ and $u_{i+1}$ are in send5(1,1). Hence, all vertices $u_i$ ($0 \leq i < 12$) are in send5(1,1). However, C(28) weakly appears in this case.
    
    Second, we handle case (ii). 
    The amount of charges a vertex of degree 5 sends is at most 5 by \ref{lem:vsends}, so the degree of $u_i$ ($i \bmod 2 = 0$) is not 5. We only consider the case that $u_i$ ($i \bmod 2 = 0$) is in one of the following cases: send6(1, 1), send6(1, 4), send6(1, 9) since for  other cases, the degree of $u_{i-2}$ or $u_{i+2}$ is 5. 
     % send6(1, 2), send6(1, 3), send6(1, 5), send6(1, 6), send6(1, 7), send6(1, 8), send6(1, 10)

    Suppose some $u_i$ ($i \bmod 2 = 0$) is in send6(1, 4). The degree of $u_{i+1}$ or $u_{i-1}$ is 6. Assume without loss of generality that the degree of $u_{i+1}$ is 6, so $u_{i+2}$ is also in send6(1, 4). In this case, $\phi(u_{i+1}, v)$ is at most 2.

    Suppose some $u_i$ ($i \bmod 2 = 0$) is in send6(1, 9). The degree of $u_{i+1}$ is 5. If $\phi(u_{i+1}, v)$ is 4, the degree of $u_{i+2}$ is 5 by Lemma \ref{lem:vsends} (Figure \ref{fig:deg5send4}). In this case, C(1) weakly appears.

    Hence all vertices $u_i$ ($0 \leq i < 12$, $i \bmod 2 = 0$) are in send6(1, 1).  However, C(29) weakly appears in this case.
\end{proof}

\begin{lem}\label{lem:onesidesend}\showlabel{lem:onesidesend}
Let $H$ be an internally 6-connected triangulation in the plane (or the torus). Assume that none of our reducible configurations in the set $\mathcal{K}$ weakly appears $H$.
Let $vuw$ be a triangle face in $H$ and assume that the degree of $w$ is at least $9$. The followings hold.
\begin{itemize}
    \item If the degree of $u$ is $5$, then $\phi(u, v) \leq 4$. If the equality holds, we have one of the first three cases shown in Figure \ref{fig:deg5send4} in the appendix.
    \item If the degree of $u$ is $6$, then $\phi(u, v) \leq 3$. If the equality holds, we have a situation shown in Figure~\ref{fig:proj3_deg6_oneside}.
    \item If the degree of $u$ is $7$, then $\phi(u, v) \leq 2$.
    \item If the degree of $u$ is at least $8$, then $\phi(u, v) = 0$.
\end{itemize}
\end{lem}
\tikzset{deg5/.style={thick, circle, draw, fill=black, inner sep=1.5pt,}}
\tikzset{deg6/.style={thick, circle, draw, fill=black, inner sep=0pt,}}
\tikzset{deg7/.style={thick, circle, draw, fill=white, inner sep=2pt,}}
\tikzset{deg8/.style={thick, rectangle, draw, fill=white, inner sep=2pt,}}
\tikzset{deg9/.style={thick, regular polygon, regular polygon sides=3, rotate=180, draw, fill=white, inner sep=1pt,}}
\tikzset{deg10/.style={thick, regular polygon, draw, fill=white, inner sep=2pt,}}
\tikzset{->-/.style={decoration={
    markings,
    mark=at position .6 with {\arrow{>}}}, postaction={decorate}}}
\tikzset{->>-/.style={decoration={
    markings, 
    mark=at position .5 with {\arrow{>}};, 
    mark=at position .6 with {\arrow{>}};}, postaction={decorate}}}

\begin{figure*}[htbp]
\centering
\begin{tabular}{ccc}
\begin{minipage}[t]{0.20\hsize}
\centering
% (page, row, col) = (0, 4, 0)
\begin{tikzpicture} []
    \node [deg6] at (0.8, 0.3) (v0) {};
    \node [deg7] at (1.8, 0.3) (v1) {};
    \node [above = 0.15 cm of v1, anchor=center] (v1+) { $+$ };
    \node [deg6] at (1.3, 1.166) (v2) {};
    \node [deg6] at (0.3, 1.166) (v3) {};
    \node [deg5] at (2.3, 1.166) (v4) {};
    \node [deg5] at (0.8, 2.032) (v5) {};
    \node [deg5] at (1.8, 2.032) (v6) {};
    \draw [->-] (v0) -- (v1);
    \foreach \u / \v in {v0/v1, v0/v2, v0/v3, v1/v2, v1/v4, v2/v3, v2/v4, v2/v5, v2/v6, v3/v5, v4/v6, v5/v6}
        \draw (\u) -- (\v);
\end{tikzpicture}
\end{minipage}
&
\begin{minipage}[t]{0.20\hsize}
\centering
% (page, row, col) = (1, 1, 0)
\begin{tikzpicture} []
    \node [deg6] at (0.8, 0.3) (v0) {};
    \node [deg7] at (1.8, 0.3) (v1) {};
    \node [above = 0.15 cm of v1, anchor=center] (v1+) { $+$ };
    \node [deg5] at (1.3, 1.166) (v2) {};
    \node [deg5] at (0.3, 1.166) (v3) {};
    \node [deg5] at (1.126, 2.151) (v4) {};
    \draw [->-] (v0) -- (v1);
    \foreach \u / \v in {v0/v1, v0/v2, v0/v3, v1/v2, v2/v3, v2/v4, v3/v4}
        \draw (\u) -- (\v);
\end{tikzpicture}
\end{minipage}
&
\begin{minipage}[t]{0.20\hsize}
\centering
% (page, row, col) = (1, 1, 3)
\begin{tikzpicture} []
    \node [deg6] at (0.8, 0.3) (v0) {};
    \node [deg7] at (1.8, 0.3) (v1) {};
    \node [above = 0.15 cm of v1, anchor=center] (v1+) { $+$ };
    \node [deg5] at (1.3, 1.166) (v2) {};
    \node [deg5] at (0.3, 1.166) (v3) {};
    \node [deg5] at (2.24, 1.508) (v4) {};
    \node [deg6] at (1.126, 2.151) (v5) {};
    \draw [->-] (v0) -- (v1);
    \foreach \u / \v in {v0/v1, v0/v2, v0/v3, v1/v2, v1/v4, v2/v3, v2/v4, v2/v5, v3/v5, v4/v5}
        \draw (\u) -- (\v);
\end{tikzpicture}
\end{minipage}
\\
\begin{minipage}[t]{0.20\hsize}
\centering
% (page, row, col) = (1, 2, 1)
\begin{tikzpicture} []
    \node [deg6] at (0.8, 0.3) (v0) {};
    \node [deg7] at (1.8, 0.3) (v1) {};
    \node [above = 0.15 cm of v1, anchor=center] (v1+) { $+$ };
    \node [deg5] at (1.3, 1.166) (v2) {};
    \node [deg6] at (0.3, 1.166) (v3) {};
    \node [deg5] at (2.24, 1.508) (v4) {};
    \node [deg5] at (1.126, 2.151) (v5) {};
    \draw [->-] (v0) -- (v1);
    \foreach \u / \v in {v0/v1, v0/v2, v0/v3, v1/v2, v1/v4, v2/v3, v2/v4, v2/v5, v3/v5, v4/v5}
        \draw (\u) -- (\v);
\end{tikzpicture}
\end{minipage}
&
\begin{minipage}[t]{0.20\hsize}
\centering
% (page, row, col) = (2, 0, 4)
\begin{tikzpicture} []
    \node [deg6] at (0.8, 0.3) (v0) {};
    \node [deg7] at (1.8, 0.3) (v1) {};
    \node [above = 0.15 cm of v1, anchor=center] (v1+) { $+$ };
    \node [deg6] at (1.3, 1.166) (v2) {};
    \node [deg5] at (0.3, 1.166) (v3) {};
    \node [deg5] at (0.8, 2.032) (v4) {};
    \node [deg5] at (1.8, 2.032) (v5) {};
    \draw [->-] (v0) -- (v1);
    \foreach \u / \v in {v0/v1, v0/v2, v0/v3, v1/v2, v2/v3, v2/v4, v2/v5, v3/v4, v4/v5}
        \draw (\u) -- (\v);
\end{tikzpicture}
\end{minipage}
&
\begin{minipage}[t]{0.20\hsize}
\centering
% (page, row, col) = (0, 0, 3)
\begin{tikzpicture} []
    \node [deg6] at (1.3, 0.3) (v0) {};
    \node [deg7] at (2.3, 0.3) (v1) {};
    \node [above = 0.15 cm of v1, anchor=center] (v1+) { $+$ };
    \node [deg5] at (1.8, 1.166) (v2) {};
    \node [deg5] at (0.8, 1.166) (v3) {};
    \node [deg5] at (0.3, 0.3) (v4) {};
    \draw [->-] (v0) -- (v1);
    \foreach \u / \v in {v0/v1, v0/v2, v0/v3, v0/v4, v1/v2, v2/v3, v3/v4}
        \draw (\u) -- (\v);
    \foreach \u / \v in {}
        \draw[ultra thick, blue] (\u) -- (\v);
\end{tikzpicture}
\end{minipage}
\\
\end{tabular}
\caption{A vertex of degree $6$ sends charge 3 in these cases.}
\label{fig:proj3_deg6_oneside}
\end{figure*}
Our proof of Lemma \ref{lem:onesidesend} is done by computer check (and indeed it is the same as that given in \cite{proj2024}. The algorithm used to prove Lemma \ref{lem:onesidesend} is described in \cite{proj2024}.

\tikzset{deg5/.style={thick, circle, draw, fill=black, inner sep=1.5pt,}}
\tikzset{deg6/.style={thick, circle, draw, fill=black, inner sep=0pt,}}
\tikzset{deg7/.style={thick, circle, draw, fill=white, inner sep=2pt,}}
\tikzset{deg8/.style={thick, rectangle, draw, fill=white, inner sep=2pt,}}
\tikzset{deg9/.style={thick, regular polygon, regular polygon sides=3, rotate=180, draw, fill=white, inner sep=1pt,}}
\tikzset{deg10/.style={thick, regular polygon, draw, fill=white, inner sep=2pt,}}
\tikzset{->-/.style={decoration={
    markings,
    mark=at position .6 with {\arrow{>}}}, postaction={decorate}}}
\tikzset{->>-/.style={decoration={
    markings, 
    mark=at position .5 with {\arrow{>}};, 
    mark=at position .6 with {\arrow{>}};}, postaction={decorate}}}

\begin{figure}[htbp]
\centering
\begin{minipage}{0.3\linewidth}
\centering
\captionsetup{width=.95\linewidth}
\begin{tikzpicture} [baseline=1.5cm]
    \node [deg9] at (0.5, 0) (a) {};
    \node [below = 0.4 cm of a, anchor=center] (a+) { $+$ };
    \node [deg9] at (1.5, 0) (b) {};
    \node [below = 0.4 cm of b, anchor=center] (b+) { $+$ };
    \node [deg5] at (0.1, 0.9) (c) {};
    \node [deg5] at (1.0, 0.8) (d) {};
    \node [deg5] at (1.9, 0.9) (e) {};
    \node [deg6] at (1.0, 1.6) (f) {};
    \draw [->-] (d) -- (a);
    \draw [->-] (d) -- (b);
    \foreach \u / \v in {a/b, a/c, a/d, b/d, b/e, c/d, c/f, d/e, d/f, e/f}
        \draw (\u) -- (\v);
\end{tikzpicture}
\caption{When $u$ is the vertex with two arrows and $v,w$ are two vertices of degree at least $9$, $\phi(u,v) + \phi(u,w) = 8$.}
\label{fig:twoedge8}
\end{minipage}
\begin{minipage}{0.3\linewidth}
\centering
\captionsetup{width=.95\linewidth}
\begin{tikzpicture} [baseline=1.5cm]
    \node [deg9] at (0.5, 0) (a) {};
    \node [below = 0.4 cm of a, anchor=center] (a+) { $+$ };
    \node [deg9] at (1.5, 0) (b) {};
    \node [below = 0.4 cm of b, anchor=center] (b+) { $+$ };
    \node [deg5] at (0.1, 0.9) (c) {};
    \node [deg5] at (1.0, 0.8) (d) {};
    \node [deg6] at (1.9, 0.9) (e) {};
    \node [deg5] at (1.0, 1.6) (f) {};
    \draw [->-] (d) -- (a);
    \draw [->-] (d) -- (b);
    \foreach \u / \v in {a/b, a/c, a/d, b/d, b/e, c/d, c/f, d/e, d/f, e/f}
        \draw (\u) -- (\v);
\end{tikzpicture}
\caption{When $u$ is the vertex with two arrows and $v,w$ are two vertices of degree at least $9$, $\phi(u,v) + \phi(u,w) = 7$.}
\label{fig:twoedge7}
\end{minipage}
\begin{minipage}{0.3\linewidth}
\centering
\captionsetup{width=.95\linewidth}
\begin{tikzpicture} [baseline=1.5cm]
    \node [deg9] at (1.3, 0) (a) {};
    \node [below = 0.4 cm of a, anchor=center] (a+) { $+$ };
    \node [deg9] at (0.5, 0.6) (b) {};
    \node [below = 0.4 cm of b, anchor=center] (b+) { $+$ };
    \node [deg9] at (2.1, 0.6) (c) {};
    \node [below = 0.4 cm of c, anchor=center] (c+) { $+$ };
    \node [deg5] at (1.3, 0.9) (d) {};
    \node [deg5] at (0.8, 1.5) (e) {};
    \node [deg5] at (1.8, 1.5) (f) {};
    \draw [->-] (d) -- (a);
    \draw [->-] (d) -- (b);
    \draw [->-] (d) -- (c);
    \foreach \u / \v in {a/b, a/c, a/d, b/d, b/e, c/d, c/f, d/e, d/f, e/f}
        \draw (\u) -- (\v);
\end{tikzpicture}
\caption{When $u$ is the vertex with three arrows and $v,w,x$ are three vertices of degree at least $9$, $\phi(u,v) + \phi(u,w) + \phi(u,x) = 10$.}
\label{fig:threeedge10}
\end{minipage}
\end{figure}

The proofs of Lemmas \ref{lem:twoedgesends} and \ref{lem:threeedgesends} below are the same as those given in \cite{proj2024} since Lemma \ref{lem:onesidesend} (especially when the degree of $u$ is $7$ or $8$) is the same as \cite{proj2024}.

\begin{lem}\label{lem:twoedgesends}\showlabel{lem:twoedgesends}
Let $H$ be internally 6-connected triangulation in the plane (or the torus). Assume that none of our reducible configurations in the set $\mathcal{K}$ weakly appears $H$.
Let $vuw$ be a triangle face in $H$ and assume the degree of $v, w$ is at least $9$.
\begin{itemize}
    \item If the degree of $u$ is $5$, then $\phi(u, v) + \phi(u, w) \leq 8$. If $\phi(u, v) + \phi(u, w) = 8(, 7)$, they are in Figure \ref{fig:twoedge8}(, \ref{fig:twoedge7}) respectively. 
    \item If the degree of $u$ is $6, 7$, then $\phi(u, v) + \phi(u, w) \leq 4$.
    \item If the degree of $u$ is at least $8$, then $\phi(u, v) + \phi(u, w) = 0$.
\end{itemize}
\end{lem}

\begin{lem}\label{lem:threeedgesends}\showlabel{lem:threeedgesends}
Let $H$ be an internally 6-connected triangulation in the plane (or the torus). Assume that none of our reducible configurations in the set $\mathcal{K}$ weakly appears in $H$.
Let $u$ be a vertex in $H$ and $v, w, x$ be the neighbors of $u$ so that $v,w$ are adjacent and $w,x$ are adjacent. Assume that the degree of $v, w, x$ is at least $9$.
\begin{itemize}
    \item If the degree of $u$ is $5$, then $\phi(u, v) + \phi(u, w) + \phi(u, x) \leq 10$. If the equality holds, they are in Figure \ref{fig:threeedge10}.
    \item If the degree of $u$ is $6$ (, $7$), then $\phi(u, v) + \phi(u, w) + \phi(u, x) \leq 6$ (, $4$) respectively.
    \item If the degree of $u$ is at least 8, then $\phi(u, v) + \phi(u, w) + \phi(u, x) = 0$.
\end{itemize}
\end{lem}
% \begin{proof}
%     Proof is the same as proj
% \end{proof}

\begin{figure}
\begin{minipage}{0.45\linewidth}
    \centering
    \includegraphics[width=7cm]{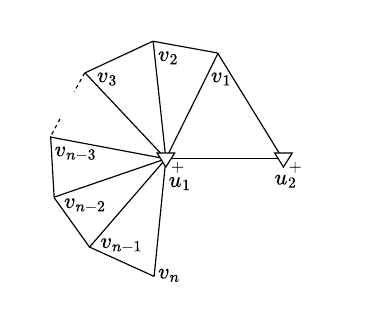}
    \caption{}
    \label{fig:send-fan1}
\end{minipage}
\begin{minipage}{0.45\linewidth}
    \centering
    \includegraphics[width=7cm]{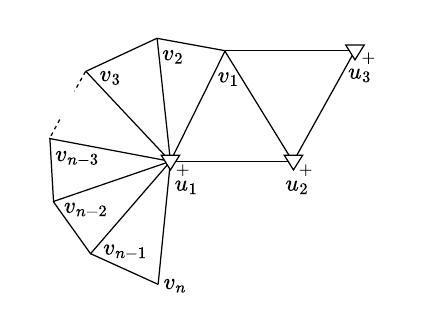}
    \caption{}
    \label{fig:send-fan2}
\end{minipage}
\end{figure}

\begin{lem}\label{lem:sendavg5-2or4}\showlabel{lem:sendavg5-2or4}
    Let $H$ be an internally 6-connected triangulation in the plane (or the torus). Assume that none of our reducible configurations in the set $\mathcal{K}$ weakly appears in $H$.
    \begin{enumerate}
        \item[(a)] Let $u_1, u_2$ be two adjacent vertices of degree at least 9 in $H$, and $v_1, v_2, ..., v_n$ be neighbors of $u_1$ listed in the clockwise order of $H$ such that $v_1u_1u_2$ constitutes a facial triangle. See Figure \ref{fig:send-fan1}. Then, 
         \begin{itemize}
            \item If the degree of $v_1$ is 5, $\phi(v_1, u_2) + \sum_{i=1}^n \phi(v_i, u_1) \leq 5n + 3$. 
            \item If the degree of $v_1$ is at least $6$, $\phi(v_1, u_2) + \sum_{i=1}^n \phi(v_i, u_1) \leq 5n $.
        \end{itemize}

        \item[(b)] Let $u_1, u_2, u_3$ be vertices of degree at least 9 in $H$, and $v_1, v_2, ..., v_n$ be neighbors of $u_1$ listed in the clockwise order of $H$ such that both $v_1u_1u_2, v_1u_2u_3$ are facial triangles. See Figure \ref{fig:send-fan2}. Then, $\phi(v_1, u_2) + \phi(v_1, u_3) + \sum_{i=1}^n \phi(v_i, u_1) \leq 5n + 6$
    \end{enumerate}
\end{lem}

\begin{proof}
    First, we prove (a). 

    By Lemma \ref{lem:twoedgesends}, $\phi(v_1, u_2) + \phi(v_1, u_1) = 8$ in the case described in Figure \ref{fig:twoedge8}. If $\phi(v_2, u_1)$ is at most 4, the claim holds since $\sum_{i=3}^n \phi(v_i, u_1) \leq 5n-9$ by Lemma \ref{lem:sendavg5+1}. Hence, we assume 
    %$\phi(v_2, u_1)$ is at least 5. When 
    $\phi(v_1,u_2) + \phi(v_1, u_1) = 8$ and $\phi(v_2, u_1) \geq 5$.  Then $v_2$ sends charge 5 in either of send5(1,1) or send5(1,2) or  send5(1,3). In every case, C(2), or C(17) weakly appears.
    
    Second, we prove (b). By Lemma \ref{lem:threeedgesends}, $\phi(v_1, u_1) + \phi(v_1, u_2) + \phi(v_1, u_3) \leq 10$. By Lemma \ref{lem:sendavg5+1} $\sum_{i=2}^n \phi(v_i, u_1) \leq 5n-4$. The statement holds by combining these two formulas.
\end{proof}

% ====================================================
\section{Looking at a disk bounded by a short cycle}
\label{sect:smaller}
\showlabel{sect:smaller}
% ====================================================

In this section, we show Lemma \ref{lem:conf-in-T}, which states that a reducible configuration appears strictly inside the disk bounded by a cycle with some specified conditions. 
This lemma is necessary to handle 6,7-edge-cuts in Section \ref{sect:6,7-cut}. The same statement of this lemma is proved in \cite{proj2024}. The configuration set and rule set used in \cite{proj2024} are different from our sets $\mathcal{K}, \mathcal{R}$, but we show that the same statement holds for our setting too. The key to proving this lemma is Lemma \ref{lem:conf-in-T-general}. In \cite{proj2024}, Lemma \ref{lem:conf-in-T-general} implied Lemma \ref{lem:conf-in-T} by combining discharging method to ensure the appearance of configuration if a positively charged vertex exists. 
Hence, to prove Lemma \ref{lem:conf-in-T} in our setting  $\mathcal{K}, \mathcal{R}$, we only have to check the hypothesis of Lemma \ref{lem:conf-in-T-general}, especially 3.1 and 3.2 below. 
In fact, we already checked these two statements in Lemma \ref{lem:sendavg5-2or4}, so Lemma \ref{lem:conf-in-T} holds for our sets $\mathcal{K}, \mathcal{R}$ too. 

It turns out that Lemma \ref{lem:conf-in-T} is quite useful for other cases, including the planar, the doublecross, the apex case, and the projective planar cases.

% configuration appear in planar-neartriangulation
\begin{lem}\label{lem:conf-in-T}\showlabel{lem:conf-in-T}
Let $T$ be an internally 6-connected near-triangulation in the plane, which is a subgraph of $G'$. Let $C$ be an induced cycle in $T$, which bounds the infinite region of $T$. Assume that there is no vertex in $V(T -C)$ which is adjacent to at least four consecutive vertices of $C$ and there are more than $\frac{18}{5} \cdot |C| - 12$ edges in $T$ between $C$ and $T - C$. Then a reducible configuration in the set $\mathcal{K}$ appears in $T$.
\end{lem}

\begin{lem}\label{lem:conf-in-T-general}\showlabel{lem:conf-in-T-general}
Let $T$ be an internally 6-connected near-triangulation in the plane and let $C$ be an induced cycle in $T$, which bounds the infinite region of $T$. Let $\mathcal{R}, \mathcal{K}$ be a set of rules and a set of reducible configurations respectively. Suppose that no configuration in $\mathcal{K}$ appears in $T$.

Assume that 
\begin{enumerate}
    \item 
    there is no vertex in $V(T-C)$ which is adjacent to at least four consecutive vertices of $C$, 
    \item 
    there are more than $\frac{18}{5} \cdot |C| - 12$ edges in $T$ between $C$ and $T - C$, 
    \item
    by applying rules in $\mathcal{R}$, the total amount of charges moved from $V(T-C)$ to $V(C)$ are at most $5k-2|C|$, where $k$ is the number of edges between $C$ and $T-C$. \\

    \medskip
    
    3 can be achieved by showing the following statements:
    \begin{itemize}
        \item[3.1] let $u_1, u_2$ be adjacent vertices of degree at least 9, and $v_1, v_2, ..., v_n$ be neighbors of $u_1$ listed in clockwise order of $T$ such that $v_1u_1u_2$ constitutes a facial triangle. See Figure \ref{fig:send-fan1}. Then, the amount of charges sent from $v_1,v_2,...,v_n$ to $u_1, u_2$ by $\mathcal{R}$ is at most $5n+3$,
        % $\phi(v_1, u_2) + \sum_{i=1}^n \phi(v_i, u_1) \leq 5n + 3$ holds.
        \item[3.2] let $u_1, u_2, u_3$ be vertices of degree at least 9, and $v_1, v_2, ..., v_n$ be neighbors of $u_1$ listed in clockwise order of $T$ such that both $v_1u_1u_2, v_1u_2u_3$ are facial triangles. See Figure \ref{fig:send-fan2}. Then, the amount of charges sent from $v_1,v_2,...,v_n$ to $u_1, u_2, u_3$ by $\mathcal{R}$ is at most $5n+6$. 
    \end{itemize}
\end{enumerate}
 Then, there is a vertex in $V(T-C)$ that ends up with a final charge positive. 
\end{lem}

In this section, we extend Lemma \ref{lem:conf-in-T} to the special triangulation.

Let $T$ be a triangulation in the plane, and let $e=ab, e'=a'b'$ to be two edges in $T$, and assume that we add two multiple edges $e_1=ab, e'_1=a'b'$ (thus there are now exactly $3n-4$ edges). We say that $T$ is \emph{nearly internally 6-connected}, if 
\begin{itemize}
    \item for any contractible cycle $C$ of order at most five such that one open disk $\Delta$ bounded by $C$ contains none of $a, b, a', b'$, $\Delta$ contains no vertex or $\Delta$ consists of one vertex and $|V(C)|=5$, and
    \item there is no contractible cycle $C$ of order at most three such that both sides divided by $C$ contain at least one of $a, b, a', b'$. 
\end{itemize}

The vertices $a, b, a', b'$ in $T$ are incident to multiple edges. The degree of a vertex $v$ of $T$ is calculated by the number of edges that are incident to $v$, not by the number of adjacent vertices.
% configuration appear in rep=1
\begin{lem}\label{lem:conf-in-rep1}\showlabel{lem:conf-in-rep1}
Let $T$ be a nearly internally 6-connected triangulation. 
Let $c, d$($c', d'$, resp.) be the vertices of $T$ such that the facial triangles $abc$, $abd$($a'b'c'$, $a'b'd'$, resp.) exist respectively.(since $T$ is nearly internally 6-connected, the number of such vertices for each edge $ab$, $a'b'$ is exactly two).
Let $k$ be the number of vertices in $v \in \{c, d, c', d'\}$ such that the degree of $v$ is more than 5.
If one of the following holds, there is a reducible configuration in the set $\mathcal{K}$ that weakly appears in $T$ and that contains neither $e, e_1$ nor $e', e'_1$.
\begin{itemize}
    \item $k=4$, or
    \item $k \leq 3$ and $d(a) + d(b) + d(a') + d(b') > \lfloor (112-3k)/5 \rfloor$. 
\end{itemize}
\end{lem}
Note that the nearly internally 6-connectivity of $T$ implies that all degrees of $T$ including $a, b, a', b'$ are at least 5. We prove Lemma \ref{lem:conf-in-rep1} by using the following claims.

% calculate initial charges
\begin{claim}\label{clm:initial-charge-in-T}\showlabel{lm:initial-charge-in-T}
    \[
        \sum_{v \in V(T) \setminus \{a, b, a', b'\}} 10 \cdot (6 - d(v)) = 10 \cdot (d(a) + d(b) + d(a') + d(b')) - 160. 
    \]
\end{claim}
\begin{proof}
    $|V(T)|-|E(T)|+|F(T)|=2$ holds by Euler formula. By counting the number of edges, $2|E(T)| = 3(|F(T)| - 2) + 2 \cdot 2$ holds. The combination of the two formulas imply $|E(T)|=3|V(T)|-4$. $\sum_{v \in V(T)} 10 \cdot (6 - d(v)) = 60|V(T)| - 10 \sum_{v \in V(T)} d(v) = 60 |V(T)| - 20 |E(T)| = 80$. $\sum_{v \in \{a, b, a', b'\}} 10 \cdot (6 - d(v)) = 240 - 10 \cdot (d(a) + d(b) + d(a') + d(b'))$. Hence, this claim holds.
\end{proof}

\begin{figure}
    \centering
    \includegraphics[width=6.5cm]{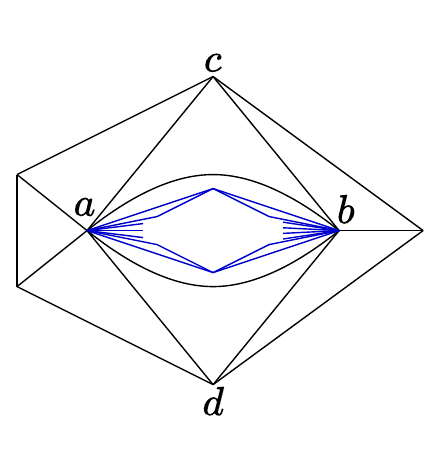}
    \caption{The blue edges are added to increase the degrees of $a, b$. The detailed structure inside the disk bounded by multiple edges is omitted but can be triangulated.}
    \label{fig:multiedge-deg13}
\end{figure}

After assigning a charge to each vertex, we increase the degree of $a, b$, ($a', b'$, resp.) to at least 13 by adding edges inside the disk bounded by multiple edges $ab$, ($a'b'$, resp.) like the Figure \ref{fig:multiedge-deg13}. We triangulate faces by adding edges appropriately. The number 13 represents the least degree that is not assigned to any vertex in the configurations in  $\mathcal{K}$. Even after increasing the degree, we use the symbol $d(\cdot)$ to represent the degree in the original graph $T$. In the following claim, we apply rules in the graph after increasing degrees of $a, b, a', b'$.

% calculate charge move
\begin{claim}\label{clm:move-charge-from-T}\showlabel{clm:move-charge-from-T}
    The total amount of charges sent from $T - \{a, b, a', b'\}$ to $\{a, b, a', b'\}$ by rules in $\mathcal{R}$ is at most $5 \cdot (d(a) + d(b) + d(a') + d(b')) - 48 - 3k$.
\end{claim}
\begin{proof}
    By Lemma \ref{lem:sendavg5-2or4}, $\sum_{v \in \{a, b, a', b'\}, d(v) = 5} \{ 5(d(v)-3)+3 \} + \sum_{v \in \{a, b, a', b'\}, d(v) > 5} \{ 5(d(v)-3) \} = 5 \cdot (d(a) + d(b) + d(a') + d(b')) - 48 - 3k$
\end{proof}

% calculate the final charge
\begin{claim}\label{clm:final-charge-in-T}\showlabel{clm:final-charge-in-T}
    The total amount of final charges accumulated over all vertices $v \in V(T)$, except $\{a, b, a', b'\}$, is at least $5 \cdot (d(a) + d(b) + d(a') + d(b')) - 112 + 3k$.
\end{claim}
\begin{proof}
    By Claims \ref{clm:initial-charge-in-T} and \ref{clm:move-charge-from-T}, $\{ 10 \cdot (d(a) + d(b) + d(a') + d(b')) - 160\} - \{ 5 \cdot (d(a) + d(b) + d(a') + d(b')) - 48 - 3k \} = 5 \cdot (d(a) + d(b) + d(a') + d(b')) - 112 + 3k$.
\end{proof}
\begin{proof}[Proof of Lemma \ref{lem:conf-in-rep1}]
    By Claim \ref{clm:final-charge-in-T}, if $k \leq 3$ and $d(a) + d(b) + d(a') + d(b') > \lfloor (112 - 3k)/5 \rfloor $, the total amount of final charges accumulated over all vertices $V(T)$, except $a,b,a',b'$, is positive. 
    
    If $k=4$, the total amount of final charges accumulated over all vertices $V(T)$, except $a, b, a', b'$, is at least zero but can be exactly zero. This happens only when $d(a)+d(b)+d(a')+d(b')=20$, so degree of $a, b, a', b'$ is all 5. Let $e$ denote the neighbor of $a$ other than $b, c, d$. The amount $\phi(c,a) + \phi(c,b) + \phi(e, a)$ must be 10 by claim \ref{clm:move-charge-from-T}. Since the degree of $c$ is at least $6$, this is possible only when $\phi(c, a) + \phi(c, b) = 4$ and $\phi(e, a) = 6$ by Lemma \ref{lem:twoedgesends}. However, either of the degree of $c, d$ is 5 when $\phi(e, a) = 6$ by Lemma \ref{lem:vsends} (Figure \ref{fig:send6}). This contradicts the formula $k=4$. Hence, the total amount of final charges accumulated over all vertices $V(T)$, except $a,b,a',b'$, does not become exactly zero.

    Therefore, in both cases, at least one vertex of $V(T) \setminus \{a, b, a', b'\}$ has the final positive charge. Before moving charges, the degrees of $a, b, a', b'$ are increased to at least 13, so no configuration of $\mathcal{K}$ contains a vertex in $a, b, a', b'$. Hence, a reducible configuration in the set $\mathcal{K}$ weakly appears in $T$ and it contains neither $e$ nor $e'$.
\end{proof}

\subsection{When representativity $=1$}\label{sect:rep=1}\showlabel{sect:rep=1}

In this subsection, we assume that the representativity of the embedding of $G$ in the torus is exactly one. The purpose of this subsection is to apply Lemma \ref{lem:conf-in-rep1} to avoid a loop in the discharging and find a reducible configuration that does not ``wrap around" a non-contractible loop. 
%When the representativity of $G$ is exactly one, 
Let $e=uv$ be an edge of $G$ such that $G-e$ is planar. Such an edge $e$ is unique since two different such edges constitute a 2-edge cut that contradicts Lemma \ref{minc}. We embed $G-e$ in the plane so that the local rotation of incident edges of each vertex is the same as that of the original embedding of $G$ in the torus. Since $G-e$ is planar, such an embedding exists. Let $T$ be the dual graph of $G-e$ in the plane. The degree of vertices $u, v$ of $G-e$ is 2, and the degrees of other vertices are 3, so $T$ has two multiple edges, and all faces except two faces bounded by two multiple edges are triangles. Also, $T$ is nearly internally 6-connected by Lemmas \ref{minc} and \ref{lem:(2,2)-annulus-cut}, \ref{lem:(1,3)-annulus-cut}. 
By deleting $e$ from $G$, the toroidal embedding of $G-e$ has one face that is homomorphic to the annulus. The face is bounded by two paths. Each of the two paths becomes a facial cycle in the embedding of $G-e$ in the plane. Let $a, a'$ be the two vertices of $T$ that correspond to these two facial cycles. The vertices $a, a'$ incident with multiple edges. Let $b, b'$ be the two vertices of $T$ such that each of $ab$, $a'b'$ is a multiple edge. We use the symbols $c,d,c',d'$ to denote the same vertices of Lemma \ref{lem:conf-in-rep1}.

\begin{figure}
    \centering
    \begin{tabular}{cc}
        \includegraphics[width=4cm]{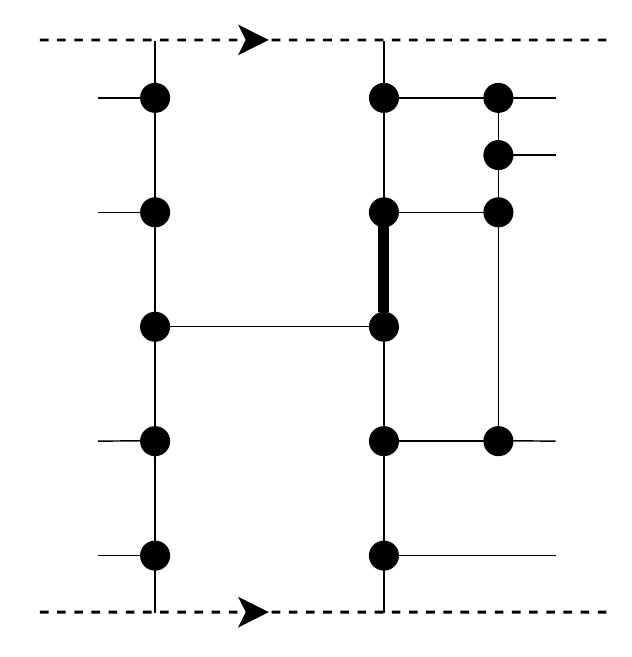} &
        \includegraphics[width=4cm]{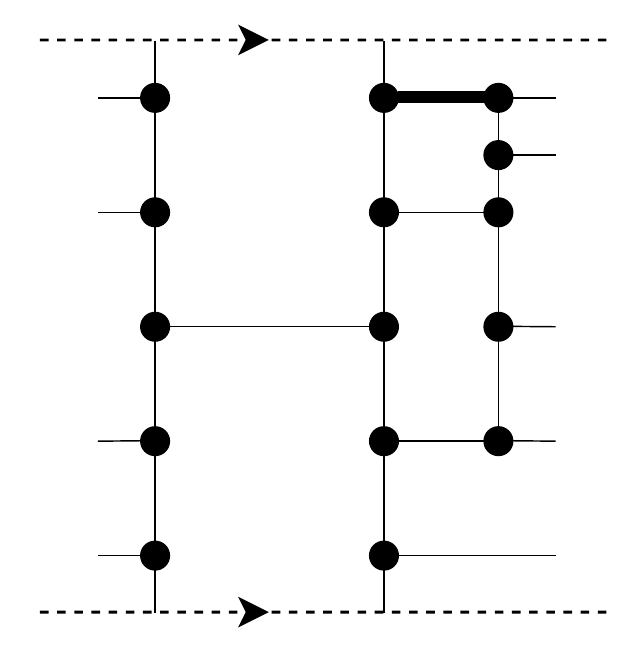}
    \end{tabular}
    \vspace{0.5cm}
    \begin{tabular}{ccc}
        \includegraphics[width=4cm]{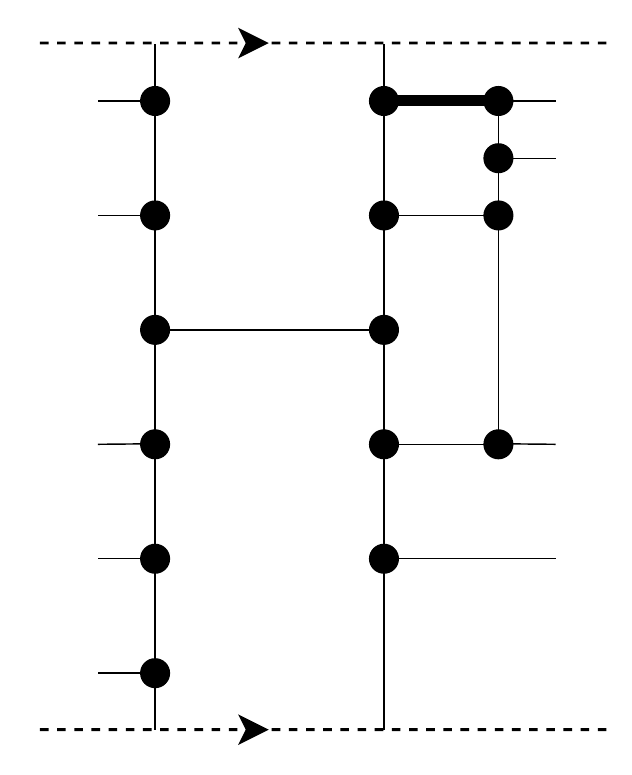} &
        \includegraphics[width=4cm]{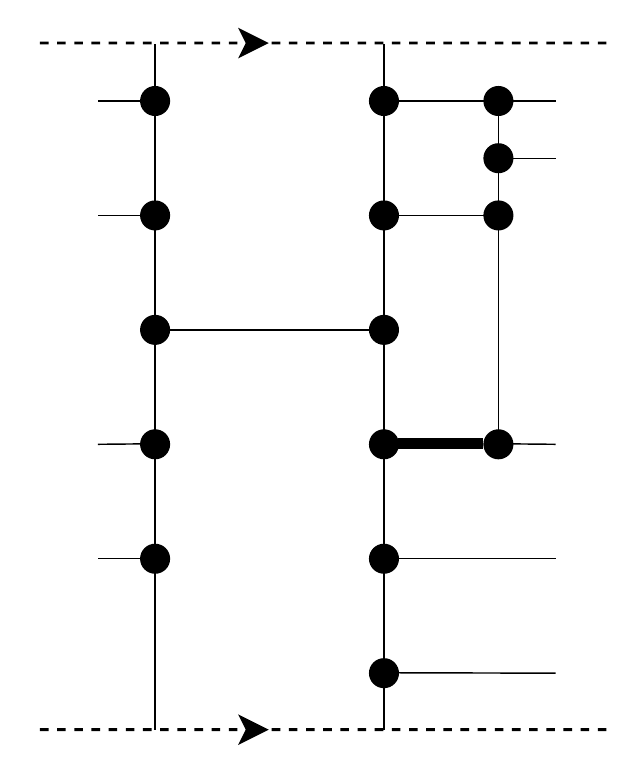} &
        \includegraphics[width=4cm]{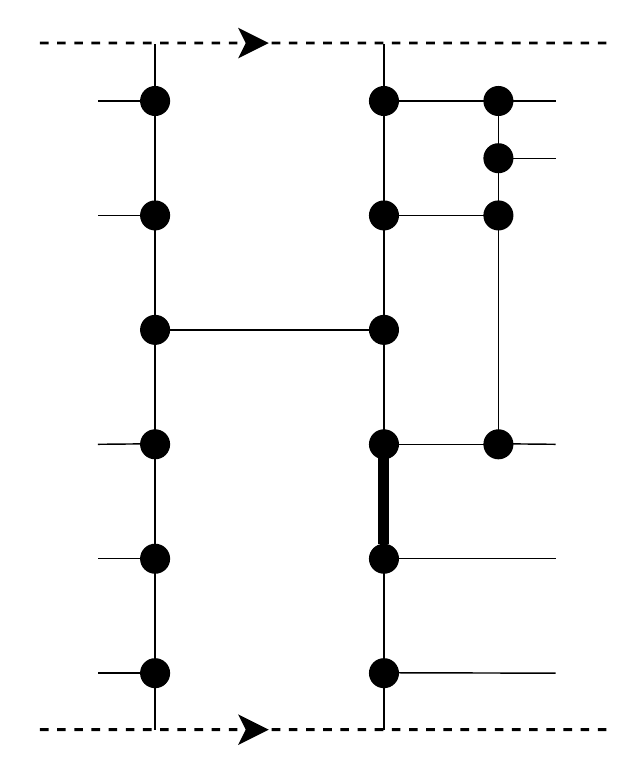}
    \end{tabular}
    \caption{The reducible annular islands used in the proof of Lemma \ref{lem:rep1}.}
    \label{fig:rep1-annular-islands}
\end{figure}

To prove Lemma \ref{lem:rep1}, we use the fact that all the annular islands described in Figure \ref{fig:rep1-annular-islands} are reducible. The bold edges depicted in the figure represent the edges to be deleted to become C-reducible. The size of edges deleted for all islands is exactly one, so we do not have to worry about the resulting graph (after deletion) having a bridge or becoming Petersen-like. We denote the annular islands in Figure \ref{fig:rep1-annular-islands} by A($i$) ($1 \leq i \leq 5$) in the natural order.
\begin{lem}\label{lem:rep1}\showlabel{lem:rep1}
    If the representativity of the embedding of $G$ in the torus is exactly one, either one of the annular islands described in Figure \ref{fig:rep1-annular-islands} appears, or a reducible configuration $K$ in the set $\mathcal{K}$ weakly appears in $G'$ (which is the dual of $G$) in such a way that $K$ does not contain a non-contractible loop in $G'$.
\end{lem}
\begin{proof}
    We divide the case analysis depending on $d_T(a), d_T(a')$. By Lemma \ref{lem:(1,3)-annulus-cut}, $d_T(a), d_T(a') \geq 5$. By Lemma \ref{minc}, $d_T(b), d_T(b') \geq 5$, so if $d_T(a) + d_T(a') > 12$, $d_T(a)+d_T(a')+d_T(b)+d_T(b')>22$, so a reducible configuration appears by Lemma \ref{lem:conf-in-rep1}. We handle the remaining cases.
   
    \begin{itemize}
        \item[(1)] $d_T(a) = 5, d_T(a') = 5$
    \end{itemize}
    If $d_T(b) + d_T(b') > 12$, $d_T(a) + d_T(a') + d_T(b) + d_T(b') > 22$, so we can apply Lemma \ref{lem:conf-in-rep1}.
    If $d_T(b) = 5$ and $d_T(b') = 7$, the degrees of $c, d$ are more than 5 since A(1) is reducible. Hence, we can apply Lemma \ref{lem:conf-in-rep1} for $k=2$.
    If $d_T(b) \leq 6 $ and $ d_T(b') \leq 6$, the degrees of $c, d, c', d'$ are more than 5 since A(1), A(2) are reducible. Hence, we can apply Lemma \ref{lem:conf-in-rep1} for $k=4$.
    
    \begin{itemize}
        \item[(2)] $d_T(a) = 5, d_T(a') = 6$ (or $d_T(a) = 6, d_T(a') = 5$)
    \end{itemize}
    If $d_T(b) + d_T(b') > 11$, $d_T(a) + d_T(a') + d_T(b) + d_T(b') > 22$, so we can apply Lemma \ref{lem:conf-in-rep1}.
    If $d_T(b) = 5, d_T(b') = 6$, the degrees of $c, d$ are more than 5 since A(3), A(4) are reducible. Hence we can apply Lemma \ref{lem:conf-in-rep1} for $k \geq 2$. If $d_T(b) = 6, d_T(b') = 5$, we can prove it in the same way.
    If $d_T(b) = 5, d_T(b') = 5$, the degrees of $c, d, c', d'$ are more than 5 since A(3), A(4) are reducible. Hence we can apply Lemma \ref{lem:conf-in-rep1} for $k=4$.

    \begin{itemize}
        \item[(3)] $d_T(a) = 6, d_T(a') = 6$
    \end{itemize}
    If $d_T(b) + d_T(b') > 10$, $d_T(a) + d_T(a') + d_T(b) + d_T(b') > 22$, so we can apply Lemma \ref{lem:conf-in-rep1}. If $d_T(b) = 5$ and $d_T(b') = 5$, all degrees of $c, c', d, d'$ are more than 5, since A(5) is reducible. Hence we can apply Lemma \ref{lem:conf-in-rep1} for $k=4$.
\end{proof}

% ====================================================
\section{Cuts of size 6 or 7}
\label{sect:6,7-cut}\showlabel{sect:6,7-cut}
% ====================================================

Let us remind the reader that $G$ is a minimal counterexample and $G'$ is the dual of the embedding of $G$ in torus. In this section, we will show that one side of $G$ divided by a 6-edge-cut or a 7-edge-cut is relatively small. By Lemma \ref{minc}, we know that one side of $G$ divided by a cut of order at most five is small.
Let us observe that edges in a cut of $G$ correspond to a circuit\footnote{A circuit is a closed walk, allowing some vertex to be used twice or more. A cycle is a closed walk but does not allow each vertex to be used twice or more. We use ''circuit'' here to handle the case that a subwalk of $C$ may be a noncontractible cycle.} $C$ in $G'$. Throughout this section, we shall consider $G'$. The goal of this section is to prove Lemma \ref{lem:6,7-cut} that roughly says that the disk bounded by a circuit of length six or seven contains at most three or four vertices strictly inside. 

The proof of this lemma needs a lot of case analysis, unfortunately. 
The problem is that when we find a C-reducible configuration $K$ inside the circuit $C$ in $G'$, after contractions in $K$, it is still possible that the resulting graph would end up with the dual of a Petersen-like graph after 2,3-cycle reductions\footnote{2,3-cycle reductions corresponds to 2,3-cut reductions in dual}. There are a lot of cases we have to handle, so we do need some computer checks, see Section \ref{sect:cut6,7-Petersen-like}. 

To this end, we first handle C(1), which is a special case of Lemma \ref{lem:6,7-cut}, because the ring of C(1) is of length exactly six, and C(1) consists of four vertices (which would violate Lemma \ref{lem:6,7-cut} below). The configuration C(1) has two different contractions of size exactly six for C-reducibility as shown in Figure \ref{fig:5555-contractions}. We prove that $G'$ would not result in the dual graph of any Petersen-like graph after one of the contractions in Lemma \ref{lem:5555}.
The proof of the lemma is deferred to Appendix \ref{subsect:5555}.
We also handle C(27) separately. C(27) is used only when all degrees are 6 in Section \ref{sect:flatproof}. Hence, we do not need to care about it here.

Let $C$ be a separating circuit of length $l=6,7$. If $C$ bounds no disk (i.e. $C$ is not a cycle but $C$ consists of several contractible cycles), it is of no interest since one component of $G'-C$ has at most one vertex by Lemma \ref{minc}.
A separating circuit $C$ of length $l=6,7$ in $G'$ is \emph{balanced} if there is a closed disk $D$ that is bounded by $C$ and both of $D-C$ and $G'-D$ contains at least $l-2$ vertices.

% C(1)
\begin{figure}[htbp]
    \centering
    \includegraphics[width=12cm]{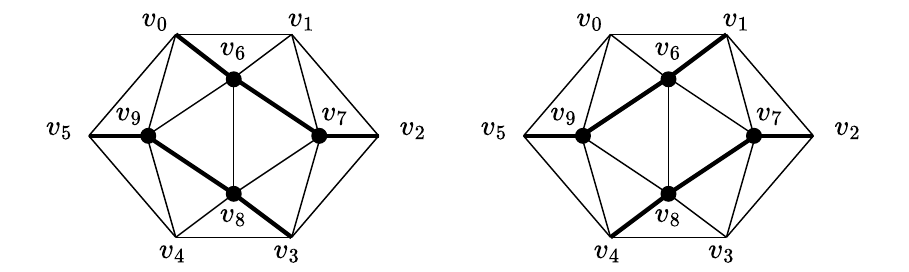}
    \caption{A free completion of $K$ with the ring, where $K$ is C(1). The bold lines represent the contraction edges used for C-reducibility.}
    \label{fig:5555-contractions}
\end{figure}

\begin{lem}\label{lem:5555}\showlabel{lem:5555}
    If C(1) appears in $G'$, then after some edge contraction and low-cycle reductions, $G'$ results in a graph in $\mathcal{T}_1 \setminus \mathcal{T}_0$.
\end{lem}

\begin{lem}
\label{lem:6,7-cut}
\showlabel{lem:6,7-cut}
    \begin{itemize}
        \item If there is a balanced separating circuit $C'$ in $G'$, there is a balanced separating circuit $C$ in $G'$ such that a configuration $K$ of $\mathcal{K}$ appears in a disk bounded by $C$ and the dual graph obtained from $G'$ by contracting edges of $c(K)$ is in $\mathcal{T}_1 \setminus \mathcal{T}_0$.
        \item If there is a separating circuit $C$ of length $l \in \{6,7\}$ in $G'$ that bounds a disk $D'$ where $D'-C$ contains at most $l-3$ vertices. Then a disk $D$ bounded by $C$  is one of the graphs in Figure \ref{fig:6cut} $(l=6)$ or Figure \ref{fig:7cut} $(l=7)$.
    \end{itemize}
\end{lem}

\begin{figure}[htbp]
  \centering
  \includegraphics[width=0.5\hsize]{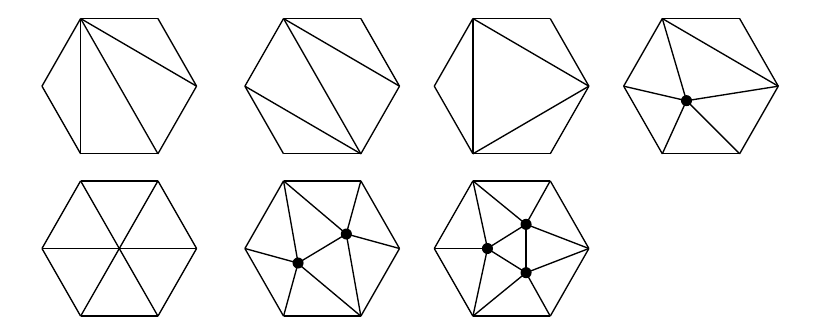}
  \caption{The disks bounded by a separating circuit of length six.}
  \label{fig:6cut}
\end{figure}
\begin{figure}[htbp]
  \centering
  \includegraphics[width=0.6\hsize]{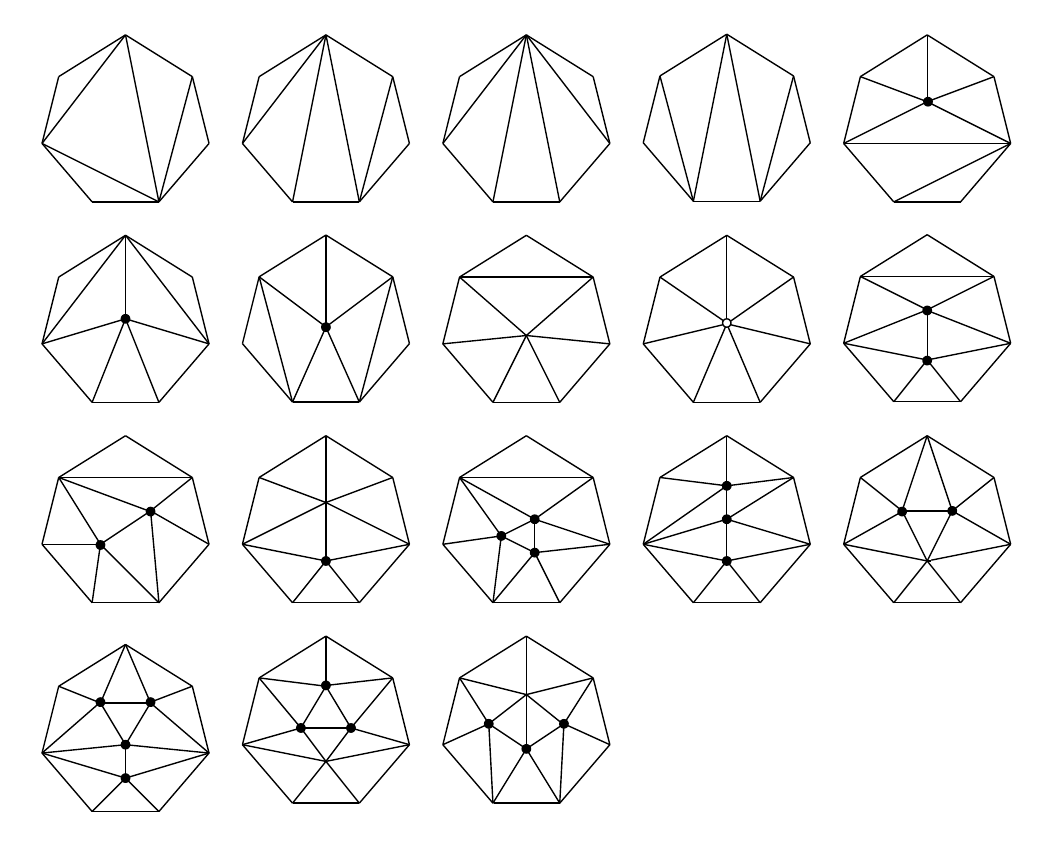}
  \caption{The disks bounded by a separating circuit of length seven.}
  \label{fig:7cut}
\end{figure}

The first claim states that $G'$ has no balanced circuit.
The second claim states that we have the precise characterization of a closed disk $D$ bounded by $C$ if $D-C$ has at most $l-3$ vertices. 
The second claim is easy to verify by enumerating the structure of graphs of at most $l-3$ vertices inside the disk.
Our proof of the first claim of Lemma \ref{lem:6,7-cut} depends on the computer-check described in Section \ref{sect:cut6,7-bridge}, \ref{sect:cut6,7-Petersen-like}.

To prove the first claim of Lemma \ref{lem:6,7-cut}, we choose $C$ as a separating circuit of length $l \in \{6,7\}$ with the following condition. Let $D$ denote the closed disk bounded by $C$.
\begin{enumerate}
    \item[(1)] $C$ is a balanced separating circuit.
    \item[(2)] $D-C$ contains the minimum number of vertices under the condition (1) (If there are two or more minimizers, we preferentially choose one with $l=6$.).
    \item[(3)] If a chosen circuit $c_0c_1c_2c_3c_4c_5c_6$ by the condition (1), (2) has a chord $c_0c_2$ outside the disk bounded by the circuit and a facial triangle $c_0c_1c_2$ exists, then we choose a circuit $c_0c_2c_3c_4c_5c_6$ as $C$.
    \item[(4)] If a chosen circuit $c_0c_1c_2c_3c_4c_5c_6$ by the condition (1), (2), (3) has a path $c_0vc_3$ such that $v$ is outside the disk bounded by the circuit and facial triangles $c_0c_1v, c_1c_2v, c_2c_3v$ exist, then we choose a circuit $c_0vc_3c_4c_5c_6$ as $C$.
\end{enumerate}
Obviously, a chosen circuit by conditions (1), (2) is balanced.
Also, if a circuit chosen by conditions (1), (2) is balanced, the same applies for a circuit chosen by adding conditions (3), (4). 
In the following, we show $C$ is the desired balanced separating circuit of the first claim of Lemma \ref{lem:6,7-cut}.

By the condition (2) of $C$, a circuit $C'(\neq C)$ in $D$ of length $k \in \{6, 7\}$ that bounds the disk that contains at least $k-2$ vertices strictly inside contradicts the definition of $C$ in most cases. An exceptional case is a circuit of length $k=7$ that shares five or four consecutive edges with $C$ and $l=6$. It does not contradict the definition of $C$ due to the conditions (3), (4) and because this circuit may not be balanced. We explain why this circuit may not be balanced. 
If $l=6$, then exactly four vertices may be in $G'-D$, but there must be five vertices outside a disk bounded by a balanced circuit of length seven, so a circuit of length $k=7$ that shares five edges with $C$ may not be balanced.
If $l=6$ and the number of vertices in $C$ is at most five (a vertex appears twice in $C$), then $C$ is not a cycle, and there exists a subwalk of $C$ that is a noncontractible cycle. 
If the number of vertices in $C$ is five, $C$ passes through one vertex twice, so a circuit of length $k=7$ that shares four edges with $C$ may not be balanced.
If the number of vertices in $C$ is at most four, the region obtained by deleting $D$ consists of (i) one disk, or (ii) two or three regions (each region is a disk or an annulus). In the first case, if $G'-D$ has at most four vertices, $D-C$ has also at most four vertices by the condition (2), so $G'$ has at most twelve vertices. It is not a minimal counterexample by the result of \cite{brinkmann2013generation}. In the second case, it contradicts $G'-D$ has at least four vertices by Lemma \ref{minc}, \ref{lem:(2,2)-annulus-cut}, \ref{lem:(1,3)-annulus-cut}. \footnote{A (1,4)-annulus-cut is possible but edges of $C$ do not constitute it when $l=6$.}. Hence, a circuit of length $k=7$ that shares at most three edges with $C$ can contradict the minimality of $C$.

We assume that a path $P$ that connects two vertices $u_i, u_j (i \neq j)$ of $C$ outside $D$ exists. The symbol $E$ represents $P + u_iu_{i+1} + u_{i+1}u_{i+2} + ... + u_{j-1}u_j$ when $i < j$, otherwise $P + u_iu_{i+1} + ... + u_{l-1}u_0 + ... + u_{j-1}u_j$. If $E$ bounds a disk that is disjoint from $D$, $P$ is called \textit{a contractibly connected $u_iu_j$-path of $C$}. Note that a contractible connected $u_0u_2$-path and $u_2u_0$-path are different since $E$ is made from $u_0u_1+u_1u_2$ in the former case, while $E$ is made from $u_2u_3+u_3u_4+...+u_{l-1}u_0$ in the latter case.
We call a contractibly connected $u_iu_j$-path of $C$ whose length is 1 \textit{a contractibly connected $u_iu_j$-chord of $C$}.

The condition (3) seems to be tricky, but the condition (3) 
is important for the following claim.
\begin{claim}\label{clm:no-chord-C-outiside-D}
    A contractibly connected chord of $C$ does not exist.
\end{claim}
\begin{proof}
    When $l=7$, $u_iu_{i+1}u_{i+2}$ is not a facial triangle by the condition (3) of $C$, so a contractibly connected $u_iu_{i+2}$-chord does not exist. If other kinds of contractibly connected chords exist, it would constitute a cycle of length $k \in \{2,3,4,5\}$ with $C$ that contradicts Lemma \ref{minc}.
\end{proof}
The condition (4) is also useful for the following claim.
\begin{claim}\label{clm:no-two-path-C-outside-D}
    If a contractibly connected $u_iu_j$-path of length 2 of $C$ exists, $j - i \leq 2$ or $(j + l) - i \leq 2$ holds.
\end{claim}
\begin{proof}
    Suppose $l=7$, and a contractibly connected path $u_ivu_{i+3}$ exists. Then, a contractibly connected chord of $C$ outside $D$ exists or there exists a vertex $v'$ outside $D$ ($v'$ can be the same as $v$) such that a path $u_iv'u_{i+3}$ exists and facial triangles $c_0c_1v', c_1c_2v', c_2c_3v'$ exist by Lemma \ref{minc}. That contradicts Claim \ref{clm:no-chord-C-outiside-D} or the condition (4) of $C$. If other kinds of contractibly connected path of length $2$ that contradict this claim exist, it contradicts Lemma \ref{minc}.
\end{proof}
Claims \ref{clm:no-chord-C-outiside-D}, \ref{clm:no-two-path-C-outside-D} are important in enumerating all possibilities of 2,3-cycles that contain vertices outside $D$ after contraction.

To use Lemma \ref{lem:conf-in-T}, we have to handle a few cases when $C$ is in a particular situation: when $C$ has a chord in $D$, or a vertex strictly inside $D$ is adjacent to at least four consecutive vertices of $C$, or at most $\frac{18}{5}l-12$ edges exist between $C$ and vertices in $D - C$. In these cases, we find a circuit of length at most six in $D$, which enables us to get a characterization of $D$ by Lemma \ref{minc} or the almost-minimality of $C$ or find a reducible configuration strictly inside $D$. The characterization found here is enumerated in Figure \ref{fig:6cut} ($l=6$), \ref{fig:7cut} ($l=7$), which viloates $D-C$ has at least $l-2$ vertices ($C$ is balanced). 
Therefore, we find a reducible configuration strictly inside $D$ by Lemma \ref{lem:conf-in-T}. In the rest of the section, we show that the dual of the resulting graph belongs to $\mathcal{T}_1 \setminus \mathcal{T}_0$ after contracting $c(K)$ of a C-reducible configuration $K$ and 2,3-cycle reductions. More precisely, we have to show that the dual of the resulting graph has no bridge and is not a Petersen-like graph.

% bridge
\subsection{Bridge}\label{sect:cut6,7-bridge}\showlabel{sect:cut6,7-bridge}
A bridge corresponds to a loop in the dual graph, so we care about a loop in $G'$ after contraction. There are two types of loops in the toroidal graphs: contractible, noncontractible. 
We handle a noncontractible loop and one particular type of contractible loop in Section 
\ref{subsubsect:noncontractible-loop}, \ref{subsubsect:particular-contractible-loop}
(The checks in those sections are independent of the result of this section). In this section, we handle other kinds of contractible loops.
By computer check, we show the following claim.
\begin{claim}\label{clm:loop-6,7cycle}
For a C-reducible configuration $K \in \mathcal{K}$, assume that $K$ appears in $G'$, and a contractibly connected chord of the ring of $K$ becomes a contractible loop after contracting edges of $c(K)$. Then, a contractible cycle of length $k \in \{6,7\}$ with the following properties exists.
\begin{itemize}
    \item The cycle contains an edge in $G'$ that becomes a loop after contraction, and 
    \item the cycle separates $G'$ into two components, each of which has at least $k-2$ vertices.
    \item If $k=7$, at least four edges of the cycle are incident with a vertex of $K$. 
\end{itemize}
\end{claim}
The third condition is necessary to handle the almost-minimality of $C$ as explained above.
When a cycle described in Claim \ref{clm:loop-6,7cycle} is inside $D$, that contradicts the definition of $C$. Otherwise, a chord of $C$ outside $D$ exists since the edges contracted are strictly inside $D$. That contradicts Claim \ref{clm:no-chord-C-outiside-D}.
In this discussion, we show that a contractible loop does not exist after contraction, so the following claim holds.
\begin{claim}\label{clm:not-contractible-bridge-cut6,7}
    For every C-reducible configuration $K \in \mathcal{K}$ that appears strictly inside $D$, after contracting c($K$) in $G'$, no contractible loops occur.
\end{claim}
By the same argument of Claim \ref{clm:not-contractible-bridge-cut6,7}, we can prove Lemma \ref{pairs6} in Section \ref{sect:reducible configurations}, by  using Lemma \ref{lem:6,7-cut}. 

% Petersen-like graphs
\subsection{Petersen-like graph}\label{sect:cut6,7-Petersen-like}\showlabel{sect:cut6,7-Petersen-like}
The graph obtained by a 2,3-edge-cut reduction that is caused by (2,1),(1,1)-annulus-cut never becomes a Petersen-like graph since such a cut ensures that the representativity of this graph is at most one. Hence, we only consider 2,3-edge-cut reductions that correspond to a contractible cycle of length 2,3 in the dual (2,3-cycle reductions).
Also, the representativity of $G, G'$ is at least two since a graph of the representativity one does not become Petersen-like.
We denote $H$ as the graph obtained from $G'$ by deleting vertices outside $D$ and splitting vertices of $C$ if identified (we need to split vertices if $C$ is not a cycle as explained at the beginning of this proof). The graph $H$ is a near triangulation in the plane such that $C$ bounds the outer face boundary. 
Remind that $K$ is a configuration inside $D$ (also in $H$), and $c(K)$ is a set of edges to be contracted for C-reducibility.

\subsubsection{No 2,3-cycle reductions outside $D$}\label{sect:no-2,3-cycle-Petersen-like}\showlabel{sect:no-2,3-cycle-Petersen-like}
First, we consider the case that no two different vertices of $C$ are identified and no chord of $C$ exists in $H$ after contraction. In this case, no 2,3-cycle that contains a vertex outside $D$ appears after contractions since two non-consecutive vertices in $C$ have a distance at least 2 in $H / c(K)$, and no contractibly connected chord of $C$ outside $D$ exists by Claim \ref{clm:no-chord-C-outiside-D}. Hence, our approach is to search a circuit $C'$ in the dual of the Petersen graph or graphs in the Blanuša snark family (e.g. Petersen graph, Blanuša, Blanuša-V1, Blanuša-H1 snark) that was originally $C$ before contractions. To do that, we search a circuit $C'$ of length $l \in \{6, 7\}$ that bounds a disk $D'$ with the following conditions in the dual of the Petersen graph or graphs in the Blanuša snark family.
\begin{itemize}
    \item[(a)] No two vertices of $C'$ are identified and no chord of $C'$ exists in $D'$.
    \item[(b)] The number of vertices not in $D'$ is at least $l - 2$.
    \item[(c)] For every two noncontractible cycles of length two $v_0v_1, v_2v_3$ ($i \neq j \rightarrow v_i \neq v_j $) of the Petersen graph or graphs in the Blanuša snark family, both two annulus regions separated by these cycles share a common area with the region strictly inside $D'$.
\end{itemize}
The condition (a) is obvious in this case. The condition (b) comes from the definition of $C$. The condition (c) comes from Lemma \ref{lem:(2,2)-annulus-cut}. The restriction $i \neq j \rightarrow v_i \neq v_j$ represents the cyclicity of a cut in Lemma \ref{lem:(2,2)-annulus-cut}. The condition (c) is important in restricting infinite Petersen-like graphs to finite ones. More precisely, we do not need to care about Blanuša-H$k$ ($k \geq 1$), Blanuša-V$l$ ($l \geq 2$) since every circuit that satisfies the condition (c) in every embedding of these graphs is of length at least 8, so our targets are only the Petersen graph, Blanuša snark, and Blanuša-V1 snark. Luckily, almost all circuits in the dual of every embedding of these graphs do not satisfy all conditions (a), (b), (c). Only one circuit in these graphs satisfies all the conditions; it is the circuit of length 7 that surrounds the vertex of degree 7 in Blanuša-V1 snark (see Figure \ref{fig:blanusaV1-embedding}).
We handle the circuit by checking that all the configurations of $\mathcal{K}$ have at least two vertices or a vertex of degree not 7 after contraction and 2,3-cycle reductions by computer. Hence, we do not need to worry about getting such a circuit after contraction. Thus we have the following.
\begin{claim}\label{clm:not-Petersen-like1-cut6,7}
    For every C-reducible configuration $K \in \mathcal{K}$ except C(1), C(27), which appears strictly inside $D$, after contracting c($K$) in $G'$, assume that every two non-consecutive vertices of $C$ are of distance at least $2$ in $H / c(K)$. Then, $G'$ does not become a Petersen-like graph after contraction and 2,3-cycle reductions.
\end{claim}

\subsubsection{2,3-cycle reductions outside $D$}\label{sect:2,3-cycle-Petersen-like}
\showlabel{sect:2,3-cycle-Petersen-like}
\begin{figure}
    \centering
    \includegraphics[width=0.9\linewidth]{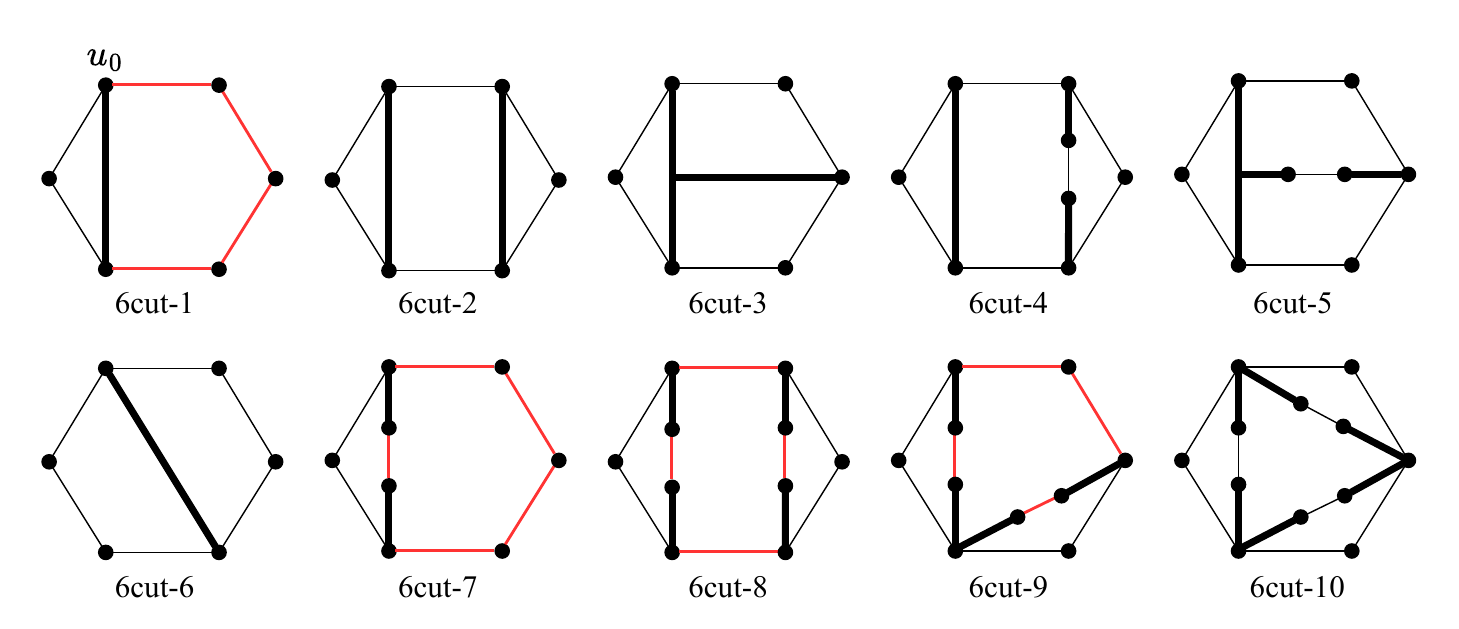}
    \caption{The 10 cases checked by computer. 
    In the figure, $u_0$ is the vertex in the upper left. The remaining vertices $u_1,u_2,...$ start after $u_0$ in the counter-clockwise order.
    The bold line represents contraction edges of $K$. 
    The red line represents a set of edges that consist of a contractible cycle of length $4$, $5$, or $6$ that will not be erased after reductions.}
    \label{fig:cut6-check-cases}
\end{figure}

\begin{figure}
    \centering
    \includegraphics[width=0.7\linewidth]{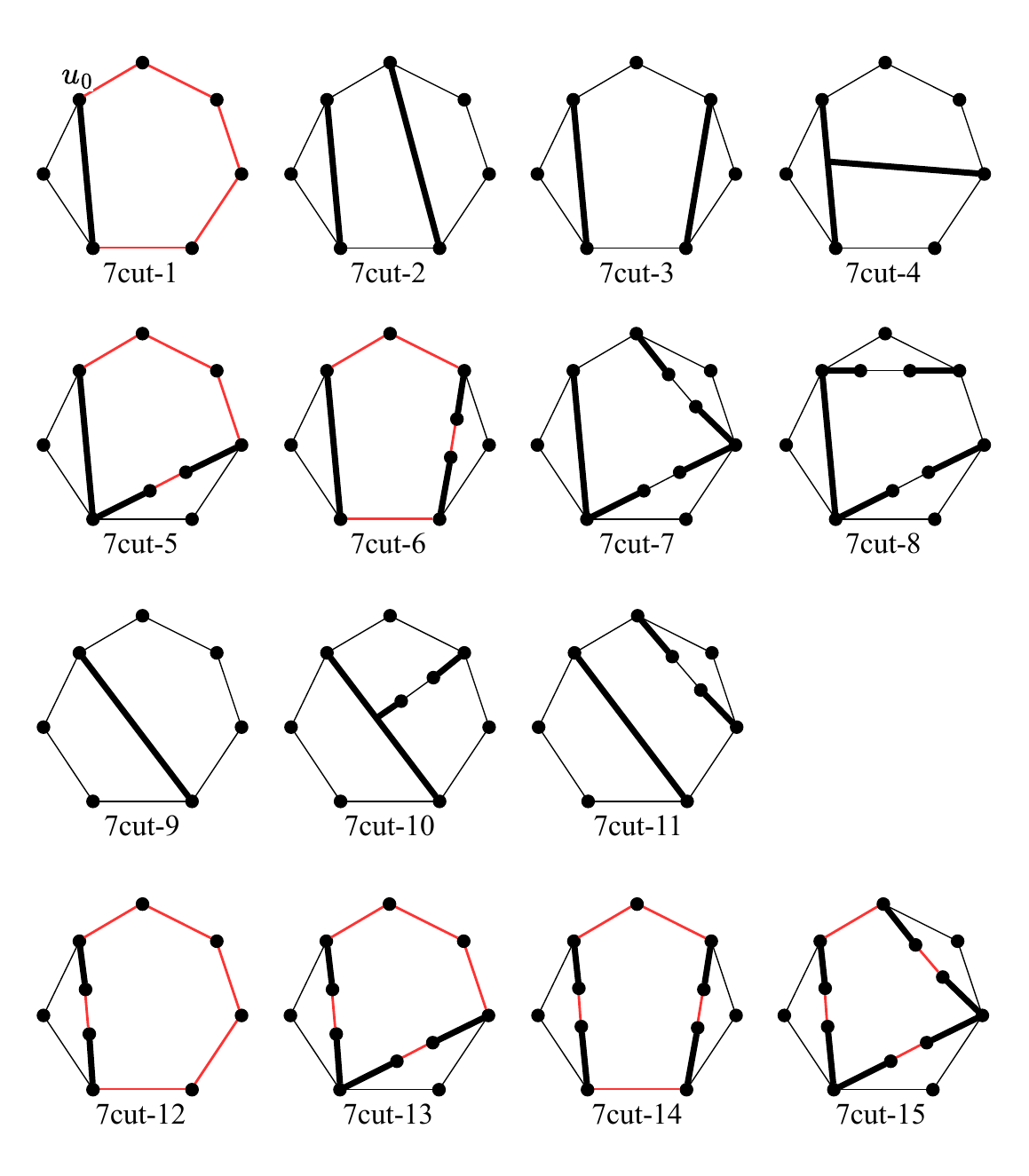}
    \caption{The 15 cases checked by computer. 
    In the figure, $u_0$ is the vertex in the upper left. The remaining vertices $u_1,u_2,...$ start after $u_0$ in the counter-clockwise order. 
    The bold line represents contraction edges of $K$.
    The red line represents a set of edges that consist of a contractible cycle of length $4$, $5$, or $6$ that will not be erased after reductions.}
    \label{fig:cut7-check-cases}
\end{figure}

Second, we consider the other cases: the distance between two non-consecutive vertices of $C$ may be at most 1 in $H$ after contraction.
In this case, we check the number of vertices inside a circuit of specified length in $H$ after contraction and 2,3-cycle reductions, which is explained later. We have to enumerate contractible 2,3-cycles that get rid of vertices. Especially, cycles that can contain edges outside $D$ are dangerous. This is because such a cycle does not contradict the almost-minimality of $C$. There are two possibilities of appearing such a cycle after contractions. One consists of a path inside $D$ and a contractibly connected path of $C$ whose length is at most 3. The other consists of two paths inside $D$ and two homotopic noncontractibly connected paths of $C$, where a \textit{noncontractibly connected path} of $C$ is a path $P$ outside $D$ that connects two vertices of $C$ such that $C + P$ contains some noncontractible cycle. Two noncontractibly connected paths $P_1, P_2$ of $C$ are \emph{homotopic} if two noncontractible cycles in $P_1+C, P_2+C$ are homotopic. The former kind of cycle appears in restricted situations by Claims \ref{clm:no-chord-C-outiside-D} and \ref{clm:no-two-path-C-outside-D} (See two cases (A), (B) below.). However, in these restricted situations, such a cycle still appears and we have no idea how to count vertices inside the cycle. Hence, we quit assessing the number of vertices in the region that can be bounded by such a cycle, but in other regions, we calculate it. To do so,
we enumerate all possible combinations of distances in $H / c(K)$ obtained by combining the following two cases (A), (B). 
\begin{itemize}
    \item[(A)] Two different nonconsecutive vertices in $C$ are identified (of distance 0 in $H / c(K)$) after contraction, 
    \item[(B)] Two vertices $u_i, u_{i+2}$ in $C$ are of distance 1 in $H / c(K)$.
\end{itemize}
These two cases are important in tracking 2,3-cycles that contain vertices both inside and outside $D$. By Claims \ref{clm:no-chord-C-outiside-D} and \ref{clm:no-two-path-C-outside-D}, vertices described in (A), (B) combined with a contractibly connected path of length 2, or 3 are the only cases that can be a 2,3-cycle that contain vertices both inside and outside $D$ after contraction of $c(K)$ by one contractibly connected path. 
We enumerate all combinations of distances and check that these situations do not happen by the contraction of $K \in \mathcal{K}$ needed for C-reducibility. This check is done by computer. All the cases are enumerated in Claim \ref{clm:not-Petersen-like2-cut6,7}: (6cut-1),...,(6cut-10),(7cut-1),....,(7cut-15). There are 10 cases for $l=6$, and 15 cases for $l=7$. The distance function $d(\cdot,\cdot)$ is measured in $H / c(K)$. The corresponding figures are shown in Figures \ref{fig:cut6-check-cases} and \ref{fig:cut7-check-cases}. 
Note that the red-colored edges are related to the evidence used to prove the resulting graph after contraction and 2.3-cycle reductions is not a Petersen-like graph. Indeed, a set of red-colored edges becomes a cycle of length 4, 5, or 6 after contraction. Also, the disk bounded by the cycle after contraction is not surrounded by 2,3-cycle after contraction by Lemma \ref{minc}, and Claims \ref{clm:no-chord-C-outiside-D} and \ref{clm:no-two-path-C-outside-D}, so the disk remains after 2,3-cycle reductions. Note that it is not a mistake that the cycle of length $4$ in (7cut-9) after contraction is not red-colored since the disk bounded by the cycle can be removed by a contractibly connected $u_3u_0$-path of length 3. 
No cycle of length 4, 5, or 6 of the dual of any embedding of the Petersen graph or graphs in the Blanuša snark family has the following properties (i), (ii), (iii).
The reader can check these properties by considering Figures \ref{fig:petersen-embedding}--\ref{fig:blanusaV2-embedding}.

\begin{itemize}
    \item[(i)] A circuit of length $4$ that bounds a disk that contains at least 1 vertex strictly inside.
    \item[(ii)] A circuit of length $5$ that bounds a disk that contains at least 2 vertices strictly inside.
    \item[(iii)] A circuit of length $6$ that bounds a disk that contains at least 4 vertices strictly inside. 
\end{itemize}
Hence, the number of vertices that remain after contraction inside the cycle that consists of red edges and several contracted edges can be a certificate of being non-Petersen-like graphs.

To do so, we calculate the number of vertices after contraction. If a cycle inside $D$ becomes a 2,3-cycle after contractions, it is possible that this cycle may originally contradict the minimality of $C$. We only consider cycles inside $D$ that do not contradict the minimality of $C$ and assume that vertices inside these cycles are erased.
We have to take care also about cycles that are not completely inside $D$.
As explained at the beginning of the section, a path inside $D$ and a contractibly connected path of $C$ can constitute a 2,3-cycle that is outside $D$, but we do not have to care about it since such a cycle does appear in the region where we count the number of vertices \footnote{For example, when checking (6cut-1), you may wonder if the distance between $u_2$ and $u_4$ can be 1 after contractions and a contractibly connected $u_2u_4$ path of length 2 can become a 2,3-cycle that get rid of vertices inside the cycle that consists of red-colored edges and contracted edges. However, we checked this case in (6cut-5), so we can assume that no such cycle appears in (6cut-1).} (i.e. the region bounded by red edges and several contracted edges).

Also, we care about a cycle that consists of two paths inside $D$ and two homotopic noncontractibly connected paths of $C$. The length of two homotopic noncontractibly connected paths of $C$ must be short, namely at most 3, to become a 2,3-cycle after reductions. Using this, we check whether the existence of such paths contradicts the length of cycle $C$ (which can be 6, 7) or not.
We assume that vertices are deleted only by cycles that pass the check (i.e. that are compatible with the length of $C$.).
% We assume vertices are deleted by cycles that contain such paths that are compatible with $C$.
More detailed implementations of this check are explained in Appendix \ref{subsect:code-cuts6,7}.

We treat (7cut-12) as a special case. If we get a Petersen-like graph in (7cut-12), and a 6-cycle in (7cut-12) bounds a disk that consists of exactly 3 vertices strictly inside, all these 3 vertices have degree 5 by observing a 6-cycle in the dual of any embedding of the Petersen graph or graphs in Blanuša snark family.
When such a case happens, the cycle does not satisfy both conditions (b) and (c) in Section \ref{sect:no-2,3-cycle-Petersen-like} (Here, we have to restrict that the number in the condition (b) is $l-3$ by considering the possibility of being removed vertices by a 3-cycle that contains vertices outside $D$). Hence, our condition (iii) in (7cut-12) can become weaker than described above: the number $4$ becomes $3$. We actually check under this condition.

The other criterion of this check is checking the almost-minimality of $C$. The criterion is the same as the one used in Section 7.1 in \cite{proj2024}. We explain it in short to make this paper self-contained. 
When some two vertices of $C$ become identified with each other after contractions of $c(K)$, these two vertices are in the ring of $K$.
When some vertices on $C$ are of distance 1 after contractions of $c(K)$, these two vertices are in the ring of $K$, or one is in the ring, and the other is adjacent to a vertex of the ring. 
Hence, assuming one of (6cut-1),...,(7cut-15) happens, we can find a short contractibly connected path of the ring of $K$, which is part of $C$. This path would constitute a cycle with a path in $K$ that could contradict the almost-minimality of $C$.
Let us explain how to check the argument above, using a specific example. In fact, this example corresponds to (6cut-2).
Let us assume $l=6$ and $d(u_0, u_2) = d(u_3,u_5) = 0$ after contraction of $c(K)$. 
Note that these four vertices belong to a ring of $K$.
For each configuration $K \in \mathcal{K}$, we enumerate four vertices $a, b, c, d$ in the ring such that $a, b, c, d$ are in the clockwise order listed in the ring of $K$, and $d(a, b) = d(c, d) = 0$ after contracting edges of $c(K)$.
In this case, we have several constraints that two vertices of the ring of $K$ are connected by a path of specified length outside $K$ (e.g. a contractibly connected $ab$(, $ba$)-path of length 2(, 4) of the ring exists respectively). We check by computer whether or not all such paths constitute a cycle that contradicts the almost-minimality of $C$. 

By computer-check described above, we ensure that almost all configurations in $\mathcal{K}$ except C(1), C(27) does not become a Petersen-like graph after contraction and reductions. However, there is one configuration in (7cut-9) that we cannot conclude immediately. We can handle it in Claim \ref{clm:proj1113-cut6,7}.

% \textbf{BETTER??: move the claim proj1113 to the appendix.}
\begin{figure}
    \centering
    \includegraphics[width=7cm]{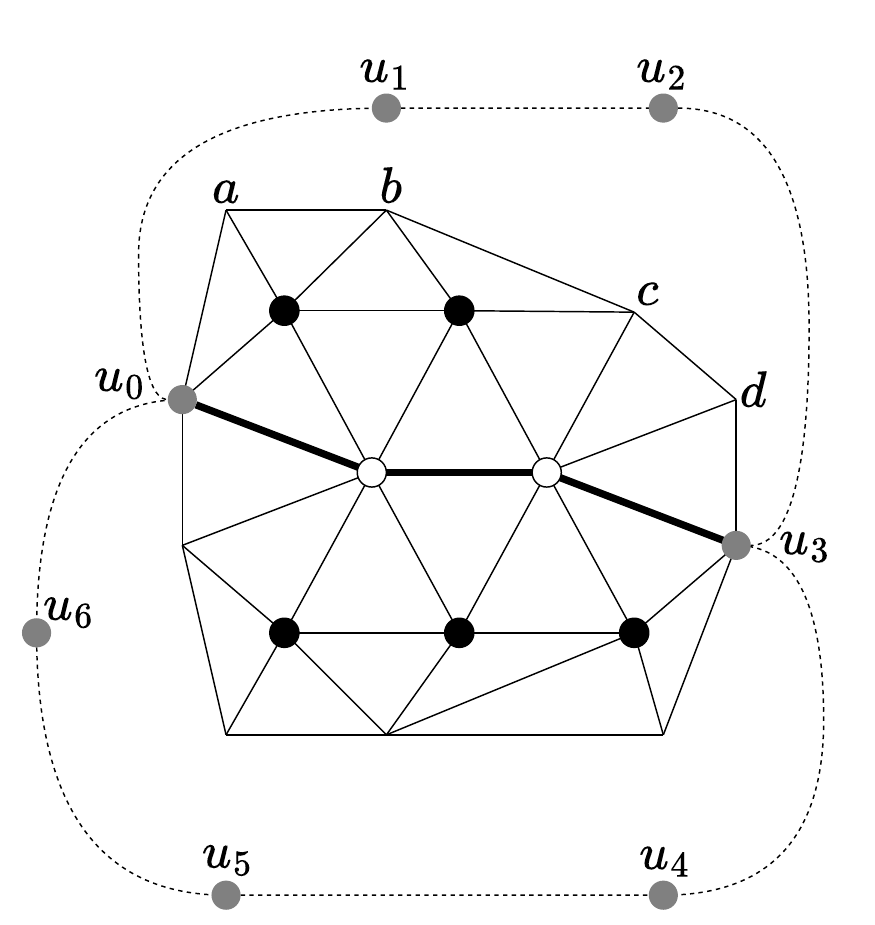}
    \caption{}
    \label{fig:cut6,7-proj1113}
\end{figure}
\begin{claim}\label{clm:proj1113-cut6,7}\showlabel{clm:proj1113-cut6,7}
    The configuration shown in Figure \ref{fig:cut6,7-proj1113}, (7cut-9) does not occur.
\end{claim}
\begin{proof}
    By computer check, (7cut-9) can occur when $C$ is the cycle described in Figure \ref{fig:cut6,7-proj1113}. The bold lines represent contraction edges. We denote $C'$ as the cycle that consists of contraction edges and $u_0u_1$, $u_1u_2$, $u_2u_3$.
    The vertex $u_1$ is not identical to any vertex of $b,c,d$, since this leads to a cycle that contradicts Lemma \ref{minc}. The vertex $u_2$ is not identical to any vertices of $a,b,c$ by the same reason. Hence, $C'$ bounds a disk where at least 4 vertices ($b,c$ and two vertices of degree $5$ in this configuration) exist. That contradicts the minimality of $C$.
\end{proof}

The above discussion and the result of computer-check imply the following claim.

\begin{claim}\label{clm:not-Petersen-like2-cut6,7}\showlabel{clm:not-Petersen-like2-cut6,7}
    For every C-reducible configuration $K \in \mathcal{K}$ except C(1), C(27), which appears strictly inside $D$, after contracting c($K$) in $G'$, the combination of the distances enumerated in (6cut-$k$) ($1 \leq k \leq 10)$, (7cut-$l$) ($1 \leq l \leq 15$) (distance $d(\cdot, \cdot)$ is measured in $H / c(K)$, rotational symmetry and mirror symmetry of indexes are ignored) does not happen or, $G'$ does not become a Petersen-like graph. 

    \begin{itemize}
        \item (6cut-1) $d(u_0,u_2) = 0$,
        \item (6cut-2) $d(u_0,u_2) = 0$ and $d(u_3,u_5) = 0$,
        \item (6cut-3) $d(u_0,u_2) = 0$ and $d(u_2,u_4) = 0$,
        \item (6cut-4) $d(u_0,u_2) = 0$ and $d(u_3,u_5) = 1$,
        \item (6cut-5) $d(u_0,u_2) = 0$ and $d(u_2,u_4) = 1$,
        \item (6cut-6) $d(u_0,u_3) = 0$,
        \item (6cut-7) $d(u_0,u_2) = 1$,
        \item (6cut-8) $d(u_0,u_2) = 1$ and $d(u_3,u_5) = 1$,
        \item (6cut-9) $d(u_0,u_2) = 1$ and $d(u_2,u_4) = 1$,
        \item (6cut-10) $d(u_0,u_2) = 1$ and $d(u_2,u_4) = 1$ and $d(u_4,u_0) = 1$,
        \item (7cut-1) $d(u_0,u_2) = 0$,
        \item (7cut-2) $d(u_0,u_2) = 0$ and $d(u_3,u_6) = 0$,
        \item (7cut-3) $d(u_0,u_2) = 0$ and $d(u_3,u_5) = 0$,
        \item (7cut-4) $d(u_0,u_2) = 0$ and $d(u_2,u_4) = 0$,
        \item (7cut-5) $d(u_0,u_2) = 0$ and $d(u_2,u_4) = 1$,
        \item (7cut-6) $d(u_0,u_2) = 0$ and $d(u_3,u_5) = 1$,
        \item (7cut-7) $d(u_0,u_2) = 0$ and $d(u_2,u_4) = 1$ and $d(u_4,u_6) = 1$,
        \item (7cut-8) $d(u_0,u_2) = 0$ and $d(u_2,u_4) = 1$ and $d(u_5,u_0) = 1$,
        \item (7cut-9) $d(u_0,u_3) = 0$,
        \item (7cut-10) $d(u_0,u_3) = 0$ and $d(u_3,u_5) = 1$,
        \item (7cut-11) $d(u_0,u_3) = 0$ and $d(u_4,u_6) = 1$,
        \item (7cut-12) $d(u_0,u_2) = 1$,
        \item (7cut-13) $d(u_0,u_2) = 1$ and $d(u_2,u_4) = 1$,
        \item (7cut-14) $d(u_0,u_2) = 1$ and $d(u_3,u_5) = 1$,
        \item (7cut-15) $d(u_0,u_2) = 1$ and $d(u_2,u_4) = 1$ and $d(u_4,u_6) = 1$.
    \end{itemize}
\end{claim}

The proof of Lemma \ref{lem:6,7-cut} is obtained by applying Claims \ref{clm:not-contractible-bridge-cut6,7}, \ref{clm:not-Petersen-like1-cut6,7} and \ref{clm:not-Petersen-like2-cut6,7}, and the result in Section 
\ref{subsubsect:noncontractible-loop}, \ref{subsubsect:particular-contractible-loop}.
% \ref{subsubsect:noncontractible,particular-contractible,loop}.

% ====================================================
\section{Wrappings of Configurations}
\label{sect:lowrep}\showlabel{sect:lowrep}
% ====================================================

With the discharging algorithm described in Section \ref{sect:discharging}, we have confirmed that either a vertex $v \in G'$ with the resulting charge $T(v) > 0$ or one of its neighbors is contained in some weakly appearing (defined in Section \ref{sect:reducible configurations}) reducible configuration $K$.
If $K$ appears in $G'$ as well (as an induced subgraph), we can simply reduce $G'$ by contracting some arbitrary edge in $G(K)$ when $K$ is D-reducible, or the edges in $c(K)$ when $K$ is C-reducible.
However, if this is not the case, there is a wrapping of $K$ (say, $K'$) such that $K'$ appears in $G'$, instead of $K$.
We confirm that for almost all the cases of $K'$, $G(K')$ admits a representativity one curve (in the resulting graph), or $K'$ is reducible as well.

To be more precise, we perform a brute-force check on every configuration $K\in \mathcal{K}$ to confirm that the following property holds.

\begin{lem}\label{lem:weakly}\showlabel{lem:weakly}
    For every general configuration $K'$ such that $K'$ is a strict wrapping of $K \in \mathcal{K}$ (i.e. $G(K')$ is not isomorphic to $G(K)$), $K'$ is reducible with $c(K') \leq 4$.
    Moreover, for all cases such that $3 \leq c(K') \leq 4$, there is always a small cyclic cut in $G \dotdiv c(K')$ which violates Lemma \ref{lem:(2,2)-annulus-cut}.
\end{lem}

We have implemented an algorithm that takes an input configuration $K$, and outputs a set of general configurations that are all possible wrappings of $K$.
This is done by taking two distinct vertices $u,v \in V(S)$ and an edge $e_u, e_v \in E(S)$ such that $u \in e_u, v \in e_v$, and trying to ``identify" the two vertex/edge pairs $(u, e_u)$ and $(v, e_v)$
(Here, $S$ is the \emph{free completion} of the configuration $K$). Then we try to identify vertices as many as possible. 
Formally, the identification process succeeds if the following condition is satisfied.

\begin{dfn}
    A \emph{rooted path} on $S$ is a sequence of vertices $v_1, v_2, \cdots, v_k \in V(S)$ (for $k \geq 2$) such that the path $P = v_1 v_2 \cdots v_k$ is on $S$.
    We call $v_1$ the \emph{starting vertex}, $v_1v_2$ the \emph{starting edge} (if it exists), and $v_k$ the \emph{ending vertex} of the rooted path.
    The \emph{positive/negative angle} on $v_i$ ($1 < i < k$) in a rooted path $P$ is the smallest $n$ such that $\pi_{v_{i}}^{\pm n}(v_{i-1}) = v_{i+1}$, when $\pi_{v_{i}}$ is the local rotation\footnote{Here, we borrow the term \emph{local rotation} from \cite{MT} but we slightly alter the definition so that it returns vertices instead of edges: A local rotation $\pi_v$ is a cyclic permutation of the neighbors of $v$, such that if we fix a neighbor $u$ of $v$, the vertices $u, \pi_v(u), \pi_v(\pi_v(u)), \cdots$ form a cyclic walk around $v$. Since the underlying graph is a triangulation, the three vertices $v, u, \pi_v(u)$ must always construct a triangle.} on $v_i$ and $\pi_{v_{i}}^{\pm n}$ is the function $\pi_{v_{i}}$ or $\pi_{v_{i}}^{-1}$ applied $n$ times.

    Two rooted paths $P$ and $Q$ are \emph{parallel} if the following conditions are satisfied. 
    \begin{itemize}
        \item $P = (v_1, \cdots, v_k)$ and $Q = (u_1, \cdots, u_k)$ both have the same length.
        \item For all $1 \leq i < k$, if $v_i$ and $u_i$ are both in $G(K)$, $\gamma_K(v_i) = \gamma_K(u_i)$
        \item For all $1 < i < k$, either the positive angle or negative angle on $v_i$ on $P$ and $u_i$ on $Q$ is the same.
    \end{itemize}
\end{dfn}

\begin{claim}
    Let $K$ be a configuration and let $u \in V(K), v \in V(S(K))$ be two different vertices, $e_u, e_v$ be two different edges incident to $u,v$, respectively, in $S(K)$ (The free completion of $K$).
    If for any parallel rooted paths $P$ and $Q$ in $S(K)$ where the starting vertex is $u,v$ and the starting edge is $e_u, e_v$, respectively, one of the following conditions holds, there exists a general configuration $K'$ such that $K'$ is a wrapping of $K$. 
    \begin{itemize}
        \item One of the ending vertices in $P, Q$ is not in $V(K)$.
        \item The degree (value of $\gamma_K$) of the ending vertices of $P, Q$ are equal.
    \end{itemize}
    This $K'$ can be constructed by ``gluing" the vertices $u, v$ and the edges $e_u, e_v$ so that any path between $u$ and $v$ becomes a non-contractible curve on the torus.
\end{claim}

We iterate over all possible parallel rooted paths using DFS, and identify all corresponding vertex pairs that lie on each path.
If for some paths, the endpoints of them have different $\gamma_K$ values, we abort the identification process.
If not, we proceed to construct a wrapping $K'$ by identifying all pairs of vertices that we accumulated in the DFS process.

For more detailed implementations of this algorithm, refer to the pseudocode in Appendix \ref{subsect:code-wrapping}.

% ====================================================
\section{Configuration of contraction size $\geq 3$}
\label{sect:safecont}\showlabel{sect:safecont}
% ====================================================

The main result of this section is the following.

\begin{lem}\label{T(v)>0T}\showlabel{T(v)>0T}
  Let $G$ be a minimal counterexample, and let $K$ be a C-reducible configuration in $\mathcal{K}$ \emph{weakly appearing} in $G$.
%  (Note that weak appearance is insufficient here.)
  Then $G \dotdiv c(K)$ is 2-edge connected and not Petersen-like.
\end{lem}

We first observe that if $K'$ is a strict wrapping of $K \in \mathcal{K}$, by Lemma \ref{lem:weakly}, we confirm $|c(K')| \leq 2$. 

When the contraction size of a configuration $K$ is two or less, it is easy to prove this lemma. 
%even without the assertion that $K \in \mathcal{K}$.
Indeed, in this case, due to the high cyclic-edge-connectivity of $G$, a $2/3$-edge-cut reduction occurs in $G \dotdiv c(K)$ if and only if $c(K)$ are members of an edge cut of size exactly 5.
Even in this case, the reduction will only reduce the order of $G$ by 4. 
As mentioned in the introduction, we observe that $G$ is of order at least 38 by the result of \cite{brinkmann2013generation} (that generates all the snarks of order at most 36. It turns out that only the snarks in $\mathcal{T}_0$ can be embedded in the torus). 
Therefore, the resulting graph is of order at least $34$ but if the resulting graph is in $\mathcal{T}_0$, this would yield a $(2, 2)$-annulus-cut, a contradiction to Lemma \ref{lem:(2,2)-annulus-cut}. 

Therefore, we can assume that in Lemma \ref{T(v)>0T}, we may assume that $K$ \emph{does indeed appear}. Hereafter, we only consider this case. 

However, when $|c(K)| \geq 3$, we would need to individually track the contraction pattern in every $K \in \mathcal{K}$.
We have implemented an algorithm to track the possible degree count of every vertex after contraction and check if its neighboring structure contains a bridge or matches the structure of (dual of) Petersen-like graphs.
Upon checking, we also make use of the following fact, which is a direct corollary of Lemmas \ref{minc} and \ref{lem:6,7-cut}.
\begin{fact}
    \label{fact:567-cuts}\showlabel{fact:567-cuts}
    Let a contractible $k$-cycle ($5 \leq k \leq 7$) $F$ be contained in the dual graph $G'$.
    Then, the size of the smaller connected component of $G' - F$ is less than or equal to $1$ when $k = 5$, $3$ when $k = 6$, and $4$ when $k = 7$.
\end{fact}

\subsection{Avoiding bridges}
\label{subsect:bridge}
\showlabel{subsect:bridge}

For $G \dotdiv c(K)$ to contain a bridge, there must be some edge cut $F \subset E(G)$ such that $|F \setminus c(K)| = 1$.
We classify possible cases by the type of cut $F$; when $F$ is a disk cut or when $F$ is an annular cut.

\subsubsection{Non contractible loop}
\label{subsubsect:noncontractible-loop}
\showlabel{subsubsect:noncontractible-loop}
\begin{figure}[htbp]
    \centering
    \includegraphics[width=10cm]{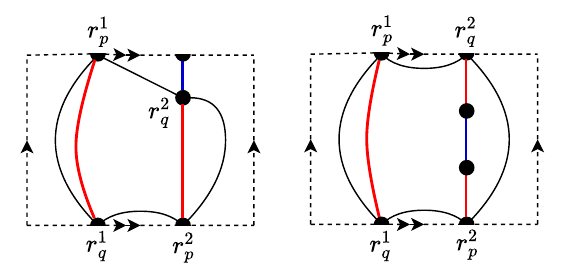}
    \caption{The dark solid line represents the ring, the red lines represent the contraction edges and the blue edge becomes a loop after contraction. The left and right figure represents case (i), (ii) respectively.}
    \label{fig:noncontractible-loop}
\end{figure}
Here, we care about a noncontractible loop in the torus, after contraction. 
Let $K$ be a configuration and $S$ be a free completion of $K$ with its ring. A noncontractilbe loop exists if two pairs of vertices $(r_p^1, r_q^1), (r_p^2, r_q^2)$ in the ring of $K$ satisfy either
(i) $d_{S / c(K)}(r_p^1, r_q^1) = 0, d_{S / c(K)}(r_p^2, r_q^2) = 0$, or
(ii) $d_{S / c(K)}(r_p^1, r_q^1) = 0, d_{S / c(K)}(r_p^2, r_q^2) = 1$. 
If a loop exists after contraction, 
$r_p^1$ is identical to $r_q^1$ and the distance between $r_p^2$ and $r_q^2$ is 1 (or swapping the roles of ($r_p^1, r_q^1$) and ($r_p^2, r_q^2$)) in the case (i), or 
$r_p^1$ is identical to $r_q^1$ and $r_p^2$ is identical to $r_q^2$ in the case (ii) (See Figure \ref{fig:noncontractible-loop}). 
If the contraction size is at most 3, a noncontractible loop does not exist. It is because case (i) is impossible and case (ii) is possible only when $d_S(r_p^1,r_q^1)=2$ and $d_S(r_p^2, r_q^2) = 2$, but this would yield a cut that contradicts Lemma \ref{minc} or Lemma \ref{lem:(2,2)-annulus-cut}.
For all $K$ of our set $\mathcal{K}$, we check whether the above condition can be satisfied after contracting $c(K)$.
If it can be satisfied, we re-check to see whether the annular configuration is reducible in this case. To simplify the implementation, we create an annular configuration from $K$ by only identifying one pair of vertices of the ring of $K$.
Almost all configurations obtained are reducible with contraction size at most 3.
However, 6 cases were not concluded immediately. We re-checked the annular configuration obtained when a loop appears (i.e., two pairs of vertices are identical or one pair of vertices is identical, and one pair of vertices is of distance exactly one).
It turns out that all the configurations obtained are D-reducible. Hence, a noncontractible loop does not exist after contraction.

\subsubsection{Contractible loop when a free completion contains a noncontractible cycle}
\label{subsubsect:particular-contractible-loop}
\showlabel{subsubsect:paricular-contractible-loop}
\begin{figure}[htbp]
    \centering
    \includegraphics[width=10cm]{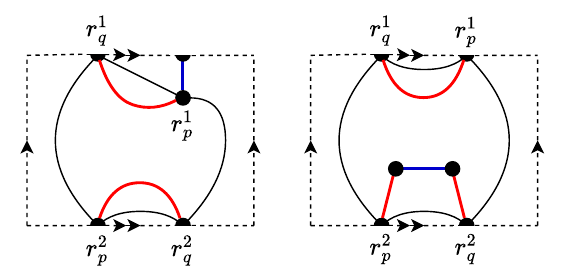}
    \caption{The dark solid line represents the ring, the red lines represent the contraction edges, and the blue edge becomes a loop after contraction. The left and right figures represent case (i) and (ii), one of the degree patterns in the above lists, respectively.}
    \label{fig:contractible-loop}
\end{figure}
Here, we handle one particular type of contractible loop\footnote{This type of contractible loop is different from other contractible loops. because it does not violate the minimality of $C$ in proving Lemma \ref{lem:6,7-cut}. So we have to deal with it separately.}. 
This loop exists in a similar situation to a noncontractible loop (See Figure \ref{fig:contractible-loop}). The distance described above for cases (i) or (ii) is the same.
The difference is how these vertices are connected outside $S$.
$r_q^1$ is identical to $r_p^1$ and the distance between $r_q^2$ and $r_p^1$ is 1 (or swapping the roles of ($r_q^1, r_p^1$) and ($r_q^2, r_p^1$)) in the case (i), or 
$r_q^1$ is identical to $r_p^1$ and $r_q^2$ is identical to $r_p^2$ in the case (ii).
The contraction size of at most three does not cause this kind of loop in the same discussion as the noncontractible loop case.
We check in the same way as the noncontractible loop case. The annular configurations obtained by identifying one pair of vertices are C-reducible with contraction size at most 3. There are 2 cases not concluded immediately, but we check the situation when such a loop appears. It turns out that they are D-reducible.

\subsubsection{Other contractible loops}
\label{subsubsect:contractible-loop}
\showlabel{subsubsect:contractible-loop}

Finally, there is a case when a normal contractible loop of length $l$ turns into a self-loop after $l-1$ edges are contracted. 
This may happen when a configuration contains a particularly long path consisting of $5$ or more contraction edges.
However, even in such extreme cases, we have implemented a program and checked that the outside structure does not form a loop by using Fact \ref{fact:567-cuts}.

\subsection{Avoiding Petersen-like graphs}
\label{subsect:avoidpet}
\showlabel{subsect:avoidpet}

We now check whether the contraction of $c(K)$ would make $G$ turn into a Petersen-like graph.
To start, we use the algorithm introduced in the projective planar case \cite{proj2024} to eliminate vertices that could disappear after $2/3$-cut reductions.
Although we do need to modify some parts to consider annular cuts, the overall algorithm is the same. (For implementation specifics, refer to the Appendix.)

After the above elimination, we want to check whether or not the remaining structure matches any graph in $\mathcal{T}_0$.

Hereafter, we use the dual of the remaining graph (hereafter, we use $T$) since the degree of vertices is easier to handle than the size of faces in the original cubic graph.
For every graph in $\mathcal{T}_{0}$, we take its dual and inspect every vertex $v$.
Since the dual is a triangulation, there should be a walk of size $d$ surrounding $v$, where $d$ is the degree of $v$.
For example, when $v$ is the vertex of degree 6 in the Petersen graph's dual, the walk surrounding $v$ consists of vertices whose degree is 5, 9, 5, 9, 5, 9, in clockwise order.
The degree list of all possible walks (excluding rotational and reflective symmetry) is given in Table \ref{tab:neighbor-degree-pattern}. 
Note that the graphs of the Blanuša snark family do not have a unique embedding. Each embedding is determined by the differences between $(2, 2)$-annulus-cuts, and we want all the degree lists for each possible embedding.
The list below was generated by listing all embeddings of the Petersen graph, Blanuša snark, Blanuša-V1 snark, Blanuša-V2 snark, and the Blanuša-H1 snark. It is easy to see that no other degree lists are possible in any of the bigger snarks in either Blanuša family. 

\begin{table}[htbp]
    \centering
    \begin{tabular}{c|l}
         Size & Patterns\\
         \hline
         \multirow{2}{*}{5}
         & 55679, 55688, 55689, 5568T, 55779, 5577T, 55858, 55859, 55959, 55969, 55T5T, 56868, 56869, \\
         & 56877, 56886, 56888, 56897, 57768, 57769, 57777, 57788, 57797, 57988 \\
         6 & 557557, 557558, 558558, 575868, 575959, 585868, 585959, 595959 \\
         7 & 5575657, 5575757, 5575857, 5575957, 5575T57, 5595659 \\
         \multirow{2}{*}{8} 
         & 55558558, 55665566, 56565656, 56565657, 56565658, 5657565T, 56585658, 56585758, 56585858, \\
         & 56585958, 56585659, 56595659, 565T565T \\
         9 & 555575757, 555575857, 556556556, 556556557, 556556558 \\
         10 & 5555755558, 5555855558
    \end{tabular}
    \caption{
        Table of neighboring degree patterns of Petersen-like (dual) graphs in the torus.
        Each row contains all the possible degree patterns of a certain degree $k$.
        Every degree pattern of degree $k$ is shortened to a string of size $k$, and each letter of the string is the size of the face. ($T$ represents $10$).
    }
    \label{tab:neighbor-degree-pattern}
\end{table}

With this table, we track the degree of each vertex of $K / c(K)$ and check if it matches with one of the above lists' degree patterns.
When a vertex is of degree exactly $4$ or $11$ or more, the resulting graph $G \dotdiv c(K)$ would not be a Petersen-like graph.
If not, we list up the degree of each neighbor vertex and check if it matches with one of the patterns in the corresponding row of the table above.
If it does not, the resulting graph would not be a Petersen-like graph.

We implement the above matching algorithm and run it over every configuration in $\mathcal{K}$ with contraction size $\geq 3$ (the pseudocode is provided in Appendix \ref{subsect:code-petlike-check}). 
For some configurations, since some contraction edge sets $c(K)$ failed to pass the test, we have tested multiple contraction edge sets that admit C-reducibility of $K$, and selected the edge set with a successful outcome of the above algorithm.

As a result, we found that all the configurations in $\mathcal{K}$ with some contraction edge set of size $\geq 3$ pass these tests, except for one configuration (namely C(1), which was handled in Section \ref{sect:6,7-cut} because the ring size is exactly six.).

% ====================================================
\section{The flat case}
\label{sect:flatproof}\showlabel{sect:flatproof}
% ====================================================

In this section, we assume that every vertex $v$ of $G'$ satisfies $T(v)=0$ (final charge is zero).
In Lemma \ref{T(v)>0}, we show when a vertex of final charge positive exists, a reducible configuration weakly appears within its third neighbors. Our approach here is similar: when vertices whose final charge is zero are concentrated, we will find a reducible configuration near these vertices. We prove this by tiling several vertices and their neighbors, whose center vertices have final charge zero, and then by checking whether a reducible configuration weakly appears by computer in Lemma \ref{lem:charge0-concentrated}, \ref{lem:charge0-77} (An example is in Figure \ref{fig:example-combine-graph}). To ensure that the execution time of this program is practical, we take the following approach. We check the cases in Lemma \ref{lem:charge0-concentrated} from case 1 to case 9 in this order listed. First, we handle the cases when higher degree vertices of final charge zero are concentrated such as cases 1 and 2. An amount of initial charge of a higher degree vertex is a very small negative value, so when the final charge of such a vertex is zero, many neighbors send a charge to the vertex. That enables us to find many vertices of degree 5 or 6 that constitute a reducible configuration. 

Then, we can restrict the structure of a minimal counterexample $G$ and its dual graph $G'$. For example, the degree of neighbors of $v, u$ in $G'$ in case 3 is $5,6,7,$ or $8$ after handling cases 1, and 2. In this way, we prune the case analysis. We check all the cases in Lemma \ref{lem:charge0-concentrated} and all the cases in Lemma \ref{lem:charge0-77} by computer.
Note that case $i$ can be a subcase of case $j (j < i)$ (e.g. $i=8,j=6$) since we prune the analysis of case $i (1 \leq i \leq 9)$ by using the fact case $j (j < i)$ is already proved.

\begin{lem}\label{lem:charge0-concentrated}
  Suppose that one of the following is contained as a subgraph of $G'$: 
  \begin{enumerate}
    \item a vertex $v$ such that $d(v)=10,11$,
    \item an edge $vu$ such that $d(v)=9,d(u)=7,8,$ or $9$,
    \item an edge $vu$ such that $d(v)=d(u)=8$,
    \item a facial triangle $uvw$ such that $d(v)=d(u)=7, d(w)=8$,
    \item a facial triangle $uvw$ such that $d(v)=d(u)=d(w)=7$,
    \item a vertex $v$ such that $d(v)=8$, and at least three neighbors of $v$ are of degree $7$,
    \item a vertex $v$ such that $d(v)=7$, and the clockwise consecutive neighbors $u_1, u_2, u_3$ of $v$ satisfy $d(u_1)=d(u_3)=8, d(u_2) \leq 6$,
    \item an edge $vu$ such that $d(v)=7,d(u)=8$, or
    \item a vertex $v$ such that $d(v)=7$, and at least three neighbors of $v$ are of degree $7$.
    % \item An edge $vu$ such that $d(v)=d(u)=7$.
  \end{enumerate}
  Then a configuration 
  in $\mathcal{K}$ weakly appears.
\end{lem}

In addition to Lemma \ref{lem:charge0-concentrated}, we show Lemma \ref{lem:charge0-77}.

\begin{lem}\label{lem:charge0-77}
  If an edge $vu$ such that $d(v)=d(u)=7$ exists in $G'$, a configuration %that is isomorphic to one of the reducible configurations of 
  in $\mathcal{K}$ weakly appears, unless the subgraph that consists of neighbors of $uv$ is one of the graphs in the Figure \ref{fig:charge0-77}.
\end{lem}

\begin{figure}[htbp]
    \centering
    \includegraphics[width=3cm]{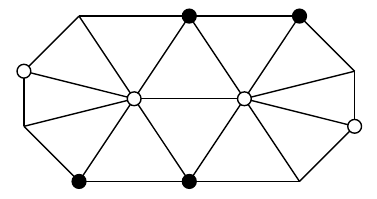}
    \includegraphics[width=3cm]{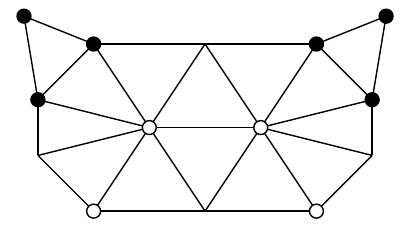}
    \includegraphics[width=3cm]{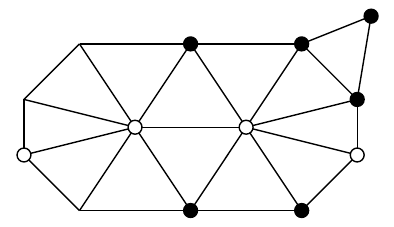}
    \includegraphics[width=3cm]{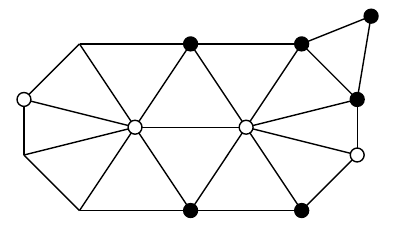}
    \includegraphics[width=3cm]{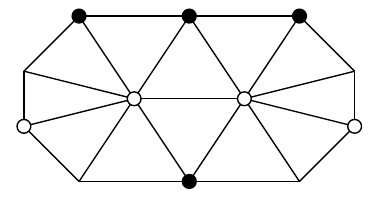}
    \caption{The exceptional cases of Lemma \ref{lem:charge0-77}. Two vertices placed in the center represent $v$ and $u$.}
    \label{fig:charge0-77}
  \end{figure}

By Lemma \ref{lem:charge0-concentrated}, every vertex $v$ of degree $8$ or $9$ is surrounded by neighbors of degree $5,6$. 
The same statement does not hold when degree is $7$ since a vertex $v$ of degree $7$ can have a neighbor of degree $7$, but this happens to only a few cases by Lemma \ref{lem:charge0-77}. We can also prove Lemma \ref{lem:charge0-77-2}.

\begin{lem}\label{lem:charge0-77-2}
  If a path of length two such that its vertices are all of degree 7 exists, 
  a configuration 
  %that is isomorphic to one of the reducible configurations of 
  in $\mathcal{K}$ weakly appears.
\end{lem}
\begin{proof}
  We consider the case that $uvw$ is a path that satisfies the assumption of this lemma. By Lemma \ref{lem:charge0-77}, each of $uv$ and $vw$ corresponds to the edge placed in the center of one of the graphs enumerated in Figure \ref{fig:charge0-77}. We say that an edge $uv$ corresponds to edge $i (1 \leq i \leq 5)$ when $uv$ corresponds to the edge in the center of $i$-th graph from the left in Figure \ref{fig:charge0-77}. \textit{A diagonal vertex of $e$} is a vertex that constitutes a facial triangle with an edge $e$. 
  We cannot tile two graphs in Figure \ref{fig:charge0-77} such that $uvw$ constitutes a path of length two, except when ($uv$, $vw$) correspond to (edge2, edge2), (edge3, edge3), (edge3, edge4), or (edge4, edge4) respectively, by the difference of degrees of diagonal vertices of $uv$ or $vw$. In each case, C(1), or C(26) is contained as a subgraph, which implies the statement of this Lemma.
\end{proof}

By combining Lemmas \ref{lem:charge0-77} and \ref{lem:charge0-77-2}, the existence of an edge $uv$ such that $d(v)=d(u)=7$ implies the same conclusion of Lemma \ref{T(v)>0}. Hence, every vertex $v$ of degree $7$ is surrounded by neighbors of degree $5,6$. It is the same situation as the degree of $8$ or $9$. We can find a reducible configuration in this situation. We check the following lemma by computer.

\begin{lem}\label{lem:charge0-789-surrounded56}\showlabel{lem:charge0-789-surrounded56}
    If all neighbors of a vertex of degree $7,8,$ or $9$, are of degree $5$ or $6$, a configuration 
    %that is isomorphic to one of the reducible configurations of 
    in $\mathcal{K}$ weakly appears.
\end{lem}

Therefore, we only have to handle the case if every vertex is of degree $5$ or $6$. A triangulation in the torus has an average degree of exactly $6$, implying that all vertices are of degree $6$. Hence, the reducible configuration C(27) weakly appears. This is the only case we use C(27). Actually, we do not need C(27) in checking the weak appearance of reducible configurations if a vertex of positive charge exists (Lemma \ref{T(v)>0}) or vertices of zero charge are concentrated. Hence we have to ensure the graph after contracting the contraction edges of C(27) does not become Petersen-like only when all degrees of $G'$ are 6. This check is easy since a vertex whose degree is at least 18, which is a clear evidence of being non-Petersen-like, exists after contraction.

In summary, we have the following.
\begin{lem}\label{lem:charge0F}\showlabel{lem:charge0F}
Assume that every vertex $v$ of $G'$ satisfies $T(v)=0$ (final charge is zero). Then a reducible configuration $K$ in the set $\mathcal{K}$ weakly appears in $G'$. 
\end{lem}

% ====================================================
\section{Main proof}
\label{sect:mainproof}\showlabel{sect:mainproof}
% ====================================================

We now discuss Theorem \ref{mainth}. Let $G$ be a minimal counterexample. We know the following properties by Lemmas \ref{minc}, \ref{lem:(2,2)-annulus-cut} and \ref{lem:(1,3)-annulus-cut}. 

\begin{enumerate}
    \item  $G$ is embedded in the torus $\Sigma$, 
    \item $G$ is cyclically 5-edge-connected, and if there is an edge-cut $F$ of size exactly five, one of the connected components of $G-F$ is a 5-cycle and 
    \item 
     no cyclic (2, $\leq$ 3)-annulus-cut nor cyclic (1, $\leq$ 3)-annulus-cut
\end{enumerate}

If the representativity of $G$ is exactly one, by Lemma \ref{lem:rep1}, either there is a subgraph $K$ of the dual graph $G'$, that is a reducible configuration in the set $\mathcal{K}$ and that does not hit the vertex in $G'$ that intersects a non-contractible curve of length exactly one in $\Sigma$, or there is an annular island described in Lemma \ref{lem:rep1}. In the latter case, reducing this annular island does not result in a Petersen-like graph, because the original graph has representativity one embedding in the torus (and Petersen-like graphs need to have representativity two for any embedding). 

Assume that the representativity of $G$ is at least two. Now we apply our discharging rules $\mathcal{R}$. By Lemma \ref{T(v)>0}, if there is a vertex $v$ of the dual graph $G'$ with $T(v) >0$,  there must exist a subgraph $K'$ of $G'$ which contains either $v$ or at least one of its neighbors and a reducible configuration in $\mathcal{K}$ weakly appears in $G'$. 

Consider the case $T(v)=0$ for all the vertices $v$ of $V(G)$. 
Then, according to Lemma \ref{lem:charge0F}, a reducible configuration in $\mathcal{K}$ weakly appears in $G'$.

But by Lemma \ref{T(v)>0T}, such a subgraph $K$ would not appear in $G'$. Indeed, reducing such a graph does not result in any Petersen-like graph. 
This completes the proof of Theorem \ref{mainth}.\qed

\section{Algorithm}
\label{sect:algorithm}\showlabel{sect:algorithm}

In this section, we consider the algorithmic corollary of our main theorem, Theorem \ref{mainth}, see Theorem \ref{algo}. Let us mention the problem again. 

\begin{quote}{\bf Algorithmic Problem}\\
{\bf Input}: A cubic graph $G$.\\
{\bf Output}: Either a three edge-coloring of $G$ or an obstruction that is not three edge-colorable, or else conclude that $G$ cannot be embedded in the torus (and output a forbidden minor for embedding a graph in the torus).\\
{\bf Time complexity}: $O(n^2)$, where $n=|V(G)|$.
\end{quote}

\medskip

{\bf Description}

\medskip

{\bf Step 1:} Embedding $G$ in the torus. 

\medskip

This can be done using the algorithm by Mohar \cite{MOHARlinear}. If $G$ cannot be embedded in the torus, we just output the obstruction (i.e., a forbidden minor) obtained from the algorithm by Mohar \cite{MOHARlinear}.

Hereafter, we may assume that the input graph $G$ is embedded in the torus $\Sigma$.

\medskip

{\bf Step 2:} Reduction to a cyclically 5-edge-connected graph. 

\medskip

This step is almost identical to that for the planar case, see \cite{RSST-STOC}. So we just give a brief sketch here, and give more details in Appendix (Section \ref{secalgoappen}). 

We try to find a cyclic cut of size four or less.  
%or a cyclic cut of size five where both connected components are not 5-cycles. 
We find such a cut (if it exists) with the smallest size. If there is such a cut $F$, let the two connected components of $G-F$ be $G_1$ and $G_2$, and we obtain a 3-edge-coloring by applying this algorithm to both $G_i$ recursively.
If we fail to three-edge color one of $G_i$, we obtain an obstruction and output it.

Hereafter we assume that $G$ is a cyclically 5-edge-connected graph. If there is a cut $F$ of size five, where both connected components are not 5-cycles, again, 
we obtain a three-edge coloring by applying this algorithm to both $G_i$ recursively (we never fail because we do not reduce the obstructions). 
As in Lemmas \ref{minc}, \ref{lem:(2,2)-annulus-cut} and \ref{lem:(1,3)-annulus-cut}, we can obtain the coloring of $G$ by merging the edge-colorings of $G_1$ and $G_2$.

\medskip

{\bf Step 3:} Find a reducible configuration in $\mathcal{K}$ 

\medskip

 By Lemmas \ref{T(v)>0} and \ref{T(v)>0T},  a reducible configuration in $\mathcal{K}$ weakly appears. We give more details in Appendix (Section \ref{secalgoappen}).  

\medskip

Now let us estimate the time complexity of this algorithm. Step 1 takes $O(n)$ time as in \cite{MOHARlinear}. Both steps 2 and 3 take an $O(n)$ time and we recuse. For correctness, we follow the proof in \cite{proj2024}, which is given in the appendix for the sake of completeness. 

\bibliographystyle{plain}
\bibliography{4CT}

% Start Appendices
%
\appendix
%

% ====================================================
\section{List of cases that a vertex sends charge 4,5,6}
\label{sect:sendscharge}\showlabel{sect:sendcharge}
% ====================================================

\tikzset{deg5/.style={thick, circle, draw, fill=black, inner sep=1.5pt,}}
\tikzset{deg6/.style={thick, circle, draw, fill=black, inner sep=0pt,}}
\tikzset{deg7/.style={thick, circle, draw, fill=white, inner sep=2pt,}}
\tikzset{deg8/.style={thick, rectangle, draw, fill=white, inner sep=2pt,}}
\tikzset{deg9/.style={thick, regular polygon, regular polygon sides=3, rotate=180, draw, fill=white, inner sep=1pt,}}
\tikzset{deg10/.style={thick, regular polygon, draw, fill=white, inner sep=2pt,}}
\tikzset{->-/.style={decoration={
    markings,
    mark=at position .6 with {\arrow{>}}}, postaction={decorate}}}
\tikzset{->>-/.style={decoration={
    markings, 
    mark=at position .5 with {\arrow{>}};, 
    mark=at position .6 with {\arrow{>}};}, postaction={decorate}}}

\begin{figure}[htbp]
\begin{tabular}{ccccc}

\begin{minipage}[t]{0.2\hsize}
\centering
% (page, row, col) = (0, 0, 0)
\begin{tikzpicture} []
    \node [deg6] at (0.8, 2.151) (v0) {};
    \node [deg7] at (1.8, 2.151) (v1) {};
    \node [above = 0.15 cm of v1, anchor=center] (v1+) { $+$ };
    \node [deg5] at (1.3, 1.285) (v2) {};
    \node [deg5] at (0.3, 1.285) (v3) {};
    \node [deg5] at (0.3, 3.017) (v4) {};
    \node [deg5] at (1.3, 3.017) (v5) {};
    \node [deg5] at (1.126, 0.3) (v6) {};
    \node [deg5] at (1.126, 4.002) (v7) {};
    \draw [->-] (v0) -- (v1);
    \foreach \u / \v in {v0/v1, v0/v2, v0/v3, v0/v4, v0/v5, v1/v2, v1/v5, v2/v3, v2/v6, v3/v6, v4/v5, v4/v7, v5/v7}
        \draw (\u) -- (\v);
    \foreach \u / \v in {}
        \draw[ultra thick, blue] (\u) -- (\v);
\end{tikzpicture}
\end{minipage}
&
\begin{minipage}[t]{0.2\hsize}
\centering
% (page, row, col) = (0, 0, 1)
\begin{tikzpicture} []
    \node [deg6] at (0.8, 2.151) (v0) {};
    \node [deg7] at (1.8, 2.151) (v1) {};
    \node [above = 0.15 cm of v1, anchor=center] (v1+) { $+$ };
    \node [deg5] at (1.3, 1.285) (v2) {};
    \node [deg5] at (0.3, 1.285) (v3) {};
    \node [deg5] at (0.3, 3.017) (v4) {};
    \node [deg5] at (1.3, 3.017) (v5) {};
    \node [deg5] at (1.126, 0.3) (v6) {};
    \node [deg6] at (1.126, 4.002) (v7) {};
    \node [deg5] at (2.24, 3.359) (v8) {};
    \draw [->-] (v0) -- (v1);
    \foreach \u / \v in {v0/v1, v0/v2, v0/v3, v0/v4, v0/v5, v1/v2, v1/v5, v1/v8, v2/v3, v2/v6, v3/v6, v4/v5, v4/v7, v5/v7, v5/v8, v7/v8}
        \draw (\u) -- (\v);
    \foreach \u / \v in {}
        \draw[ultra thick, blue] (\u) -- (\v);
\end{tikzpicture}
\end{minipage}
&
\begin{minipage}[t]{0.2\hsize}
\centering
% (page, row, col) = (0, 0, 2)
\begin{tikzpicture} []
    \node [deg6] at (0.8, 2.151) (v0) {};
    \node [deg7] at (1.8, 2.151) (v1) {};
    \node [above = 0.15 cm of v1, anchor=center] (v1+) { $+$ };
    \node [deg5] at (1.3, 1.285) (v2) {};
    \node [deg5] at (0.3, 1.285) (v3) {};
    \node [deg5] at (0.3, 3.017) (v4) {};
    \node [deg5] at (1.3, 3.017) (v5) {};
    \node [deg5] at (2.24, 0.943) (v6) {};
    \node [deg6] at (1.126, 0.3) (v7) {};
    \node [deg6] at (1.455, 4.434) (v8) {};
    \node [deg5] at (2.187, 3.377) (v9) {};
    \draw [->-] (v0) -- (v1);
    \foreach \u / \v in {v0/v1, v0/v2, v0/v3, v0/v4, v0/v5, v1/v2, v1/v5, v1/v6, v1/v9, v2/v3, v2/v6, v2/v7, v3/v7, v4/v5, v4/v8, v5/v8, v5/v9, v6/v7, v8/v9}
        \draw (\u) -- (\v);
    \foreach \u / \v in {}
        \draw[ultra thick, blue] (\u) -- (\v);
\end{tikzpicture}
\end{minipage}
&
\begin{minipage}[t]{0.2\hsize}
\centering
% (page, row, col) = (0, 0, 3)
\begin{tikzpicture} []
    \node [deg6] at (0.8, 2.032) (v0) {};
    \node [deg7] at (1.8, 2.032) (v1) {};
    \node [above = 0.15 cm of v1, anchor=center] (v1+) { $+$ };
    \node [deg5] at (1.3, 2.898) (v2) {};
    \node [deg5] at (0.3, 2.898) (v3) {};
    \node [deg5] at (0.3, 1.166) (v4) {};
    \node [deg6] at (1.3, 1.166) (v5) {};
    \node [deg5] at (1.126, 3.883) (v6) {};
    \node [deg5] at (0.8, 0.3) (v7) {};
    \node [deg5] at (1.8, 0.3) (v8) {};
    \draw [->-] (v0) -- (v1);
    \foreach \u / \v in {v0/v1, v0/v2, v0/v3, v0/v4, v0/v5, v1/v2, v1/v5, v2/v3, v2/v6, v3/v6, v4/v5, v4/v7, v5/v7, v5/v8, v7/v8}
        \draw (\u) -- (\v);
    \foreach \u / \v in {}
        \draw[ultra thick, blue] (\u) -- (\v);
\end{tikzpicture}
\end{minipage}
&
\begin{minipage}[t]{0.2\hsize}
\centering
% (page, row, col) = (0, 0, 4)
\begin{tikzpicture} []
    \node [deg6] at (0.8, 2.151) (v0) {};
    \node [deg7] at (1.8, 2.151) (v1) {};
    \node [above = 0.15 cm of v1, anchor=center] (v1+) { $+$ };
    \node [deg5] at (1.3, 1.285) (v2) {};
    \node [deg5] at (0.3, 1.285) (v3) {};
    \node [deg6] at (0.3, 3.017) (v4) {};
    \node [deg5] at (1.3, 3.017) (v5) {};
    \node [deg5] at (1.126, 0.3) (v6) {};
    \node [deg5] at (1.126, 4.002) (v7) {};
    \node [deg5] at (2.24, 3.359) (v8) {};
    \draw [->-] (v0) -- (v1);
    \foreach \u / \v in {v0/v1, v0/v2, v0/v3, v0/v4, v0/v5, v1/v2, v1/v5, v1/v8, v2/v3, v2/v6, v3/v6, v4/v5, v4/v7, v5/v7, v5/v8, v7/v8}
        \draw (\u) -- (\v);
    \foreach \u / \v in {}
        \draw[ultra thick, blue] (\u) -- (\v);
\end{tikzpicture}
\end{minipage}
\\
\begin{minipage}[t]{0.2\hsize}
\centering
% (page, row, col) = (0, 1, 0)
\begin{tikzpicture} []
    \node [deg6] at (0.8, 2.151) (v0) {};
    \node [deg7] at (1.8, 2.151) (v1) {};
    \node [above = 0.15 cm of v1, anchor=center] (v1+) { $+$ };
    \node [deg5] at (1.3, 1.285) (v2) {};
    \node [deg5] at (0.3, 1.285) (v3) {};
    \node [deg6] at (0.3, 3.017) (v4) {};
    \node [deg6] at (1.3, 3.017) (v5) {};
    \node [deg5] at (1.126, 0.3) (v6) {};
    \node [deg5] at (0.8, 3.883) (v7) {};
    \node [deg5] at (2.3, 3.017) (v8) {};
    \node [deg5] at (1.8, 3.883) (v9) {};
    \draw [->-] (v0) -- (v1);
    \foreach \u / \v in {v0/v1, v0/v2, v0/v3, v0/v4, v0/v5, v1/v2, v1/v5, v1/v8, v2/v3, v2/v6, v3/v6, v4/v5, v4/v7, v5/v7, v5/v8, v5/v9, v7/v9, v8/v9}
        \draw (\u) -- (\v);
    \foreach \u / \v in {}
        \draw[ultra thick, blue] (\u) -- (\v);
\end{tikzpicture}
\end{minipage}
&
\begin{minipage}[t]{0.2\hsize}
\centering
% (page, row, col) = (0, 1, 1)
\begin{tikzpicture} []
    \node [deg6] at (0.8, 2.583) (v0) {};
    \node [deg7] at (1.8, 2.583) (v1) {};
    \node [above = 0.15 cm of v1, anchor=center] (v1+) { $+$ };
    \node [deg5] at (1.3, 1.717) (v2) {};
    \node [deg5] at (0.3, 1.717) (v3) {};
    \node [deg6] at (0.3, 3.449) (v4) {};
    \node [deg5] at (1.3, 3.449) (v5) {};
    \node [deg5] at (2.187, 1.357) (v6) {};
    \node [deg6] at (1.455, 0.3) (v7) {};
    \node [deg5] at (1.126, 4.434) (v8) {};
    \node [deg5] at (2.24, 3.791) (v9) {};
    \draw [->-] (v0) -- (v1);
    \foreach \u / \v in {v0/v1, v0/v2, v0/v3, v0/v4, v0/v5, v1/v2, v1/v5, v1/v6, v1/v9, v2/v3, v2/v6, v2/v7, v3/v7, v4/v5, v4/v8, v5/v8, v5/v9, v6/v7, v8/v9}
        \draw (\u) -- (\v);
    \foreach \u / \v in {}
        \draw[ultra thick, blue] (\u) -- (\v);
\end{tikzpicture}
\end{minipage}
&
\begin{minipage}[t]{0.2\hsize}
\centering
% (page, row, col) = (0, 1, 2)
\begin{tikzpicture} []
    \node [deg6] at (0.8, 2.151) (v0) {};
    \node [deg7] at (1.8, 2.151) (v1) {};
    \node [above = 0.15 cm of v1, anchor=center] (v1+) { $+$ };
    \node [deg5] at (1.3, 1.285) (v2) {};
    \node [deg6] at (0.3, 1.285) (v3) {};
    \node [deg6] at (0.3, 3.017) (v4) {};
    \node [deg5] at (1.3, 3.017) (v5) {};
    \node [deg5] at (2.24, 0.943) (v6) {};
    \node [deg5] at (1.126, 0.3) (v7) {};
    \node [deg5] at (1.455, 4.434) (v8) {};
    \node [deg5] at (2.187, 3.377) (v9) {};
    \draw [->-] (v0) -- (v1);
    \foreach \u / \v in {v0/v1, v0/v2, v0/v3, v0/v4, v0/v5, v1/v2, v1/v5, v1/v6, v1/v9, v2/v3, v2/v6, v2/v7, v3/v7, v4/v5, v4/v8, v5/v8, v5/v9, v6/v7, v8/v9}
        \draw (\u) -- (\v);
    \foreach \u / \v in {}
        \draw[ultra thick, blue] (\u) -- (\v);
\end{tikzpicture}
\end{minipage}
&
\begin{minipage}[t]{0.2\hsize}
\centering
% (page, row, col) = (0, 1, 3)
\begin{tikzpicture} []
    \node [deg7] at (1.201, 1.95) (v0) {};
    \node [deg7] at (2.201, 1.95) (v1) {};
    \node [above = 0.15 cm of v1, anchor=center] (v1+) { $+$ };
    \node [deg5] at (1.824, 1.168) (v2) {};
    \node [deg5] at (0.978, 0.975) (v3) {};
    \node [deg7] at (0.3, 1.516) (v4) {};
    \node [above = 0.15 cm of v4, anchor=center] (v4+) { $+$ };
    \node [deg7] at (0.3, 2.383) (v5) {};
    \node [above = 0.15 cm of v5, anchor=center] (v5+) { $+$ };
    \node [deg5] at (0.978, 2.925) (v6) {};
    \node [deg5] at (1.824, 2.731) (v7) {};
    \node [deg5] at (1.824, 0.3) (v8) {};
    \node [deg5] at (1.355, 3.706) (v9) {};
    \draw [->-] (v0) -- (v1);
    \foreach \u / \v in {v0/v1, v0/v2, v0/v3, v0/v4, v0/v5, v0/v6, v0/v7, v1/v2, v1/v7, v2/v3, v2/v8, v3/v4, v3/v8, v4/v5, v5/v6, v6/v7, v6/v9, v7/v9}
        \draw (\u) -- (\v);
    \foreach \u / \v in {}
        \draw[ultra thick, blue] (\u) -- (\v);
\end{tikzpicture}
\end{minipage}
&
\begin{minipage}[t]{0.2\hsize}
\centering
% (page, row, col) = (0, 1, 4)
\begin{tikzpicture} []
    \node [deg7] at (2.085, 1.275) (v0) {};
    \node [deg7] at (3.085, 1.275) (v1) {};
    \node [above = 0.15 cm of v1, anchor=center] (v1+) { $+$ };
    \node [deg5] at (2.709, 0.493) (v2) {};
    \node [deg5] at (1.863, 0.3) (v3) {};
    \node [deg5] at (1.184, 1.709) (v4) {};
    \node [deg6] at (1.863, 2.25) (v5) {};
    \node [deg6] at (2.709, 2.057) (v6) {};
    \node [deg5] at (0.3, 2.167) (v7) {};
    \node [deg5] at (1.065, 2.512) (v8) {};
    \node [deg5] at (2.422, 2.876) (v9) {};
    \node [deg5] at (3.571, 1.96) (v10) {};
    \node [deg5] at (3.25, 2.735) (v11) {};
    \draw [->-] (v0) -- (v1);
    \foreach \u / \v in {v0/v1, v0/v2, v0/v3, v0/v4, v0/v5, v0/v6, v1/v2, v1/v6, v1/v10, v2/v3, v4/v5, v4/v7, v4/v8, v5/v6, v5/v8, v5/v9, v6/v9, v6/v10, v6/v11, v7/v8, v9/v11, v10/v11}
        \draw (\u) -- (\v);
    \foreach \u / \v in {}
        \draw[ultra thick, blue] (\u) -- (\v);
\end{tikzpicture}
\end{minipage}
\\
\end{tabular}
\caption{A vertex sends charge 6 in these cases.}
\label{fig:send6}
\end{figure}
\subfile{tikz/send5}
\tikzset{deg5/.style={thick, circle, draw, fill=black, inner sep=1.5pt,}}
\tikzset{deg6/.style={thick, circle, draw, fill=black, inner sep=0pt,}}
\tikzset{deg7/.style={thick, circle, draw, fill=white, inner sep=2pt,}}
\tikzset{deg8/.style={thick, rectangle, draw, fill=white, inner sep=2pt,}}
\tikzset{deg9/.style={thick, regular polygon, regular polygon sides=3, rotate=180, draw, fill=white, inner sep=1pt,}}
\tikzset{deg10/.style={thick, regular polygon, draw, fill=white, inner sep=2pt,}}
\tikzset{->-/.style={decoration={
    markings,
    mark=at position .6 with {\arrow{>}}}, postaction={decorate}}}
\tikzset{->>-/.style={decoration={
    markings, 
    mark=at position .5 with {\arrow{>}};, 
    mark=at position .6 with {\arrow{>}};}, postaction={decorate}}}

\begin{figure}[htbp]
\begin{tabular}{ccccc}

\begin{minipage}[t]{0.2\hsize}
\centering
% (page, row, col) = (0, 0, 0)
\begin{tikzpicture} []
    \node [deg5] at (1.109, 0.3) (v0) {};
    \node [deg7] at (2.109, 0.3) (v1) {};
    \node [above = 0.15 cm of v1, anchor=center] (v1+) { $+$ };
    \node [deg5] at (1.418, 1.251) (v2) {};
    \node [deg5] at (0.3, 0.888) (v3) {};
    \draw [->-] (v0) -- (v1);
    \foreach \u / \v in {v0/v1, v0/v2, v0/v3, v1/v2, v2/v3}
        \draw (\u) -- (\v);
    \foreach \u / \v in {}
        \draw[ultra thick, blue] (\u) -- (\v);
\end{tikzpicture}
\end{minipage}
&
\begin{minipage}[t]{0.2\hsize}
\centering
% (page, row, col) = (0, 0, 1)
\begin{tikzpicture} []
    \node [deg5] at (1.109, 0.888) (v0) {};
    \node [deg7] at (2.109, 0.888) (v1) {};
    \node [above = 0.15 cm of v1, anchor=center] (v1+) { $+$ };
    \node [deg5] at (1.418, 1.839) (v2) {};
    \node [deg6] at (0.3, 1.476) (v3) {};
    \node [deg5] at (0.3, 0.3) (v4) {};
    \draw [->-] (v0) -- (v1);
    \foreach \u / \v in {v0/v1, v0/v2, v0/v3, v0/v4, v1/v2, v2/v3, v3/v4}
        \draw (\u) -- (\v);
    \foreach \u / \v in {}
        \draw[ultra thick, blue] (\u) -- (\v);
\end{tikzpicture}
\end{minipage}
&
\begin{minipage}[t]{0.2\hsize}
\centering
% (page, row, col) = (0, 0, 3)
\begin{tikzpicture} []
    \node [deg5] at (1.109, 0.3) (v0) {};
    \node [deg7] at (2.109, 0.3) (v1) {};
    \node [above = 0.15 cm of v1, anchor=center] (v1+) { $+$ };
    \node [deg5] at (1.418, 1.251) (v2) {};
    \node [deg6] at (0.3, 0.888) (v3) {};
    \node [deg5] at (2.436, 1.839) (v4) {};
    \node [deg6] at (0.94, 2.325) (v5) {};
    \draw [->-] (v0) -- (v1);
    \foreach \u / \v in {v0/v1, v0/v2, v0/v3, v1/v2, v1/v4, v2/v3, v2/v4, v2/v5, v3/v5, v4/v5}
        \draw (\u) -- (\v);
    \foreach \u / \v in {}
        \draw[ultra thick, blue] (\u) -- (\v);
\end{tikzpicture}
\end{minipage}
&
\begin{minipage}[t]{0.2\hsize}
\centering
% (page, row, col) = (0, 0, 2)
\begin{tikzpicture} []
    \node [deg5] at (0.3, 1.251) (v0) {};
    \node [deg7] at (1.3, 1.251) (v1) {};
    \node [above = 0.15 cm of v1, anchor=center] (v1+) { $+$ };
    \node [deg5] at (0.609, 0.3) (v2) {};
    \node [deg5] at (0.609, 2.202) (v3) {};
    \draw [->-] (v0) -- (v1);
    \foreach \u / \v in {v0/v1, v0/v2, v0/v3, v1/v2, v1/v3}
        \draw (\u) -- (\v);
    \foreach \u / \v in {}
        \draw[ultra thick, blue] (\u) -- (\v);
\end{tikzpicture}
\end{minipage}
&
\begin{minipage}[t]{0.2\hsize}
\centering
% (page, row, col) = (0, 0, 4)
\begin{tikzpicture} []
    \node [deg5] at (1.109, 1.251) (v0) {};
    \node [deg7] at (2.109, 1.251) (v1) {};
    \node [above = 0.15 cm of v1, anchor=center] (v1+) { $+$ };
    \node [deg5] at (1.418, 0.3) (v2) {};
    \node [deg5] at (0.3, 1.839) (v3) {};
    \node [deg6] at (1.418, 2.202) (v4) {};
    \draw [->-] (v0) -- (v1);
    \foreach \u / \v in {v0/v1, v0/v2, v0/v3, v0/v4, v1/v2, v1/v4, v3/v4}
        \draw (\u) -- (\v);
    \foreach \u / \v in {}
        \draw[ultra thick, blue] (\u) -- (\v);
\end{tikzpicture}
\end{minipage}
\\
\begin{minipage}[t]{0.2\hsize}
\centering
% (page, row, col) = (0, 1, 0)
\begin{tikzpicture} []
    \node [deg5] at (1.109, 1.251) (v0) {};
    \node [deg7] at (2.109, 1.251) (v1) {};
    \node [above = 0.15 cm of v1, anchor=center] (v1+) { $+$ };
    \node [deg5] at (1.418, 0.3) (v2) {};
    \node [deg6] at (0.3, 1.839) (v3) {};
    \node [deg6] at (1.418, 2.202) (v4) {};
    \node [deg6] at (0.587, 3.033) (v5) {};
    \node [deg5] at (1.781, 3.32) (v6) {};
    \draw [->-] (v0) -- (v1);
    \foreach \u / \v in {v0/v1, v0/v2, v0/v3, v0/v4, v1/v2, v1/v4, v3/v4, v3/v5, v4/v5, v4/v6, v5/v6}
        \draw (\u) -- (\v);
    \foreach \u / \v in {}
        \draw[ultra thick, blue] (\u) -- (\v);
\end{tikzpicture}
\end{minipage}
&
\begin{minipage}[t]{0.2\hsize}
\centering
% (page, row, col) = (0, 1, 1)
\begin{tikzpicture} []
    \node [deg5] at (1.109, 1.251) (v0) {};
    \node [deg7] at (2.109, 1.251) (v1) {};
    \node [above = 0.15 cm of v1, anchor=center] (v1+) { $+$ };
    \node [deg6] at (1.418, 0.3) (v2) {};
    \node [deg5] at (0.3, 0.663) (v3) {};
    \node [deg5] at (0.3, 1.839) (v4) {};
    \node [deg6] at (1.418, 2.202) (v5) {};
    \draw [->-] (v0) -- (v1);
    \foreach \u / \v in {v0/v1, v0/v2, v0/v3, v0/v4, v0/v5, v1/v2, v1/v5, v2/v3, v3/v4, v4/v5}
        \draw (\u) -- (\v);
    \foreach \u / \v in {}
        \draw[ultra thick, blue] (\u) -- (\v);
\end{tikzpicture}
\end{minipage}
&
\end{tabular}
\caption{A vertex of degree $5$ sends charge 4 in these cases.}
\label{fig:deg5send4}
\end{figure}
\clearpage

% ====================================================
\section{List of discharging rules}
\label{sect:rule}\showlabel{sect:rule}
% ====================================================
\subfile{tikz/rule}
\clearpage

% ====================================================
\section{Appendix: Pseudo codes}
\label{sect:code}\showlabel{sect:code}
% ====================================================

\subsection{Pseudocode for checking cuts of size 6 or 7}
\label{subsect:code-cuts6,7}\showlabel{subsect:code-cuts6,7}

We explain a pseudocode that checks Claim \ref{clm:not-Petersen-like2-cut6,7} in Section \ref{sect:6,7-cut}. Let $K$ be a configuration in $\mathcal{K}$, and $S$ be a free completion of $K$ with ring $R$. Let us remind $c(K)$ is a set of edges to be contracted to prove C-reducbility. When checking the combinations of distances in (6cut-$k$) ($1 \leq k \leq 10$), (7cut-$l$) ($1 \leq l \leq 15$), we check the following patterns of distances for each configuration $K$ in the set $\mathcal{K}$. 
These precise check patterns for each configuration are rather complicated as shown below. This is because we have to check two patterns when the distance of two vertices of $C$ is 1: two vertices $r, r'$ of the ring of $K$ satisfy (i) $d_{S / c(K)}(r, r') = 1$ and both $r, r'$ belong to $C$ or, (ii) $d_{S / c(K)}(r, r') = 0$ and $r$ belong to $C$ and $r'$ is adjacent to a vertex of $C$. In the patterns below, to simplify notations when the symbols $r, r', r''$ are used, they unconditionally represent vertices that do not belong to $C$ but belong to the ring $R$.
\begin{itemize}
    \item(6cut-1)
    \begin{itemize}
        \item $d_{S / c(K)}(u_0, u_2) = 0$
    \end{itemize}
    \item(6cut-2)
    \begin{itemize}
        \item $d_{S / c(K)}(u_0, u_2) = 0$, $d_{S / c(K)}(u_3, u_5) = 0$
    \end{itemize}
    \item(6cut-3)
    \begin{itemize}
        \item $d_{S / c(K)}(u_0, u_2) = 0$, $d_{S / c(K)}(u_2, u_4) = 0$
    \end{itemize}
    \item(6cut-4)
    \begin{itemize}
        \item $d_{S / c(K)}(u_0, u_2) = 0$, $d_{S / c(K)}(u_3, u_5) = 1$
        \item $d_{S / c(K)}(u_0, u_2) = 0$, $d_{S / c(K)}(u_3, r) = 0$, where $r$ is adjacent to $u_5$.
    \end{itemize}
    \item(6cut-5)
    \begin{itemize}
        \item $d_{S / c(K)}(u_0, u_2) = 0$, $d_{S / c(K)}(u_2, u_4) = 1$
        \item $d_{S / c(K)}(u_0, u_2) = 0$, $d_{S / c(K)}(u_2, r) = 0$, where $r$ is adjacent to $u_4$.
    \end{itemize}
    \item(6cut-6)
    \begin{itemize}
        \item $d_{S / c(K)}(u_0, u_3) = 0$
    \end{itemize}
    \item(6cut-7)
    \begin{itemize}
        \item $d_{S / c(K)}(u_0, u_2) = 1$
        \item $d_{S / c(K)}(u_0, r) = 0$, where $r$ is adjacent to $u_2$.
    \end{itemize}
    \item(6cut-8)
    \begin{itemize}
        \item $d_{S / c(K)}(u_0, u_2) = 1$, $d_{S / c(K)}(u_3, u_5) = 1$
        \item $d_{S / c(K)}(u_0, u_2) = 1$, $d_{S / c(K)}(u_3, r) = 0$, where $r$ is adjacent to $u_5$.
        \item $d_{S / c(K)}(r, u_2) = 0$, $d_{S / c(K)}(u_3, r') = 0$, where $r$ is adjacent to $u_0$ and $r'$ is adjacent to $u_5$.
        \item $d_{S / c(K)}(u_0, r) = 0$, $d_{S / c(K)}(u_3, r') = 0$, where $r$ is adjacent to $u_2$ and $r'$ is adjacent to $u_5$.
    \end{itemize}
    \item(6cut-9)
    \begin{itemize}
        \item $d_{S / c(K)}(u_0, u_2) = 1$, $d_{S / c(K)}(u_2, u_4) = 1$
        \item $d_{S / c(K)}(u_0, u_2) = 1$, $d_{S / c(K)}(u_2, r) = 0$, where $r$ is adjacent to $u_4$.
        \item $d_{S / c(K)}(r, u_2) = 0$, $d_{S / c(K)}(u_2, r') = 0$, where $r$ is adjacent to $u_0$ and $r'$ is adjacent to $u_4$.
        \item $d_{S / c(K)}(u_0, r) = 0$, $d_{S / c(K)}(r', u_4) = 0$, where $r, r'$ are adjacent to $u_2$.
    \end{itemize}
    \item(6cut-10)
    \begin{itemize}
        \item $d_{S / c(K)}(u_0, u_2) = 1$, $d_{S / c(K)}(u_2, u_4) = 1$, $d_{S / c(K)}(u_4, u_0) = 1$
        \item $d_{S / c(K)}(r, u_2) = 0$, $d_{S / c(K)}(u_2, u_4) = 1$, $d_{S / c(K)}(u_4, r') = 0$, where $r,r'$ are adjacent to $u_0$.
    \end{itemize}
    \item(7cut-1)
    \begin{itemize}
        \item $d_{S / c(K)}(u_0, u_2) = 0$
    \end{itemize}
    \item(7cut-2)
    \begin{itemize}
        \item $d_{S / c(K)}(u_0, u_2) = 0$, $d_{S / c(K)}(u_3, r_6) = 0$
    \end{itemize}
    \item(7cut-3)
    \begin{itemize}
        \item $d_{S / c(K)}(u_0, u_2) = 0$, $d_{S / c(K)}(u_3, u_5) = 0$
    \end{itemize}
    \item(7cut-4)
    \begin{itemize}
        \item $d_{S / c(K)}(u_0, u_2) = 0$, $d_{S / c(K)}(u_2, u_4) = 0$
    \end{itemize}
    \item(7cut-5)
    \begin{itemize}
        \item $d_{S / c(K)}(u_0, u_2) = 0$, $d_{S / c(K)}(u_2, u_4) = 1$
        \item $d_{S / c(K)}(u_0, u_2) = 0$, $d_{S / c(K)}(u_2, r) = 0$, where $r$ is adjacent to $u_4$.
    \end{itemize}
    \item(7cut-6)
    \begin{itemize}
        \item $d_{S / c(K)}(u_0, u_2) = 0$, $d_{S / c(K)}(u_3, u_5) = 1$
        \item $d_{S / c(K)}(u_0, u_2) = 0$, $d_{S / c(K)}(u_3, r) = 0$, where $r$ is adjacent to $u_5$.
        \item $d_{S / c(K)}(u_0, u_2) = 0$, $d_{S / c(K)}(r, u_5) = 0$, where $r$ is adjacent to $u_3$.
    \end{itemize}
    \item(7cut-7)
    \begin{itemize}
        \item $d_{S / c(K)}(u_0, u_2) = 0$, $d_{S / c(K)}(u_2, u_4) = 1$,  $d_{S / c(K)}(u_4, r_6) = 1$
        \item $d_{S / c(K)}(u_0, u_2) = 0$, $d_{S / c(K)}(u_2, u_4) = 1$,  $d_{S / c(K)}(u_4, r) = 0$, where 
        $r$ is adjacent to $u_6$.
        \item $d_{S / c(K)}(u_0, u_2) = 0$, $d_{S / c(K)}(u_2, r) = 0$,  $d_{S / c(K)}(r', r_6) = 0$, where $r, r'$ are adjacent to $u_4$.
    \end{itemize}
    \item(7cut-8)
    \begin{itemize}
        \item $d_{S / c(K)}(u_0, u_2) = 0$, $d_{S / c(K)}(u_2, u_4) = 1$,  $d_{S / c(K)}(u_5, u_0) = 1$
        \item $d_{S / c(K)}(u_0, u_2) = 0$, $d_{S / c(K)}(u_2, u_4) = 1$,  $d_{S / c(K)}(r, u_0) = 0$, where $r$ is adjacent to $u_5$.
        \item $d_{S / c(K)}(u_0, u_2) = 0$, $d_{S / c(K)}(u_2, r) = 0$,  $d_{S / c(K)}(r', u_0) = 0$, where 
        $r$ is adjacent to $u_4$ and $r'$ is adjacent to $u_5$.
    \end{itemize}
    \item(7cut-9)
    \begin{itemize}
        \item $d_{S / c(K)}(u_0, u_3) = 0$
    \end{itemize}
    \item(7cut-10)
    \begin{itemize}
        \item $d_{S / c(K)}(u_0, u_3) = 0$, $d_{S / c(K)}(u_3, u_5) = 1$
        \item $d_{S / c(K)}(u_0, u_3) = 0$, $d_{S / c(K)}(u_3, r) = 0$, where $r$ is adjacent to $u_5$.
    \end{itemize}
    \item(7cut-11)
    \begin{itemize}
        \item $d_{S / c(K)}(u_0, u_3) = 0$, $d_{S / c(K)}(u_4, u_6) = 1$
        \item $d_{S / c(K)}(u_0, u_3) = 0$, $d_{S / c(K)}(u_4, r) = 0$, where $r$ is adjacent to $u_6$.
    \end{itemize}
    \item(7cut-12)
    \begin{itemize}
        \item $d_{S / c(K)}(u_0, u_2) = 1$
        \item $d_{S / c(K)}(u_0, r) = 0$, where $r$ is adjacent to $u_2$.
    \end{itemize}
    \item(7cut-13)
    \begin{itemize}
        \item $d_{S / c(K)}(u_0, u_2) = 1$, $d_{S / c(K)}(u_2, u_4) = 1$
        \item $d_{S / c(K)}(u_0, u_2) = 1$, $d_{S / c(K)}(u_2, r) = 0$, where $r$ is adjacent to $u_4$
        \item $d_{S / c(K)}(r, u_2) = 0$, $d_{S / c(K)}(u_2, r') = 0$, where $r$ is adjacent to $u_0$ and $r'$ is adjacent to $u_4$.
        \item $d_{S / c(K)}(u_0, r) = 0$, $d_{S / c(K)}(r', u_4) = 0$, where $r, r'$ are adjacent to $u_2$.
    \end{itemize}
    \item(7cut-14)
    \begin{itemize}
        \item $d_{S / c(K)}(u_0, u_2) = 1$, $d_{S / c(K)}(u_3, u_5) = 1$
        \item $d_{S / c(K)}(u_0, u_2) = 1$, $d_{S / c(K)}(u_3, r) = 0$, where $r$ is adjacent to $u_5$.
        \item $d_{S / c(K)}(u_0, u_2) = 1$, $d_{S / c(K)}(r, u_5) = 0$, where $r$ is adjacent to $u_3$.
        \item $d_{S / c(K)}(r, u_2) = 0$, $d_{S / c(K)}(u_3, r') = 0$, where $r$ is adjacent to $u_0$ and $r'$ is adjacent to $u_5$.
        \item $d_{S / c(K)}(u_0, r) = 0$, $d_{S / c(K)}(r', u_5) = 0$, where $r$ is adjacent to $u_2$ and $r'$ is adjacent to $u_3$.
        \item $d_{S / c(K)}(r, u_2) = 0$, $d_{S / c(K)}(r', u_5) = 0$, where $r$ is adjacent to $u_0$ and $r'$ is adjacent to $u_3$.
    \end{itemize}
    \item(7cut-15)
    \begin{itemize}
        \item $d_{S / c(K)}(u_0, u_2) = 1$, $d_{S / c(K)}(u_2, u_4) = 1$, $d_{S / c(K)}(u_4, r_6) = 1$
        \item $d_{S / c(K)}(u_0, u_2) = 1$, $d_{S / c(K)}(u_2, u_4) = 1$, $d_{S / c(K)}(u_4, r) = 0$, where $r$ is adjacent to $u_6$.
        \item $d_{S / c(K)}(r, u_2) = 0$, $d_{S / c(K)}(u_2, u_4) = 1$, $d_{S / c(K)}(u_4, r') = 0$, where $r$ is adjacent to $u_0$ and $r'$ is adjacent to $u_6$.
        \item $d_{S / c(K)}(u_0, u_2) = 1$, $d_{S / c(K)}(u_2, r) = 0$, $d_{S / c(K)}(r', u_6) = 0$, where $r, r'$ are adjacent to $u_4$.
        \item $d_{S / c(K)}(r, u_2) = 1$, $d_{S / c(K)}(u_2, r') = 0$, $d_{S / c(K)}(r'', r_6) = 0$, where $r$ is adjacent to $u_0$ and $r', r''$ are adjacent to $u_4$.
    \end{itemize}
\end{itemize}

\begin{algorithm}[htbp]
    \caption{ReducableVerticesByTwoNonContractiblePaths($K$, $S$, $R$, $l$)}
    \label{alg:reducableverticesbytwononcontractiblepaths}
    \begin{algorithmic}[1]
        \Require {A C-reducible configuration $K$, its free completion $S$ and its ring $R$, the length of $C$ is $l$}
        \Ensure {Returns the set of vertices that may be deleted due to two non-contractibly connected paths of $R$.}
        \State $d:$ a function, receives two vertices in $S$ and returns the distance between them
        \State $d':$ a function, receives two vertices in $S / c(K)$ and returns the distance between them
        \State $rRr':$ the path obtained by moving along $R$ in clockwise order from $r \in R$ to $r' \in R$.
        \State $U \gets \varnothing$
        \ForAll {$r_p^1,r_q^1,r_p^2,r_q^2 \in R$ such that $r_p^1, r_q^1, r_p^2, r_q^2$ are listed in clockwise order of $R$}
            \State \Comment{the representativity after contraction is at least 2 so the lowerbounds of $c_1, c_2$ exist.}
            \ForAll {$c_1$ in $\max \{0, 2 - d'(r_p^1, r_q^1)\}, \ldots, 3 - d'(r_q^1, r_p^2) - d'(r_q^2, r_p^1)$}
                \ForAll {$c_2$ in $\max \{0, 2 - d'(r_p^2, r_q^2)\} , \ldots, 3 - c_1 - d'(r_q^1, r_p^2) - d'(r_q^2, r_p^1)$}
                    \If {\text{calcLowerBoundLength($K$, $S$, $R$, $l$, $r_p^1$, $r_q^1$, $r_p^2$, $r_q^2$, $c_1$, $c_2$)} $> l$}
                        \Continue
                    \EndIf
                    \State \texttt{hasSmallCut} $\gets$ \False
                    \ForAll {$P_1 \gets$ all paths of $r_q^1$ and $r_p^2$ in $S$}
                        \ForAll {$P_2 \gets$ all paths of $r_q^2$ and $r_p^1$ in $S$}
                            \State $C_1 \gets$ the component of $S - P_1$ that contains midpoints of $r_q^1Rr_p^2$.
                            \State $C_2 \gets$ the component of $S - P_2$ that contains 
                            midpoints of $r_q^2Rr_p^1$.
                            \If {$|P_1| + |P_2| + c_1 + c_2 = 4, |C_1 \cup C_2| > 0$ or $|P_1| + |P_2| + c_1 + c_2 = 5, |C_1 \cup C_2| > 1$}
                                \State \texttt{hasSmallCut} $\gets$ \True
                                \Break
                            \EndIf
                        \EndFor
                    \EndFor
                    \If {\texttt{hasSmallCut} is \False}
                        \ForAll {$P_1 \gets$ all paths where $|P_1 / c(K)| = d'(r_q^1, r_p^2)$, connecting $r_q^1$ and $r_p^2$ in $S$}
                            \ForAll {$P_2 \gets$ all paths where $|P_2 / c(K)| = d'(r_q^2, r_p^1)$, connecting $r_q^2$ and $r_p^1$ in $S$}
                                \State $C_1 \gets$ the component of $S - P_1$ that contains midpoints of $r_q^1Rr_p^2$.
                                \State $C_2 \gets$ the component of $S - P_2$ that contains midpoints of $r_q^2Rr_p^1$.
                                \State $U \gets U \cup V(C_1) \cup V(C_2)$
                            \EndFor
                        \EndFor
                    \EndIf
                \EndFor
            \EndFor
        \EndFor
        \State \Return $U$
    \end{algorithmic}
\end{algorithm}

\begin{algorithm}[htbp]
    \caption{calcLowerBoundLength($K, S, R, l, r_p^1$, $r_q^1$, $r_p^2$, $r_q^2$, $c_1$, $c_2$)}
    \label{alg:calclowerboundlength}
    \begin{algorithmic}[1]
        \Require {A C-reducible configuration $K$, its free completion $S$ and its ring $R$, the length of $C$ is $l$, vertices of $R$, $r_p^1, r_q^1, r_p^2, r_q^2$ which are listed in clockwise order of $R$, the length of non-contractibly connected path of $R$ between $r_p^1, r_q^1$ (, $r_p^2, r_q^2$) is $c_1 \leq 3$ (, $c_2 \leq 3$) respectively.}
        \Ensure {The lower bound of length of a circuit that bounds the disk inside $K$.}
        \If {$c_1 = 3$ or $c_2 = 3$}
            \State \Return $0$ \Comment{return trivial lower bound.}
        \EndIf
        \ForAll {$r, r' \in R$}
            \State $L_1[r][r'] \gets$ the minimum possible length of a contractibly connected $rr'$-path of $R$ that is also a subpath of $C$. \Comment{calculated by the function forbiddenCycle in \cite{proj2024}}
            \State $x \gets$ the minimum possible length of a contractibly connected $rr'$-path of $R$ that consists of an edge whose one end is $r$ and a subpath of $C$ (its calculated length is the same when the path consists of an edge whose one end is $r'$ and a subpath of $C$). \Comment{calculated by the function forbiddenCycleOneEdge in \cite{proj2024}}
            \State $L_2[r][r'] \gets x - 1$
        \EndFor
        \State \Comment{check when all $r_p^1, r_q^1, r_p^2, r_q^2$ belong to $C$.}
        \State $L \gets \max(L_1[r_p^1][r_q^1], 2 - c_1) + \max(L_1[r_p^2][r_q^2], 2 - c_2)$ \Comment{the original representativity is at least two.}
        \State $L' \gets L_1[r_q^1][r_p^2] + L_1[r_q^2][r_p^1]$
        \If { $L + c_1 + c_2 \leq 5$ and $L' + c_1 + c_2 \leq 5$ } \Comment{the number of vertices outside $C$ is at least four}
            \State $L'' \gets L + L' + 6 - c_1 - c_2 - \max(L, L')$
        \Else 
            \State $L'' \gets L + L'$
        \EndIf
        \State 
        \If { $c_1 = 2$ }
            \State \Comment{check when $r_p^2, r_q^2$ and only one of $r_p^1$ or $r_q^1$ belong to $C$.}
            \State $L_a \gets \max(L_2[r_p^1][r_q^1], 1) + \max(L_1[r_p^2][r_q^2], 2 - c_1)$ \Comment{the original representativity is at least two.}
            \State $L_a' \gets \min(L_1[r_q^2][r_p^1] + L_2[r_q^1][r_p^2], L_2[r_q^2][r_p^1] + L_1[r_q^1][r_p^2])$
            \If { $L_a + c_2 + 1 \leq 5$ and $L_a' + c_2 + 1 \leq 5$ } \Comment{the number of vertices outside $C$ is at least four.}
                \State $L_a'' \gets L_a + L_a' + 5 - c_2 - \max(L_a, L_a')$
            \Else
                \State $L_a'' \gets L_a + L_a'$
            \EndIf
            \State $L'' \gets \min(L'', L_a'')$
            \If { $c_2 = 1$ }
            \State \Comment{check when $r_p^1, r_q^1$ and only one of $r_p^2$ or $r_q^2$ (twice) belong to $C$.}
                \State $L_b \gets \max(L_1[r_p^1][r_q^1], 2 - c_1) + \max(L_2[r_p^2][r_q^2], 2)$
                \State $L_b' \gets \min(L_1[r_q^2][r_p^1] + L_2[r_q^1][r_p^2], L_2[r_q^2][r_p^1] + L_1[r_q^1][r_p^2])$
                \If { $L_b + c_1 \leq 5$ and $L_b' + c_1 \leq 5$ }
                    \State $L_b'' \gets L_b + L_b' + 6 - c_1 - \max(L_b, L_b')$
                \Else
                    \State $L_b'' \gets L_b + L_b'$
                \EndIf
                \State $L'' \gets \min(L'', L_b'')$
            \EndIf
            % \State \Comment{check when $r_p^2, r_q^2$ and neither $r_p^1$ nor $r_q^1$ belong to $C$.}
            % \State $L_b \gets L_1[r_p^2][r_q^2] + 2$ \Comment{the original representativity is at least 2.}
            % \State $L_b' \gets L_2[r_q^1][r_p^2] + L_2[r_q^2][r_p^1]$
            % \If { $L_b + c_2 \leq 5$ and $L_b' + c_2 \leq 5$ } \Comment{the number of vertices outside $C$ is at least 3}
            %     \State $L_b'' \gets L_b + L_b' + 6 - c_2 - \max(L_b, L_b')$
            % \Else
            %     \State $L_b'' \gets L_b + L_b'$
            % \EndIf
            % \State $L'' \gets \min(L'', L_b'')$
        \EndIf
        \If { $c_2 = 2$ }
            \State... \Comment{perfom the same check when $c_1 = 2$ with the roles of $r_p^1, r_q^1$ and $r_p^2, r_q^2$ swapped, so we omit it.}
        \EndIf
        \State \Return $L''$
    \end{algorithmic}
\end{algorithm}

We have mainly two parts of implementing these checks. One part is checking the minimality of $C$ and the other part is counting the number of vertices as explained in Section \ref{sect:2,3-cycle-Petersen-like}. Note that the representativity of $G, G'$ are at least two in these algorithms since $G$ does not become Petersen-like otherwise.

The former part is almost identical to the pseudocode explained in Appendix of \cite{proj2024}, which are the functions named forbiddenCycle and forbiddenCycleOneEdge. The only difference between that pseudocode and ours is when checking the minimality, we can only use the "almost" minimality of $C$ by the definition of $C$. To be more precise, we have to check a cycle of length $7$ that would contradict the minimality contains at least four edges that are not shared with $C$ when the length of $C$ is 6. This is done by checking that cycle contains at least four edges whose at least one endpoint is a vertex of $K$. The other implementations of this part is the same so we omit it here. 

The latter part is almost identical to the pseudocode used to check the number of vertices after reductions used in \cite{proj2024}. 
The difference exists when checking a contractible cycle that consists of two paths inside $S$ and two homotopic noncontractibly connected paths of $R$, as explained in Section \ref{sect:2,3-cycle-Petersen-like}. Here, we check these paths are compatible with $C$, whose length is 6 or 7. We show the pseudocode in Algorithm \ref{alg:reducableverticesbytwononcontractiblepaths}. 

In Algorithm \ref{alg:reducableverticesbytwononcontractiblepaths}, we check whether the existence of short two noncontractibly connected paths of $R$ contradicts the length of cycle $C$ being 6, 7 or not. To do that, we calculate the lower bound of length of a circuit that bounds the disk inside $K$ in Algorithm \ref{alg:calclowerboundlength}.
Let $r_p^1, r_q^1, r_p^2, r_q^2$ be the vertices of the ring $R$ of a configuration $K \in \mathcal{K}$ such that $r_p^1, r_q^1, r_p^2, r_q^2$ are listed in the clockwise order of $R$. The symbol $P_1, P_2$ denote homotopic noncontractibly connected paths of the ring $R$ between $(r_p^1, r_q^1)$, $(r_p^2, r_q^2)$ respectively.
First, we consider the case where both of the lengths of $P_1, P_2$ are at most 1. We use the following claim that says $|P_1|=0,|P_2|=1$ (or $|P_1|=1,|P_2|=0$) is possible in the restricted situation.
\begin{claim}
\label{clm:(0,1)-noncontractible-paths}
\showlabel{clm:(0,1)-noncontractible-paths}
    We assume the followings:
    \begin{itemize}
        \item $u_i = u_j (0 \leq i < j \leq l)$ and $u_iCu_j$ (a subpath of $C$ from $u_i$ to $u_j$) is a noncontractible cycle in $G'$,
        \item a noncontractible connected chord $u_mu_n(j < m < n < l)$ of $C$ that is homotopic to $u_iCu_j$ exists, and
        \item the representativity of $G'$ is at least two.
    \end{itemize}
    Then, $l=7$ and $u_iu_mu_n$ is a facial triangle.
\end{claim}
\begin{proof}
    We obtained two regions by deleting $D$ and chord $u_m u_n$ from the torus. One is a disk and the other is an annulus or a disk. By the condition of $C$ ($C$ is balanced), these two regions have at least four vertices, but Lemma \ref{minc}, \ref{lem:(2,2)-annulus-cut} and the hypothesis that the representativiy is at least two contradicts it unless $l=7$ and $u_i u_m u_n$ is a facial cycle.
\end{proof}
By Claim \ref{clm:(0,1)-noncontractible-paths},
we know no paths in $S$ with $P_1,P_2$ erase vertices outside $D$ if $|P_1|=0, |P_2|=1$ (or $|P_1|=1, |P_2|=0$). Hence, if both of the length of $P_1, P_2$ are at most 1, all $r_p^1, r_q^1, r_p^2, r_q^2$ belong to $C$ since $C$ bounds the disk where $K$ appears. Each pair of vertices ($r_p^1, r_q^1$), ($r_q^1, r_p^2$), ($r_p^2, r_q^2$), ($r_q^2, r_p^1$) are connected by the subpath of $C$. We calculate the minimum length of such a possible path by using the function forbiddenCycle, which uses the almost minimality of $C$. The sum of these lengths represents the least number of length of a circuit that bounds $K$. We compare the length and the length of $C$, which is 6 or 7. We have one remark here.
The number of vertices outside the disk bounded by $C$ is at least four by the definition of $C$. This fact combined with Lemma \ref{minc} and \ref{lem:(2,2)-annulus-cut} implies not both two regions obtained by deleting $S$, $P_1$ and $P_2$ from the torus are surrounded at most five edges.
We can calculate a tighter lower bound by using it. We added comments when using these remarks in Algorithm \ref{alg:reducableverticesbytwononcontractiblepaths}, \ref{alg:calclowerboundlength}.
We do a similar check when the length of $P_1$ is 2 and the length of $P_2$ is at most 1. Unlike the case where the length of $P_1$ is at most 1, it is possible that $r_p^1, r_q^1$ do not belong to $C$ but the midvertex of $P_1$ belongs to $C$. In this case, we calculate the minimum length of the subpath between the midvertex of $P_1$ and $r_p^2$(or $r_q^2$) by using the function forbiddenCycleOneEdge, which also uses the almost minimality of $C$. 
Moreover, when the length of $P_2$ is one, $C$ may pass through $r_p^2$ (or $r_q^2$) twice. However, if it happens, $C$ does not pass through the midvertex of $P_1$ by Claim \ref{clm:(0,1)-noncontractible-paths}, so we do a similar check in this case.

\subsection{Pseudocode for enumerating wrappings of configurations}
\label{subsect:code-wrapping}\showlabel{subsect:code-wrapping}

We provide a pseudocode that enumerates all possible wrappings (defined in Subsection \ref{sect:weak-appear}) which is used to check Lemma \ref{lem:weakly}.

The toplevel function is \textsc{enumerateAllWrappings} in Algorithm \ref{alg:enumerateallwrappings}.
This function calls the following \textsc{identifyTwoDirectedEdges} function, which returns the general configuration after merging two vertices $a_1, b_1$ and the edges $a_1a_2, b_1b_2$.
All possible $a_1, b_1, a_2, b_2$ are tested, and the accumulated graphs are further passed into a recursive call of \textsc{enumerateAllWrappings}.
Finally, all graphs which were generated in this process are returned.

The \textsc{identifyTwoDirectedEdges} function in Algorithm \ref{alg:identifytwodirectededges} uses a Disjoint Set Union (DSU) which contains the methods Unite and Same.
This data structure is used to track whether there is a parallel path $P_1, P_2$ starting from $a_1b_1, a_2b_2$ respectively, and ending in some pair of vertices $u_1,u_2$, respectively.
In this case, DSU.Unite($u_1, u_2$) is called, and after this DSU.Same($u_1,u_2$) will return \textsf{True}.

\begin{algorithm}[h!]
    \caption{\textsc{enumerateAllWrappings}($K$)}
    \label{alg:enumerateallwrappings}
    \begin{algorithmic}[1]
        \Require {A general configuration $K$}
        \Ensure {A set of general configurations which are strict wrappings of $K$.}
        \State $S \gets$ the free completion of $K$ 
        \State $R \gets$ the ring of $S$.
        \State $\mathcal{W} \gets \varnothing$
        \ForAll{$a_1, a_2 \in V(S)$}
            \ForAll{$b_1 \in V(S)$ where $a_1b_1 \in E(S)$}
                \ForAll{$b_2 \in V(S)$ where $a_2b_2 \in E(S)$}
                    \If{$a_1, a_2, b_1, b_2$ are all in $R$}
                        \Continue \Comment{This wrapping is not strict}
                    \EndIf
                    \State $\mathcal{W}' \gets$ \textsc{identifyTwoDirectedEdges}($K$, $(a_1, b_1)$, $(a_2, b_2)$)
                    \State $\mathcal{W} \gets \mathcal{W} \cup \mathcal{W}'$
                    \ForAll{$K' \in \mathcal{W}'$}
                        \State $\mathcal{W} \gets \mathcal{W} \cup$ \textsc{enumerateAllWrappings}($K'$) \Comment{Recursively identify vertices}
                    \EndFor
                \EndFor
            \EndFor
        \EndFor
        \State \Return $\mathcal{W}$
    \end{algorithmic}
\end{algorithm}

\begin{algorithm}[htbp]
    \caption{\textsc{identifyTwoDirectedEdges}($K$, $(a_1, b_1)$, $(a_2, b_2)$)}
    \label{alg:identifytwodirectededges}
    \begin{algorithmic}[1]
        \Require {A configuration $K$, two pairs of adjacent vertices $(a_1, b_1)$ and $(a_2, b_2)$.}
        \Ensure {Either a singleton set containing a graph with the vertex pairs $a_1, a_2$ and $b_1, b_2$ identified, or an empty set if such a graph need not be considered.}
        \State $S \gets$ the free completion of $K$ 
        \State $R \gets$ the ring of $S$.
        \State $U \gets \varnothing$
        \State DSU $\gets$ a new Disjoint Set Union over $V(K)$
        \State \Comment{\textsc{OverlapDFS} is a function to search all parallel paths, and returns if the degree conditions meet.}
        \Function{OverlapDFS}{$a_1, b_1, a_2, b_2$}
            \If{$(a_1, b_1) \in U$}
                \State \Return \textsf{True}
            \EndIf
            \State $U \gets U \cup \{(a_1, b_1)\}$
            \State DSU.Unite($a_1, a_2$)
            \If{$a_1 = a_2$}
                \State \Return \textsf{True} if $b_1 = b_2$, \textsf{False} if $b_1 \neq b_2$
            \EndIf
            \If{$a_1, a_2 \notin R \land \gamma_K(a_1) \neq \gamma_K(a_2)$}
                \State \Return \textsf{False} \Comment{Degree mismatch}
            \EndIf
            \ForAll{$c_1, c_2$ s.t. the angle of $b_1$ in path $a_1b_1c_1$ and the angle of $b_2$ in $a_2b_2c_2$ is the same}
                \If{\textsc{OverlapDFS}($b_1, c_1, b_2, c_2$) is \textsf{False}}
                    \State \Return \textsf{False}
                \EndIf
            \EndFor
        \EndFunction

        \If{\textsc{OverlapDFS}$(a_1, b_1, a_2, b_2)$ is \textsf{False}}
            \State \Return $\varnothing$
        \EndIf
        \State $w \gets \min{d_K(u, v) \mid u,v \in V(K), u \neq v, \mathrm{DSU.Same}(u, v)}$ ($d_K(u,v)$ is the distance between $u,v$ in $K$)
        \If{$w \leq 1$}
            \State \Return $\varnothing$ \Comment{the representativity is $1$}
        \EndIf

        \State Define $a \sim b$ as DSU.Same($a, b$)
        \State $K' \gets$ a general configuration by merging all equivalent vertices w.r.t. $\sim$
        
        \State \Return $\{K'\}$
    \end{algorithmic}
\end{algorithm}

\subsection{Pseudocode for proving the avoidance of Petersen-like graphs after contraction}
\label{subsect:code-petlike-check}\showlabel{subsect:code-petlike-check}

We provide a pseudocode explained in \ref{subsect:avoidpet}; a code which tracks the degree of each vertex in $K$ after contraction, and checks if it does not match any structure of graphs in $\mathcal{T}_0$.
The function \textsc{matchNeighboringVertices} in Algorithm \ref{alg:matchneighboringvertices} is the toplevel function.

The first part of the algorithm is identical to the projective planar case \cite{proj2024} (the function ReducableVertices and its dependencies in Appendix C.5).
What this function did was to check for paths outside $K$ which may form a contractible cut, and see if it becomes a 2,3-cut reduction after contraction of $c(K)$.
Upon checking paths, we prune cases that violate Fact \ref{fact:567-cuts}.
This fact held in the projective planar case too, so the code implementation will be exactly the same.
Therefore, readers should refer to \cite{proj2024}.

After this, we use Table \ref{tab:neighbor-degree-pattern} to perform a degree matching of neighboring vertices.

\begin{algorithm}[htbp]
    \caption{\textsc{matchNeighboringVertices}($K$, $c(K)$)}
    \label{alg:matchneighboringvertices}
    \begin{algorithmic}[1]
        \Require {A configuration $K$ and a contraction edge set $c(K)$.}
        \Ensure {\textsf{True} if the contraction of $c(K)$ is safe, \textsf{False} otherwise.}
        \State $S,R \gets$ the free completion of $K$ and its ring
        \State $U \gets$ ReducableVertices($K,S,R$)
        \State $G' \gets G(K) - U$
        \State $G \gets G' / c(K)$
        \State Define $f(v)$ as \textsf{True} if $v \in S - R$, or $v$ is a contracted vertex of vertices which all are in $S - R$, \textsf{False} otherwise
        \ForAll{$v \in V(G)$ s.t. $f(v)$ is \textsf{True}}
            \State $d^* \gets$ the degree of $v$ in $G$
            \If {$d^* \leq 4$ or $d^* \geq 11$}
                \State \Return \textsf{True} \Comment{Degree mismatch}
            \EndIf
            \State $N_G \gets$ the induced subgraph of all neighboring vertices of $v$ where $f(v)$ is \textsf{True}
            \ForAll{maximal paths $P$ in $N_G$}
                \State $u_1\cdots u_{|P|} \gets P$
                \State $d_1\cdots d_{|P|} \gets$ the degree of each $u_1\cdots u_{|P|}$
                \State isNotPetersenLike $\gets$ \textsf{True}
                \ForAll{degree patterns $d'$ in the corresponding row (degree = $d^*$) of Table \ref{tab:neighbor-degree-pattern}}
                    \If{$d$ is a cyclical subsequence of $d'$ or its mirror}
                        \State isNotPetersenLike $\gets$ \textsf{False}
                    \EndIf
                \EndFor
                \If{isNotPetersenLike is \textsf{True}}
                    \State \Return \textsf{True} \Comment{Neighbor pattern mismatch}
                \EndIf
            \EndFor
        \EndFor
        \State \Return \textsf{False} \Comment{No degree/pattern mismatch found}
    \end{algorithmic}
\end{algorithm}

\section{Small edge-cuts in minimal counterexamples}
\label{sect:below5cuts}\showlabel{sect:below5cuts}

We provide a proof of the following lemma:

\begin{lem}
    Let $G$ be a minimal counterexample and $F$ be a cyclical edge cut.
    Then, one of the following holds.
    \begin{itemize}
        \item $|F| \geq 6$.
        \item $|F| = 5$, and one of the following two statements holds.
        \begin{itemize}
            \item One of the connected components of $G-F$ is a 5-cycle.
            \item $F$ is an \emph{annular cut} (a cut whose dual is two disjoint non-contractible cycles that are homotopic), and the size of the two dual cycles are 1 and 4, respectively.
        \end{itemize}
    \end{itemize}
\end{lem}

This lemma asserts the \emph{internal 6-connectivity} (defined in \cite{AppelHaken89}, \cite{RSST}) of the dual triangulation $G'$, except when $F$ admits a representativity $1$ structure.

When $F$ ($|F| \leq 5$) is a \emph{disk cut} (a cut whose dual is a contractible cycle), there is proof in \cite{proj2024}.
So, the problem remains when $F$ is an annular cut.
When $|F| = 2,3$, it is easy to see that a 2/3-edge reduction suffices.

When $|F| = 4$, we attempt to use the proof for \cite{proj2024}.
We consider the case when $F$ consists of two pairs of edges, each consisting of a dual cycle of size two.
For the proof to work, we need to show that for one of the graphs $G_1, G_2$ (two graphs separated by $F$), the six possible insertions of parallel edges / vertex pair insertion are all toroidal and non-Petersen-like.
By taking the contraposition, it is equivalent to saying that if $G_1, G_2$ are both subgraphs of a Petersen-like graph with either one parallel edge or vertex pair deleted, the merging of $G_1, G_2$ results in a three-edge colorable graph or a Petersen-like graph.
This merging process is actually equivalent to the \emph{dot product} of two snarks defined in \cite{chladny2010factorisation}, and it is known that the dot product of Blanuša snark families can be written like the following:

\begin{align*}
    BH_i \cdot BH_j &= BH_{i+j-2} \\
    BH_i \cdot BV_j &= BV_{i+j-2}
\end{align*}

Here, $BH_i$ is the Blanuša-H$i$ snark and $BV_i$ is the Blanuša-V$i$ snark.
The two equations above also hold when denoting the Blanuša snark as $BH_{0}$ or $BV_{0}$ and the Petersen graph as $BH_{-1}$.
Therefore, every Petersen-like graph (including the Petersen graph and Blanuša snark) merged in this manner also results in a Petersen-like graph.
Therefore, the condition holds, and the rest of the proof can be completed by \cite{proj2024}.

When $|F| = 5$ and the size of the two dual cycles are 2 and 3, respectively, we can see from Figure \ref{fig:annular-5-cycles} that any 5-cycle can be embedded onto $S_A$.

\begin{figure}[htbp]
    \centering
    \includegraphics[width=0.2\linewidth]{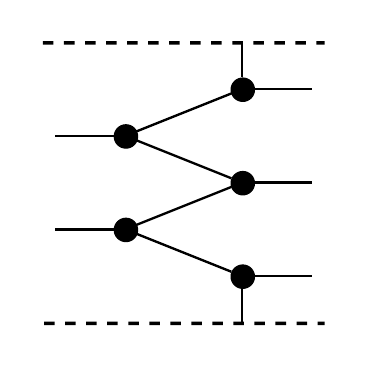}
    \includegraphics[width=0.2\linewidth]{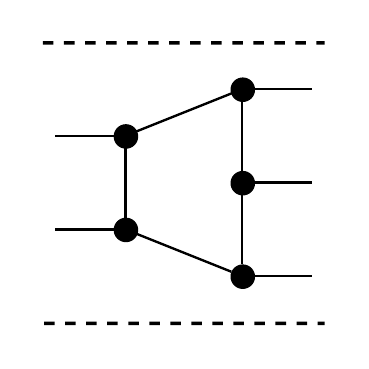}
    \caption{The two 5-cycles between a representativity two and representativity three cut. All possible 5-cycles are isomorphic to one of these 5-cycles.}
    \label{fig:annular-5-cycles}
\end{figure}

Just like in the case of $|F| = 4$, we need to show that the ten possible insertions of 5-cycles are all toroidal and non-Petersen-like, at least for one side of the graph.
These are all the cases, and the lemma holds.

% ===========================================
\section{The remaining unsafe configurations}
\label{subsec:remaining1}\showlabel{subsec:remaining1}
% ===

\subsection{Dealing with C(1)}
\label{subsect:5555}\showlabel{subsect:5555}
Our goal in this section is to prove Lemma \ref{lem:5555}. The symbol $S$ denotes the free completion of C(1) with the ring and, $v_i (0 \leq i \leq 9)$ represents a set of vertices in $S$ as shown in Figure \ref{fig:5555-contractions}.
The two sets of contraction edges for C-reducibility of C(1) are $C_1 = \{v_0v_6, v_6v_7, v_7v_2, v_5v_9, v_9v_8, v_8v_3\}$ or $C_2 = \{v_1v_6, v_6v_9, v_9v_5, v_2v_7, v_7v_8, v_8v_4\}$.

We show that $G'$ does not become the dual of any Petersen-like graphs when C(1) appears. We borrow some terms from Section \ref{sect:6,7-cut}. A \emph{noncontractibly-connected path} of the ring of C(1) is a path $P$ outside $S$ such that $P + S$ contains a noncontractible cycle.
\begin{figure}[htbp]
    \centering
    \includegraphics[width=10cm]{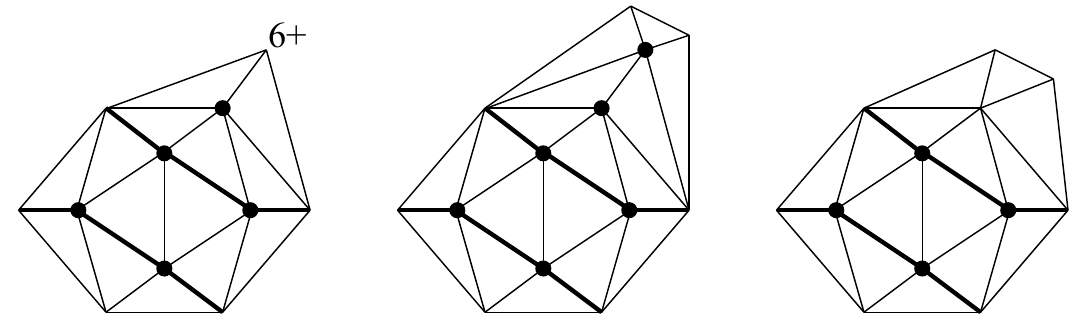}
    \caption{Three reductions occurred by contracting $C_1$ in C(1)}
    \label{fig:5555-reductions}
\end{figure}
\begin{claim}\label{clm:5555-reductions}\showlabel{clm:5555-reductions}
    If a 2,3-cycle-reduction occurs by contracting $C_1$ and $G'$ might become Petersen-like, then it is one of the three cases shown in Figure \ref{fig:5555-reductions}. Moreover, the number of faces deleted by each reduction is 2, 4, or 2 in order of Figure \ref{fig:5555-reductions} respectively.
\end{claim}
\begin{proof}
    We assume the representativity of $G$ is at least 2 since otherwise, $G$ does not become Petersen-like after contractions. 
    
    Suppose a 2,3-cycle-reduction occurs by two noncontractibly-connected paths of the ring of C(1). This reduction occurs by two paths between each pair of vertices in the ring $(v_i, v_l), (v_j, v_k)$ such that $v_i,v_j,v_k,v_l$ are listed in clockwise order of $R$ and $d_{S / C_1}(v_i, v_j) + d_{S / C_1}(v_k, v_l)$ is at most 3 (e.g. $i=0,j=2,k=3,l=4$). The sum of lengths of two paths is at most $3 - d_{S / C_1}(v_i, v_j) - d_{S / C_1}(v_k, v_l)$. Unless the representativity of $G$ is at most 1, These two paths with $S$ constitute a (2, $\leq$ 3)-annulus-cut in its dual, so a small number of vertices remain after contracting $C_1$ by considering Lemma \ref{lem:(2,2)-annulus-cut}. 
    This number is lower than $5$, which is the number of vertices of the dual of the Petersen graph. Hence, $G'$ does not become Petersen-like.
    
    If a 2,3-cycle-reduction occurs by one contractibly connected path of the ring, it contradicts Lemma \ref{minc} except in three cases in Figure \ref{fig:5555-reductions}.
    
    When no 2,3-cycle-reductions occur, the number of faces deleted by contractions is 12, which is the number of faces in C(1). We see more 2, 4, and 2 faces are deleted by each of the three reductions respectively.
\end{proof}
By the symmetry, the same statement as Claim \ref{clm:5555-reductions} holds for $C_2$.

The following fact is trivial.
\begin{fact}\label{fact:size-snark}\showlabel{fact:size-snark}
    The number of vertices of the Petersen graph, or graphs of the Blanuša snark family is represented as $10+8n$ by some integer $n \geq 0$.
\end{fact}

\begin{figure}[htbp]
  \centering
  \includegraphics[width=\linewidth]{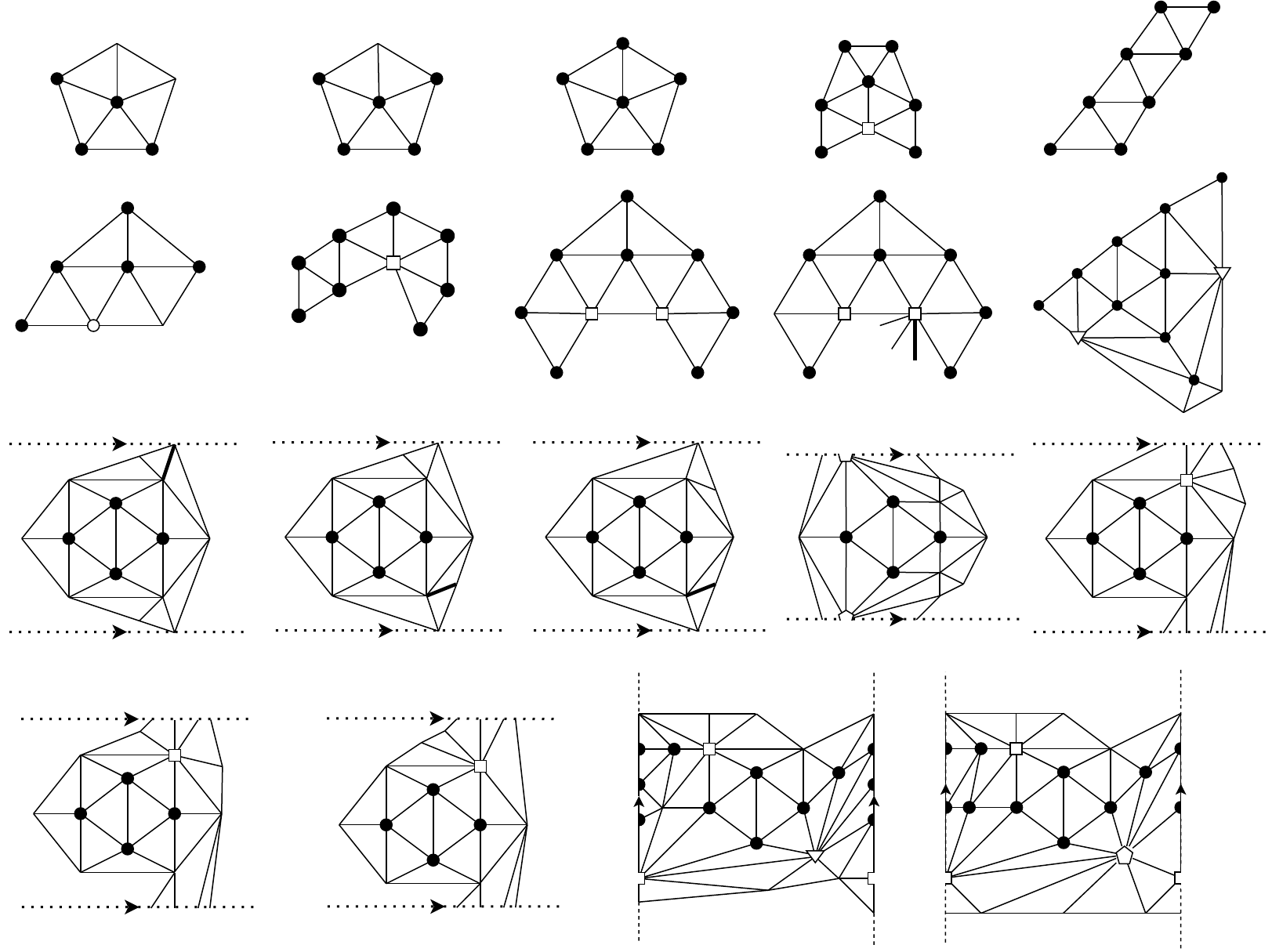}
  \caption{The D-reducible or C-reducible configurations whose contraction size is 1 used in the proof of Lemma \ref{lem:5555}.}
  \label{fig:5555conf_pre}
\end{figure}

\begin{figure}[htbp]
    \centering
    \includegraphics[width=14cm]{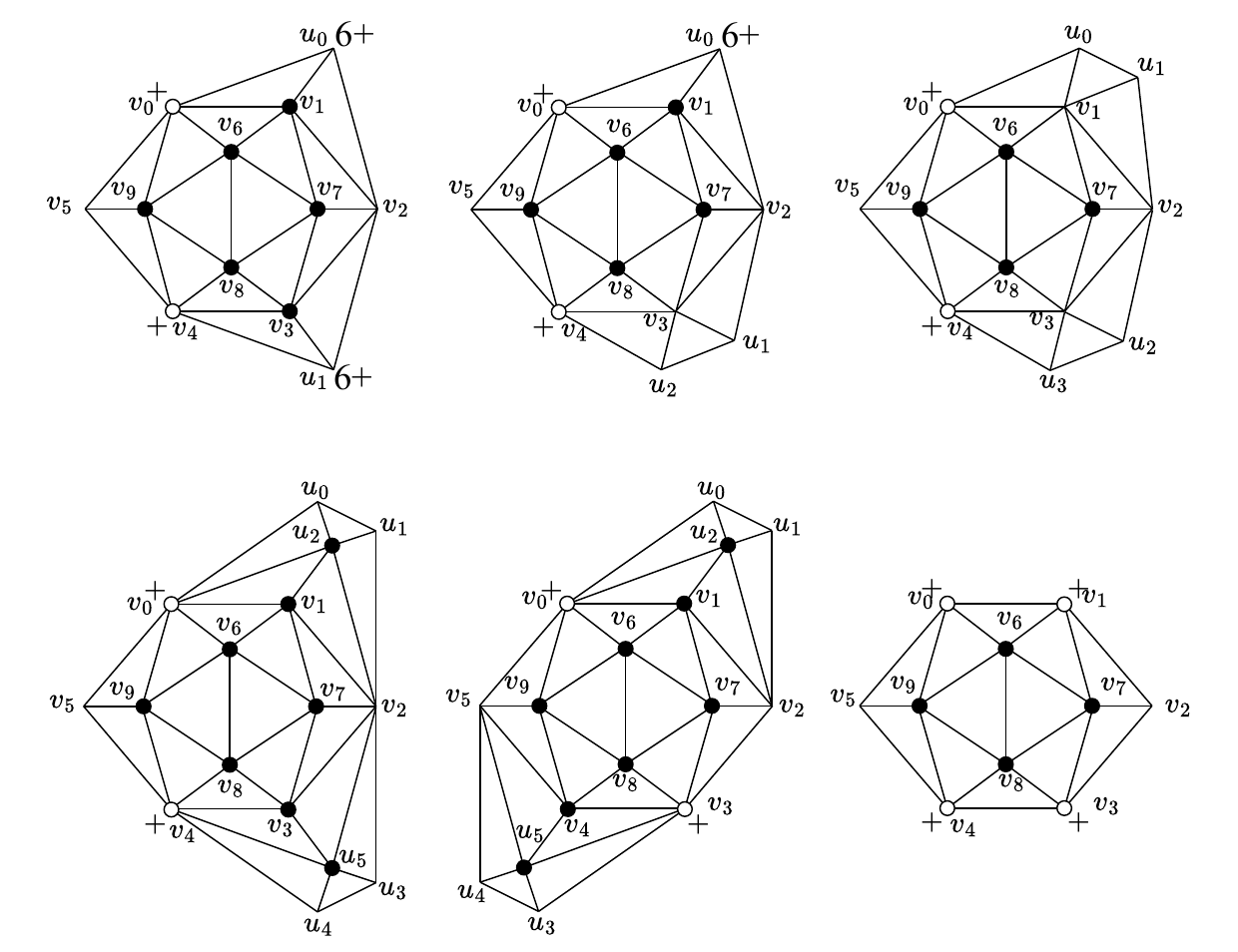}
    \caption{The possible degrees of neighbors of C(1) assuming $G'$ becomes the dual of Petersen-like graphs after contracting each of $C_1, C_2$.}
    \label{fig:5555-cases}
\end{figure}

We enumerate a configuration used in the proof of Lemma \ref{lem:5555} in Figure \ref{fig:5555conf_pre}. The bold lines represent contraction edges used for C-reducibility.
These configurations sometimes contain C(1) as a subgraph. 
The symbols $D(i) (1 \leq i \leq 19$) denote the configurations in this figure. These configurations are D-reducible or C-reducible with contraction size 1, so $G'$ does not become the dual of Petersen-like graphs.
$D(i) (11 \leq i \leq 19)$ are annular configurations. These annular configurations are shown with their rings to describe their embeddings clearly.

\begin{claim}\label{claim:5555-cases}\showlabel{claim:5555-cases}
    If $G'$ results in the dual of one of the Petersen-like graphs by contracting each of $C_1$, $C_2$, the degrees of neighbors of vertices of C(1) are one of the cases described in Figure \ref{fig:5555-cases}.
\end{claim}
\begin{proof}
    D(1), D(2), and D(3) are D-reducible, so not both degrees of $v_0, v_1$ are at most 6. The same claim holds for $v_3, v_4$.
    The number of faces in triangulation is the number of vertices in the dual cubic graphs.
    If $G'$ results in the dual of one of the Petersen-like graphs by reducing each of $C_1, C_2$, the difference between the number of faces after reductions of $C_1$ and $C_2$ is the multiple of $8$ by Fact \ref{fact:size-snark}.
    The statement follows from Claim \ref{clm:5555-reductions}.
\end{proof}

Therefore, we need to address six cases in Figure \ref{fig:5555-cases}. We use the symbol near each vertex to designate it in Figure \ref{fig:5555-cases}. In these proofs, the symbol $H$ denotes a subgraph of $G'$ that consists of $v_i (0 \leq i \leq 9), u_i (0 \leq i)$ each figure in Figure \ref{fig:5555-cases}. 

We use the following fact about the dual of the Petersen graph or graphs in the Blanuša snark family.
We checked the fact by enumerating embeddings as explained in Section \ref{sect:safecont}.
\begin{fact}\label{fact:deg5-in-3cycle}\showlabel{fact:deg5-in-3cycle}
    In the dual of any embedding of the Petersen graph or graphs in the Blanuša Snark family, 
    if all vertices of a facial cycle of length 3 are of degrees at least 6, their degrees are $6,6,8$.
\end{fact}
\begin{fact}\label{fact:no11}\showlabel{fact:no11}
    In the dual of any embedding of the Petersen graph or graphs in the Blanuša snark family, the degree of any vertex is at most 10. 
\end{fact}
\begin{fact}\label{fact:deg5-around10}\showlabel{fact:deg5-around10}
    In the dual of any embedding of the Petersen graph or graphs in the Blanuša snark family, assume a vertex $v$ of degree 10 exists. Let $v_i(0 \leq i < 10)$ of its neighbors be listed in clockwise order of the embedding. Their degrees are $5,5,5,5,x,5,5,5,5,y$, where $x,y \neq 5$, and $v_1=v_7$, $v_2=v_6$.
\end{fact}
\begin{fact}\label{fact:88-in-Blanusa}\showlabel{fact:88-in-Blanusa}
    In the dual of any embedding of the Petersen graph and graphs in the Blanuša Snark family, if an edge whose both endpoints' degrees are at least 8, their degrees are $(8,8)$, $(8,9)$, or $(8,10)$.
\end{fact}
\begin{fact}\label{fact:6,9-around8}\showlabel{fact:6,9-around8}
    In the dual of any embedding of the Petersen graph and graphs in the Blanuša Snark family,
    there is no vertex such that degree is $8$ and degrees of its neighbors listed in clockwise order of the embedding is $6,d_0,d_1,9,d_2,d_3,d_4,d_5$.
\end{fact}

\begin{figure}
    \begin{tabular}{cc}
        \begin{minipage}[t]{0.5\hsize}
            \centering
            \includegraphics[width=\hsize]{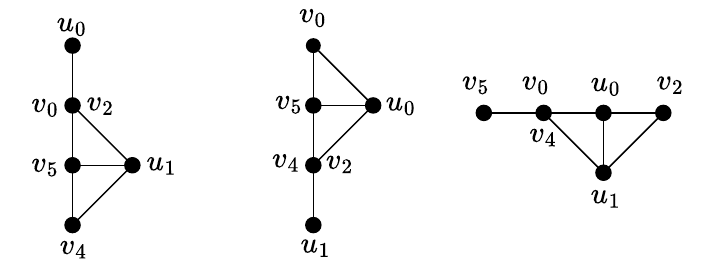}
            \caption{Subgraphs after each reduction in the first case.}
            \label{fig:5555-case1-reductions}
        \end{minipage} &
        \begin{minipage}[t]{0.5\hsize}
            \centering
            \includegraphics[width=\hsize]{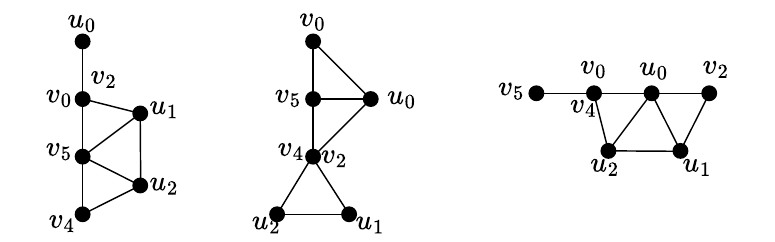}
            \caption{Subgraphs after each reduction in the second case.}
            \label{fig:5555-case2-reductions}
        \end{minipage} \\
        \begin{minipage}[t]{0.5\hsize}
            \centering
            \includegraphics[width=0.7\hsize]{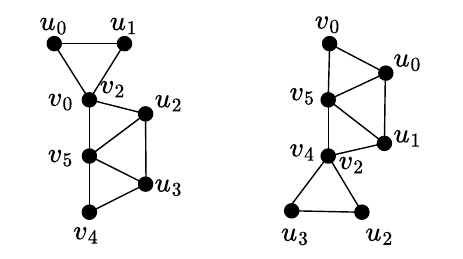}
            \caption{Subgraphs after each reduction in the third case.}
            \label{fig:5555-case3-reductions}
        \end{minipage} & 
        \begin{minipage}[t]{0.5\hsize}
            \centering
            \includegraphics[width=0.7\hsize]{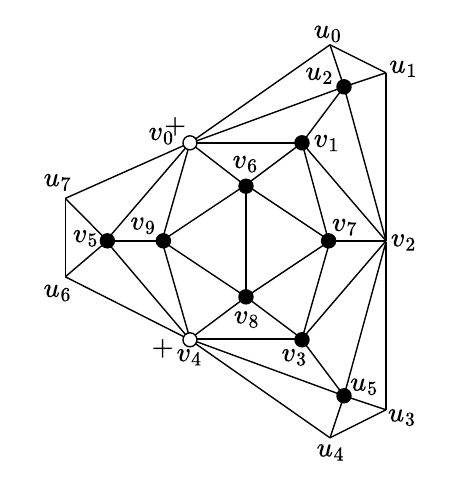}
            \caption{More detailed structure of neighbors of C(1) in the fourth case.}
            \label{fig:5555-case4-detailed}
        \end{minipage} \\
        \begin{minipage}[t]{0.5\hsize}
            \centering
            \includegraphics[width=0.5\hsize]{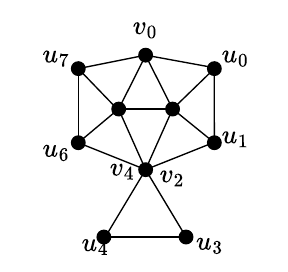}
            \caption{Subgraphs after each reduction in the fourth case.}
            \label{fig:5555-case4-reductions}
        \end{minipage} &
        \begin{minipage}[t]{0.5\hsize}
            \centering
            \includegraphics[width=0.5\hsize]{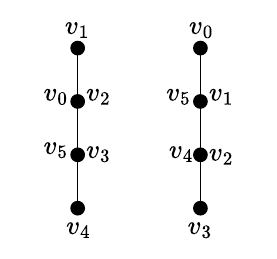}
            \caption{Subgraphs after each reduction in the sixth case.}
            \label{fig:5555-case6-reductions}
        \end{minipage} \\ 
        \multicolumn{2}{c} {
            \begin{minipage}[t]{0.8\hsize}
                \centering
                \includegraphics[width=\hsize]{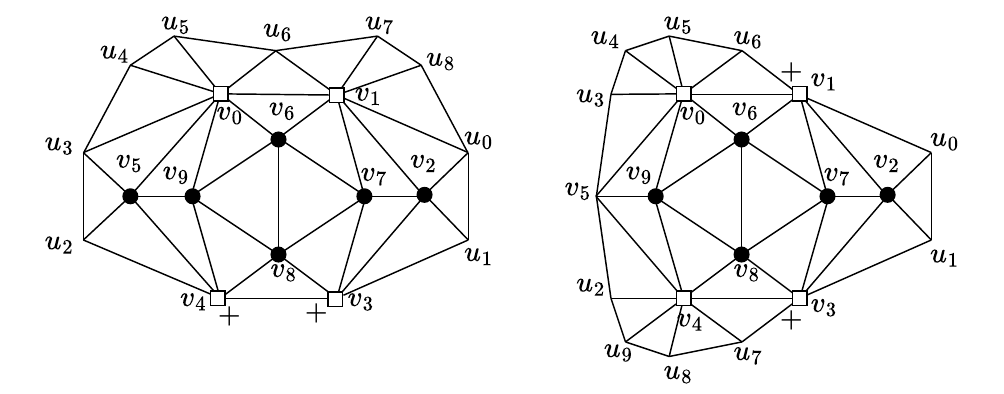}
                \caption{More detailed structure of neighbors of C(1) in the sixth case.}
                \label{fig:5555-case6-detailed}
            \end{minipage}
        } 
    \end{tabular}
\end{figure}

\begin{claim}\label{claim:5555-case1}\showlabel{claim:5555-case1}
    $G'$ does not become the dual of Petersen-like graphs in the first case of Figure \ref{fig:5555-cases}.
\end{claim}
\begin{proof}
    Four vertices $v_1, v_6, v_7, v_8$ and $v_3, v_6, v_7, v_8$ also constitute a configuration C(1), so we can use the contraction edges for these configurations. No two vertices in $v_i (0 \leq i \leq 9), u_i (0 \leq i \leq 2)$ are identical since otherwise, a cycle that contradicts Lemma \ref{minc} exists or the representativity becomes at most 1 by one of the contractions.

    Let $C_3$ be an edge set $\{u_0v_1, v_1v_7, v_7v_3, v_0v_6, v_6v_8, v_8v_4\}$. $C_3$ is also contraction edges used for C-reducibility. By considering the number of faces deleted by reductions of $C_3$ and Fact \ref{fact:size-snark}, the degree of $v_5$ is not 5.
   
    After contracting $C_1, C_2, C_3$, $H$ becomes each of three graphs in Figure \ref{fig:5555-case1-reductions} respectively, and other reductions do not occur. All vertices of the Petersen graph and graphs in the Blanuša snark family have degrees of at least 5. Hence, each of $u_0$, $u_1$ is incident with at least 4 edges outside $H$. The degrees of $v_0, v_4$ are at least 7, so each of $v_0$, $v_4$ is incident with at least 2 edges outside $H$. The cycle that consists of $u_0, u_1, v_0(=v_4)$ contradicts Fact \ref{fact:deg5-in-3cycle} after contracting $C_3$. Therefore the claim holds.
\end{proof}

\begin{claim}\label{claim:5555-case2}\showlabel{claim:5555-case2}
     $G'$ does not become the dual of Petersen-like graphs in the second case of Figure \ref{fig:5555-cases}.
\end{claim}
\begin{proof}
    Four vertices $v_1, v_6, v_7, v_8$ also constitute a configuration C(1), so we can use the contraction edges for these configurations. No two vertices in $v_i (0 \leq i \leq 9), u_i (0 \leq i \leq 3)$ are identical since otherwise, a cycle that contradicts Lemma \ref{minc} exists or the representativity becomes at most 1 by one of the contractions.

    Let $C_3$ be an edge set $\{u_0v_1, v_1v_7, v_7v_3, v_0v_6, v_6v_8, v_8v_4\}$. $C_3$ is also contraction edges used for C-reducibility. By considering the number of faces deleted by reductions of $C_3$ and Fact \ref{fact:size-snark}, the degree of $v_5$ is not 5.
   
    After contracting $C_1, C_2, C_3$, $H$ becomes each of three graphs in Figure \ref{fig:5555-case2-reductions} respectively, and other reductions do not occur. All vertices of the Petersen graph and graphs in the Blanuša snark family have degrees of at least 5. Hence, each of $u_0$(, $u_2$) is incident with at least 4(, 3) edges outside $H$ respectively. The degrees of $v_0, v_4$ are at least 7, so each of $v_0$, $v_4$ is incident with at least 2 edges outside $H$. The cycle that consists of $u_0, u_2, v_0(=v_4)$ contradicts Fact \ref{fact:deg5-in-3cycle} after contracting $C_3$. Therefore the claim holds.
\end{proof}

\begin{claim}\label{claim:5555-case3}\showlabel{claim:5555-case3}
     $G'$ does not become the dual of Petersen-like graphs in the third case of Figure \ref{fig:5555-cases}.
\end{claim}
\begin{proof}
    If $u_i$ ($i=0$, or $1$) and $u_j$ ($j=2$, or $3$) are identical and a path in $H$ between $u_i$ and $u_j$ is a noncontractible cycle, reducible configuration D(11), D(12), D(13) appears.
    It is possible that each pair of vertices $(v_0, u_3), (v_0, v_4), (u_0, v_4)$ are identical and a path in $H$ between them is a noncontractible cycle.
    However, the other two vertices are not identical since otherwise, a cycle that contradicts Lemma \ref{minc} exists or the representativity becomes at most 1 by one of the contractions.

    After contracting $C_1, C_2$, $H$ becomes each of two graphs in Figure \ref{fig:5555-case3-reductions} respectively, and other reductions do not occur. All vertices of the Petersen graph and graphs in the Blanuša snark family have degrees of at least 5. Hence, $u_1$ is incident with at least 3 edges outside $H$. 
    Each degree of $v_5, v_4$ is at least 5, 7, so each of $v_5$, $v_4$ is incident with at least 2 edges outside $H$ respectively. 
   
    The degree of $v_2$ is at least 7, since otherwise C(4), and C(5) are reducible, so $v_2$ is incident with at least 2 edges outside $H$. 
    After contracting $C_2$, the cycle that consists of $u_1, v_5, u_2(=v_4)$ exists. Each degree of $u_1, v_5$ becomes at least 6. The degree of $u_2(=v_4)$ becomes at least 8.
    If $v_4$ is identical to $v_0$ and a path in $H$ between them is a noncontractible cycle, the degree of $v_0(=v_4)$ in $G'$ is at least 11 since D(14) is reducible, so the degree of $u_2(=v_4)$ is at least 9 after contracting $C_2$.
    Otherwise, not both degrees of $v_5, v_4$ are exactly 5, 7 respectively since D(6) is reducible, so the degree of $v_5$ is at least 7 or the degree of $u_2(=v_4)$ is at least 9 after contracting $C_2$. 
    Hence the cycle contradicts Fact \ref{fact:deg5-in-3cycle} after contracting $C_2$. Therefore the claim holds.
\end{proof}

\begin{claim}\label{claim:5555-case4}\showlabel{claim:5555-case4}
    $G'$ does not become the dual of Petersen-like graphs in the fourth case of Figure \ref{fig:5555-cases}.
\end{claim}
\begin{proof}
    Four vertices $v_1, v_6, v_7, v_8$ or $v_3, v_6, v_7, v_8$ also constitute a configuration C(1), so we can use the contraction edges for these configurations.
    
    It is possible that each pair of vertices 
    $(u_0, u_3), (u_0, u_4), (u_0, v_5), (u_1, u_3)$
    , $(u_1, u_4), (u_1, v_5), (u_3, v_5), (u_4, v_5)$ are identical and a path in $H$ between them is a noncontractible cycle.
    If the other two vertices are identical, a cycle that contradicts Lemma \ref{minc} exists or the representativity becomes at most 1 by one of the contractions.

    Let $C_3$ be an edge set $\{u_2v_1, v_1v_7, v_7v_3, v_0v_6, v_6v_8, v_8v_4\}$. $C_3$ is also contraction edges used for C-reducibility. 
    By considering the number of faces deleted by reductions of $C_3$ and Fact \ref{fact:size-snark}, the degrees of $v_5$ must be 5. 
    Here, the symbol $H$ denotes a subgraph of $G'$ that consists of $v_i (0 \leq i \leq 9), u_i (0 \leq i \leq 7)$ in Figure \ref{fig:5555-case4-detailed}. 
    The degree of $v_5$ is 5, so each of $u_3, u_4$ are not identical to $v_5$ since otherwise an annulus-cut that contradicts Lemma \ref{lem:(2,2)-annulus-cut}, \ref{lem:(1,3)-annulus-cut} exists.
    The configuration D(4) is D-reducible, so all degrees of $v_0, v_2, v_4$ are at least 9. Hence each of $v_0$, $v_2$, $v_4$ is incident with at least 2 edges outside $H$.
    After contracting $C_2$, $H$ becomes a graph in Figure \ref{fig:5555-case4-reductions}. The degree of $v_2(=v_4)$ is at least 10. By Fact \ref{fact:no11}, its value is exactly 10. Hence the degrees of $v_2, v_4$ are exactly 9 in $G'$.
    By Fact \ref{fact:deg5-around10}, the degrees of $u_3, u_4$ after contracting $C_2$ are 5, so the degrees of $u_3$, $u_4$ are 6 before contraction (the degrees of these vertices change by 1 due to contraction). The reducible configuration D(10) weakly appears.
\end{proof}

\begin{claim}\label{claim:5555-case5}\showlabel{claim:5555-case5}
    $G'$ does not become the dual of Petersen-like graphs in the fifth case of Figure \ref{fig:5555-cases}.
\end{claim}
\begin{proof}

    Let $C_3$ be an edge set $\{u_2v_1, v_1v_7, v_7v_3, v_0v_6, v_6v_8, v_8v_4\}$. $C_3$ is also contraction edges used for C-reducibility. By considering the number of faces deleted by reductions of $C_3$ and Fact \ref{fact:size-snark}, the degrees of $v_5, u_1, u_2$ must be 5. The configuration D(5) weakly appears.

\end{proof}

\begin{claim}\label{claim:5555-case6}\showlabel{claim:5555-case6}
     $G'$ does not become the dual of Petersen-like graphs in the sixth case of Figure \ref{fig:5555-cases}.
\end{claim}
\begin{proof}
  
    It is possible that each pair of vertices 
    $(v_0, v_4), (v_1, v_3)$ are identical and a path in $H$ between them is a noncontractible cycle.
    If the other two vertices are identical, a cycle that contradicts Lemma \ref{minc} exists or the representativity becomes at most 1 by one of the contractions. 
    
    If ($v_0$, $v_4$) are identical and $v_0v_9v_4(=v_0)$ is a noncontractible cycle of length 2, no other disjoint homotopic noncontractible cycle of length 2 exists by Lemma \ref{lem:(2,2)-annulus-cut}. The same statement holds after contracting $C_1$ since each component that consists of edges of $C_1$ contains $v_0$ or $v_9$. 
    Hence, $G' / C_1$ is the dual of the Petersen graph if it would be the dual of Petersen-like graphs. In this case, the original number of faces of $G'$ is exactly 22, which is not a minimal counterexample by the result of \cite{brinkmann2013generation} (All Snarks of size up to 36 are enumerated).
    We can say the same thing for $(v_1, v_3)$.

    After contracting $C_1, C_2$, $H$ becomes each of two graphs in Figure \ref{fig:5555-case6-reductions} respectively, and other reductions do not occur. All vertices of the Petersen graph and graphs in the Blanuša snark family have degrees of at least 5. Hence, each of $v_0, v_1, v_3, v_4$ is incident with at least 4 edges outside $H$. Each degree of $v_2, v_5$ is at least 5, so each of $v_2$, $v_5$ is incident with at least 2 edges outside $H$. 
    After contracting $C_1$, each of the degrees of $v_0(=v_2), v_3(=v_5)$ is at least 8. By Fact \ref{fact:88-in-Blanusa}, we assume w.l.o.g. each degree of $v_0, v_2$ is 8, 5 respectively. 
    After contracting $C_2$, each of the degrees of $v_1(=v_5), v_2(=v_4)$ is at least 8. By Fact \ref{fact:88-in-Blanusa}, each degree of $v_1, v_5$ is 8, 5 respectively, or the degree of $v_4$ is 8. 
    Hence, the degrees of neighbors of $H$ are either of two graphs in Figure \ref{fig:5555-case6-detailed}. The symbols $H_1, H_2$ denote the left, and right graph in Figure \ref{fig:5555-case6-detailed} respectively.

    If $u_0$ and $v_3$ are identical and a path in $H_i(i=1,2)$ between them is a noncontractible cycle and its degree is 8, a reducible configuration D(15), D(16), D(17) appears. 
    $u_0$ are not identical to $v_0, v_1, v_4$ since otherwise, a cycle that contradicts Lemma \ref{minc} exists or the representativity becomes at most 1 by one of the contractions.
    Hence, by combining the two above claims and the degree of $u_0$ after contracting $C_2$ is 5 by Fact \ref{fact:deg5-in-3cycle}, the degree of $u_0$ in $G'$ is 5.
    We can say the degrees of $u_1,u_2,u_3$ are also 5 in the same way.
    
    For $H_1$, a reducible configuration D(8) weakly appears.
    For $H_2$, the sum of degrees of $v_3, v_5$ in $G'$ are at most 15 by Fact \ref{fact:88-in-Blanusa}. 
    If a pair of degrees of $(v_3, v_5)$ is $(8, 5)$ or $(8, 6)$, a reducible configuration D(8) or D(9) weakly appears, so their possible values are $(8,7), (9,5), (9,6)$, or $(10,5)$. By Fact \ref{fact:deg5-around10}, reducible configuration D(7), D(18), or D(19) appears when their values are $(8,7), (9,6)$, or $(10,5)$ respectively. 
    The remaining case is when their values are $(9,5)$. The degree of $v_1$ is also 9 by considering the symmetry. By Fact \ref{fact:6,9-around8}, after contracting $C_1$, the resulting graph does not become Petersen-like.
\end{proof}

\section{Appendix to Algorithm for Theorem \ref{algo} }\label{secalgoappen}\showlabel{secalgoappen}

Step 3 can be divided into the following steps.

\medskip

{\bf Step 3.1:} When representativity of $G$ in $\Sigma$ is exactly one. 

\medskip

In this case by Lemma \ref{lem:conf-in-rep1}, there is a subgraph $K$ of the dual graph $G'$, that is a reducible configuration in the set $\mathcal{K}$ and that does not hit the vertex that intersects a non-contractible curve of length exactly one in $\Sigma$. 

Algorithmically speaking, this part is just to find a vertex $v$ of the dual graph $G'$ with $T(v) >0$, after cutting $G'$ into the plane.

\medskip

{\bf Step 3.2:} When representativity of $G$ in $\Sigma$ is at least two. 

\medskip

Now we apply our discharging rules $\mathcal{R}$. By Lemma \ref{T(v)>0}, if there is a vertex $v$ of the dual graph $G'$ with $T(v) >0$,  there must exist a subgraph $K'$ of $G'$ which contains either $v$ or at least one of its neighbors and is isomorphic to one of the reducible configurations of $\mathcal{K}$ (but $K'$ may not be an induced subgraph of $G'$). 

Consider the case $T(v)=0$ for all the vertices $v$ of $V(G)$. 
Then, according to Lemma \ref{T(v)>0T}, there is a reducible configuration $K$ in the set $\mathcal{K}$ in $G'$. 

\medskip

By Lemma \ref{T(v)>0T}, such a subgraph $K$ would not appear in $G'$. 
This completes the proof of Theorem \ref{mainth}.

\medskip

Now, let us estimate the time complexity of this algorithm. Step 1 takes $O(n)$ time as in \cite{MOHARlinear}. All of Steps 2 and 3 take an $O(n)$ time, and we recuse. More precisely, for the correctness and the time complexity: 

\begin{itemize}
\item 
     We can find the shortest non-contractible curve in $O(n)$ time by \cite{CABELLO}.
    \item
     By Lemma \ref{T(v)>0T}, all reducible configurations in $\mathcal{K}$ are safely reduced. 
    
    \item 
    The discharging process takes $O(n)$ time because we have at most 201 rules, each of which can apply just once to each edge. This allows us to find the set of all vertices $V'$ that end up with a positive charge (i.e., $T(v) > 0$) in $O(n)$ time. 
    \item 
    For the reducibility process, we only need to look at a vertex $v \in V'$, the first and second neighbors, and the third neighbors of $v$. But we do not have to look at all the vertices of degree at least 13, as none of our reducible configurations in $\mathcal{K}$ contains a vertex of degree 13 or more. Let $Q_v$ be the graph induced by these neighbors, together with $v$, after deleting the vertices of degree at least 13. 

%    It is well-known that such a graph $W$ is of tree-width at most 24 (see Eppstein \cite{Eppstein00}). Moreover, 
Each reducible configuration is of order at most 20, and by our choice, $Q_v$ has at most 10000 vertices. By Lemma \ref{T(v)>0}, we know that $Q_v$ contains a reducible configuration $K_v$. Thus we can find a reducible configuration $K_v$ in $O(1)$ time. 

Consider the case $T(v)=0$ for all the vertices $v$ of $V(G)$. 
Then, by Lemma \ref{T(v)>0T}, there is a reducible configuration $K$ in the set $\mathcal{K}$ in $G'$. This can also be found in $O(n)$ time, as below.

\item 
We may apply Lemma \ref{lem:6,7-cut} if we find a separating cycle of length six or seven in $G'$. This case needs a bit of care.

We can actually enumerate all 5,6,7-cycles in $G'$ in $O(n)$ time by using \cite{FOCS05,Eppstein00}. To this end, we first refer to Theorem 3.1 in \cite{FOCS05}. The important idea is that we run a
bread-first search from some root vertex $r$ in $G'$, and assign the label to each vertex $v \in V(G')$ to be the distance between $r$ and $v$
modulo eight. The union of any seven label sets is the disjoint
union of subgraphs, each consisting of at most seven breadth-first layers. By \cite{Eppstein00} each
of these subgraphs, and therefore the union, has treewidth
at most 100. 

Thus by the standard dynamic programming on bounded treewidth graphs, we can enumerate all cycles of length 5, 6, and 7 in $O(n)$ time for each of the seven label sets. It is clear that every cycle of length at most seven must be contained in some of the seven label sets, and in addition, there are eight such bounded treewidth graphs, thus we can enumerate all cycles of length 5, 6, and 7 in $G'$ in $O(n)$ time. 
We can also prove that the number of such cycles is $O(n)$ by showing that there are at most three such cycles that can cross by Lemma \ref{minc}, see also Figures \ref{fig:6cut} and \ref{fig:7cut} (and the proofs associated with these figures). 

Among these cycles, we can take a minimal one, i.e. a cycle that satisfies Lemma \ref{lem:6,7-cut}, and subject to that, the number of vertices strictly inside the disk bounded by the cycle is as small as possible. By Lemma \ref{lem:6,7-cut}, we know that at least one of the configurations appears strictly inside the disk, which allows us to make a reduction.  

This dynamic programming also allows us to find a reducible configuration when $T(v)=0$ for all the vertices $v$ of $V(G)$. 
Then, by Lemma \ref{T(v)>0T} there is a reducible configuration $K$ in the set $\mathcal{K}$ in $G'$. This can be also found in $O(n)$ time
\end{itemize}

Thus, the total time complexity is $O(n^2)$ as claimed.

\section{Strengthening Tutte's 4-Flow Conjecture for toroidal graphs}\label{nonzero4}

An important consequence of this paper is that the Tutte 4-Flow Conjecture can be strengthened for toroidal graphs. This is a direct consequence of our main result, Theorem \ref{mainth}, where we consider cubic graphs. In this section, we prove that a minimum counterexample would need to be cubic, 
%thus reducing the full conjecture to cubic graphs.
thus obtaining the following consequence of Theorem \ref{mainth}.

\begin{thm}
    Let $G$ be a 2-connected cubic graph that can be embedded in the torus. Then $G$ admits a nowhere-zero 4-flow if and only if it is not Petersen-like.
\end{thm}

Let us extend the notion of a \emph{minimum counterexample} $G$ to be a 2-edge-connected toroidal graph $G$ with minimum $|V(G)|+|E(G)|$ that is not Petersen-like and does not admit a nowhere-zero 4-flow. 
The proof is almost identical to that given in \cite{proj2024}, and the first half of the proof follows the standard strategy used, for example, in \cite{thomasnon}. We will make use of the following observation.

\begin{lem}
    Petersen-like graphs on the torus are not apex graphs, i.e. every vertex-deleted subgraph is nonplanar.
\end{lem}

\begin{proof}\label{lem:Petersen-like is not apex}
    It is well-known that the Petersen graph is not an apex graph. It can be proved by induction that the repeated dot product of Petersen graphs contains a subdivision of the Petersen graph, so it cannot be apex either. The same is true if we replace vertices in such dot product with 2-connected graphs. We leave the details to the reader.
\end{proof}

\begin{lem}
    A minimum counterexample is a 3-connected cubic graph.
\end{lem}

\begin{proof}
    Suppose that $G$ is a minimum counterexample that is not cubic. If $G$ is not 2-connected, then $G=G_1\cup G_2$ where $G_1\cap G_2$ is a cutvertex. Since $G$ has no cutedges, both $G_1$ and $G_2$ are 2-edge-connected. If they both admit 4-flows, so does $G$. So, one of them must be Petersen-like (as otherwise we would contradict that $G$ is a minimum counterexample). However, in that case, the other one must be planar, and it is easy to see that $G$ itself would be Petersen-like.

    Similarly, if $G$ has a 2-vertex-separator $\{x,y\}$. In that case, $G$ would be a 2-sum of two graphs $G_1$ and $G_2$, both containing the edge $e=xy$. Let $\phi_i$ be a nowhere-zero 4-flow in $G_i$ for $i=1,2$. If $e\notin E(G)$, then we change $\phi_2$ so that $\phi_2(xy)=-\phi_1(xy)$. Then it is clear that $\phi_1+\phi_2$ gives rise to a nowhere-zero flow in $G$. On the other hand, if $e\in E(G)$, then we change $\phi_2$ so that $\phi_2(xy)\ne -\phi_1(xy)$. Then, again, $\phi_1+\phi_2$ gives rise to a nowhere-zero-flow in $G$. 
    
    This leaves us with the possibility that $G_1$ (say) is Petersen-like. In that case, $G_2$ must be planar. Since $G$ itself is not Petersen-like, the edge $xy\in E(G_1)$ is one of the edges confirming that $G_1$ is Petersen-like. In that case, we let $\phi_1$ be a nowhere-zero 4-flow of $G_1-xy$ (which is 2-edge-connected since otherwise it can be shown that $G$ would be Petersen-like). Now, taking a nowhere-zero 4-flow $\phi_2$ of $G_2$ (or of $G_2-xy$ if $xy\notin E(G)$), the sum $\phi_1+\phi_2$ is a nowhere-zero 4-flow of $G$, a contradiction.
    
    This confirms that $G$ is 3-connected. Next, we claim that $G$ is cyclically 4-edge-connected. For a contradiction, suppose that $F=\{x_1x_2,y_1y_2,z_1z_2\}$ is a cyclic 3-edge-cut. Let $C_1,C_2$ be the two components of $G-F$, where $x_i,y_i,z_i$ are in $C_i$ ($i=1,2$). By adding a vertex $u_i$ joined to $x_i,y_i,z_i$ in $C_i$, we obtain a 3-connected graph $H_i$ for $i=1,2$. If $H_1$ has a nowhere-zero 4-flow $\phi_1$ and $H_2$ has a nowhere-zero 4-flow $\phi_2$ (with values in the group $\mathbb{Z}_2\times \mathbb{Z}$), note first that the values on the cut $\{x_iu_i,y_iu_i,z_iu_i\}$ in $H_i$ are all different. Therefore, we can permute the flow values in $\phi_2$ such that $\phi_1(x_1u_1)=\phi_2(x_2u_2)$, $\phi_1(y_1u_1)=\phi_2(y_2u_2)$, and $\phi_1(z_1u_1)=\phi_2(z_2u_2)$. Then, the two flows can be combined to obtain a flow in $G$. This proves that at least one of the two graphs say $H_1$, has no 4-flows. Using induction, $H_1$ is Petersen-like. Now, we claim that $H_2$ must be planar. Since $H_1$ is not apex by the lemma, the embedding of $C_1=H_1-u_1$ is a 2-cell embedding with all three edges going to $u_1$ emanating into the same face. That face contains $C_2$; thus, $C_2$ plus the vertex $u_2$ is planar. This means that $G$ itself is also Petersen-like. But it has a nontrivial 3-edge-cut, so it is not minimal. 
    This contradiction proves that there are no cyclic 3-edge-cuts in $G$.

    The next claim is that $G$ has no edge $e$ for which $G-e$ would be planar. Assuming that is not the case, we can embed $G$ in the torus such that there is a 1-noose\footnote{A \emph{$1$-noose} is a non contractible closed curve on the torus that intersects the graph in one point only (either on an edge or a vertex).} that intersects $G$ only at the middle of $e$. In the proof, whose details are below, we will perform the splitting of vertices. By performing such a splitting, we get a contradiction unless the graph obtained after the splitting would be Petersen-like. However, this outcome is not possible if we have the 1-noose for the embedding of $G$ since Petersen-like graphs are non-apex. Finally, we claim that $G$ has no 1-nooses at all. If there is one, it cuts through a vertex of $G$, splitting it into two vertices, both of degree more than 1. This gives a 2-edge-connected planar graph, whose 4-flow is also a 4-flow in $G$. This shows that $G$ has representativity at least 2.
    
    Finally, if $G$ has a vertex $v$ of degree greater than 3, consider splitting off a pair of consecutive edges $vx,vy$ in a clockwise rotation around $v$. The resulting graph $H$ is 2-connected, since $G$ is 3-connected. A flow in $H$ gives a flow in $G$. As $H$ is smaller (it has fewer edges than $G$), it must be Petersen-like. Let $v'$ be the vertex of degree 2 with neighbors $x,y$ obtained from $v$ after splitting. Since the edges $vx,vy$ are consecutive around $v$, the vertices $v$ and $v'$ are on the same face in the toroidal embedding of $H$, and since $G$ is not Petersen-like, they belong to different graphs when building $H$ from one or more copies of the Petersen graph. %See Figure \ref{fig:G1 is Petersen-like}. 
    Suppose that the consecutive neighbors around $v$ are $x',x,y,y',\dots$. Now, it is easy to see that we could have split off the edges $vx',vx$ or the edges $vy,vy'$, and then $G_1$ would not be Petersen-like, thus yielding a contradiction.

    The above proof shows that the minimum counterexample is cubic. This completes the proof.    
\end{proof}

\end{document}